\documentclass[11pt]{article}
\usepackage[a4paper, margin=1in]{geometry}
\usepackage{amsmath,amsthm,amssymb,bbm,graphicx,caption,subcaption,natbib,xcolor,enumerate,enumitem,tikz}
\usepackage[hypertexnames=false]{hyperref}
\usepackage[normalem]{ulem}
\usetikzlibrary{arrows.meta,patterns,intersections,shapes.misc,calc}
\usepackage[mathscr]{euscript}
\hypersetup{colorlinks,linkcolor={red!70!black},citecolor={blue!70!black},urlcolor={blue!70!black}}
\setlist[itemize]{leftmargin=0.4cm}

\newcommand{\R}{\mathbb{R}}

\newcommand{\N}{\mathbb{N}}
\newcommand{\C}{\mathbb{C}}
\newcommand{\E}{\mathbb{E}}
\renewcommand{\Pr}{\mathbb{P}}
\newcommand{\Ind}{\mathbbm{1}}

\newcommand{\norm}[1]{\|{#1}\|}

\newtheorem{theorem}{Theorem}
\newtheorem{proposition}[theorem]{Proposition}
\newtheorem{lemma}[theorem]{Lemma}
\newtheorem{corollary}[theorem]{Corollary}

\theoremstyle{definition}

\newtheorem*{remark*}{Remark}
\newtheorem{example}{Example}

\newcommand{\cvd}{\overset{d}{\to}}
\newcommand{\eqd}{\overset{d}{=}}

\DeclareMathOperator\Cov{Cov}
\DeclareMathOperator\tr{tr}

\DeclareMathOperator*{\argmin}{argmin}

\DeclareMathOperator\Unif{Unif}
\DeclareMathOperator\dK{d_\mathrm{K}}
\DeclareMathOperator\W{\mathrm{W}}

\title{Is the regression $F$-test doubly robust?}
\author{Lucy Xia$^\ast$, Oliver Y. Feng$^\dagger$, Yang Feng$^\ddagger$, Min Xu$^\sharp$ and Richard J. Samworth$^\flat$ \\ \\
$^\ast$Department of Information Systems, Business Statistics and Operations Management, \\ Hong Kong University of Science and Technology \\
$^\dagger$Department of Statistics, London School of Economics and Political Science \\
$^\ddagger$School of Global Public Health, New York University \\
$^\sharp$Department of Statistics, Rutgers University\\
$^\flat$Statistical Laboratory, University of Cambridge
} 

\date{\today}

\begin{document}

\maketitle

\begin{abstract}
We study the robustness of the $F$-test in random design linear models, and reach a somewhat nuanced conclusion.  On the positive side, one of our main results is that the size of the test is close to its nominal level as soon as either the distribution of the normalised error vector is close to uniform on the unit sphere, or the design matrix, after applying a suitable column space-preserving orthogonalisation scheme, is close to being uniformly distributed.  This provides a sense in which the $F$-test is doubly robust.  Our conclusion is reached by establishing a Kolmogorov to Wasserstein distance H\"older continuity property controlling the departure of the $F$-statistic from its notional $F$-distribution under the null.  Writing~$n$,~$p$ and $p_0$ for the sample size and the dimensions of the full and null models respectively, we prove that the H\"older exponent is $1/3$ when $\min(p-p_0,n-p) = 1$ and $1/2$ when $\min(p-p_0,n-p) \geq 2$.  On the other hand, these exponents are relatively small and cannot be improved in general, revealing that the size of the test may depart from its nominal level quite quickly as we move away from settings where the test is exact.  In some cases, our conclusions may be improved by working with a local Kolmogorov distance that focuses on discrepancies between distribution functions in the right tail.
\end{abstract}

\section{Introduction}

Consider the linear model
\[
Y = X\beta^0 + \varepsilon,
\]
where $X$ is a random $n \times p$ design matrix with $n > p$, $\beta^0 \in \mathbb{R}^p$ is a vector of regression coefficients, and $\varepsilon = (\varepsilon_1,\dotsc,\varepsilon_n)$ is a stochastic error vector with $\mathbb{E}(\varepsilon\,|\,X) = 0$. When $\beta^0$ is partitioned as $\beta^0 = (\beta_0^0,\beta_1^0)$ with $\beta_0^0 \in \R^{p_0}$ and $\beta_1^0 \in \R^{p - p_0}$, a test of $H_0: \beta_1^0 = 0$ against $H_1: \beta_1^0 \neq 0$ is known as a \emph{variable significance test}.  We may partition $X$ correspondingly as $X = (X_0 \; \tilde{X})$, where $X_0$ and $\tilde{X}$ have $p_0$ and $p-p_0$ columns respectively, and when $X$ has full column rank, the least squares estimators of $\beta_0^0$ and $\beta^0$ under $H_0$ and $H_1$ respectively are given by
\[
\hat\beta_0 := (X_0^\top X_0)^{-1}X_0^\top Y \quad\text{and}\quad \hat{\beta} := (X^\top X)^{-1}X^\top Y.
\]
The \textit{$F$-statistic} is defined in terms of the residual sums of squares $\mathrm{RSS}_0 := \norm{Y - X_0\hat{\beta}_0}_2^2$ and $\mathrm{RSS} := \norm{Y - X\hat\beta}_2^2$ by
\[
\mathsf{F} := \frac{(\mathrm{RSS}_0 - \mathrm{RSS})/(p - p_0)}{\mathrm{RSS}/(n - p)}.
\]
It is a standard textbook result~\citep[e.g.][Proposition~2.9]{samworth24modern} that if $X$ and~$\varepsilon$ are independent with $\varepsilon \sim N_n(0,\sigma^2 I_n)$ for some $\sigma^2 > 0$, then under the null hypothesis,~$\mathsf{F}$ has an $F$-distribution with $p - p_0$ and $n - p$ degrees of freedom, denoted $F_{p-p_0,n-p}$.\footnote{Recall that if $W_1 \sim \chi_{d_1}^2$ and $W_2 \sim \chi_{d_2}^2$ are independent, then $(W_1/d_1)/(W_2/d_2) \sim F_{d_1,d_2}$.} 

As an immediate consequence, if $\alpha \in (0,1)$ and $q_\alpha^*$ denotes the $(1 - \alpha)$-level quantile of the $F_{p-p_0,n-p}$ distribution, then under the same assumptions on $X$ and $\varepsilon$, the \emph{$F$-test}, which rejects~$H_0$ if $\mathsf{F} > q_\alpha^*$, has exact size $\alpha$. An important special case is that of $p - p_0 = 1$, in which case $\mathsf{F} = \mathsf{T}^2$, where the usual regression $t$-statistic $\mathsf{T}$ has a $t_{n-p}$ distribution under~$H_0$.

This paper concerns the robustness of the $F$-test to departures from the assumptions that~$X$ and $\varepsilon$ are independent with $\varepsilon \sim N_n(0,\sigma^2 I_n)$.  As we outline in our literature review in Section~\ref{sec:relatedWork}, this is a thorny question with a long and detailed history.  The following quotations give a flavour of the contrasting viewpoints expressed:

\begin{quote}
$\ldots$ [in the $K$-way ANOVA setting with equal group sizes] the sampling distribution of\footnote{Pearson writes $\eta^2 := \frac{\mathrm{RSS}_0 - \mathrm{RSS}}{\mathrm{RSS}_0} = \frac{p-p_0}{n-p} \cdot \mathsf{F}$.} $\eta^2$ is remarkably insensitive to changes in population form.

\hfill \citep{pearson1931analysis}
\end{quote}

\begin{quote}
$\ldots$ probabilities derived from well-known analyses of variance and other `small sample' tables, which postulate universal normality, may differ seriously from the true probabilities when the universes are non-normal, even, in some cases, when the degree of non-normality is not considerable.

\hfill \citep{geary1947testing}
\end{quote}

\begin{quote}
Our results may be summarized in the simple statement that sensitivity to non-normality  in the $y$'s is determined by the extent of the `non-normality of the $x$'s.

\hfill \citep{box1962robustness}
\end{quote}

\begin{quote}
In sum, there have been tremendous efforts over the past century put into the robustness of the regression t- and F-test and it was agreed that the t-test is insensitive to non-normality, high dimensions and irregularity of design matrices to certain extent while the F-test is less robust in general.\footnote{Based on our reading of this literature, we would qualify this quotation given in the remarkable bibliographic survey of \citet{lei2021assumption} by replacing `t-test is insensitive$\ldots$' with `one-sample t-test is insensitive$\ldots$'.}

\hfill \citep{lei2021assumption}
\end{quote}

Our first result (Propositions~\ref{prop:spherical-sym-A} and~\ref{prop:spherical-sym}) extends the range of joint distributions of $(X,\varepsilon)$ for which the $F$-test is known to be exact.  It turns out that the critical quantity is a trivariate random vector $T$, representing a self-normalised version of the decomposition of $\|\varepsilon\|^2$ into its projections onto three mutually orthogonal subspaces determined by $X$.  In fact, whenever $T$ has an appropriate Dirichlet distribution, the $F$-test is exact.  In particular, this holds whenever (suitably projected versions of) $X$ and $\varepsilon$ are independent and either one of them has a rotationally invariant distribution. For instance, this is the case if $X$ and $\varepsilon$ are independent, with $\varepsilon$ having an arbitrary distribution, and where $X$ may have a combination of fixed design covariates in the null model (e.g.~an intercept term) and rows of the remaining submatrix being independent and identically distributed mean-zero Gaussian random vectors.  

Our main contributions, however, concern the way in which the size of the $F$-test departs from its nominal level as $T$ departs from the corresponding analogue $T^*$ having the relevant Dirichlet distribution.  Theorem~\ref{thm:dK-W1-A} reveals that the Kolmogorov distance between $\mathsf{F}$ and its nominal $F$-distribution is $\gamma$-H\"older continuous in the $L^1$-Wasserstein distance between $nT$ and $nT^*$, where $\gamma=1/3$ when $\min(p-p_0,n-p) = 1$ and $\gamma=1/2$ otherwise.  Moreover, we show through examples that these H\"older exponents cannot be improved in general.  Thus, we do find a certain level of robustness of the $F$-test, but especially when $\min(p-p_0,n-p) = 1$, which includes the $t$-test where $p-p_0 = 1$, relatively small Wasserstein departures of $nT$ from $nT^*$ can lead to significant changes in the size of the $F$-test.

The distribution of $T$ depends on the joint distribution of $X$ and $\varepsilon$.  Our next main result (Theorem~\ref{thm:doubly-robust-A}) therefore shows that the $L^1$-Wasserstein distance between $nT$ and~$nT^*$ can be controlled in terms of the minimum of Wasserstein distances representing the extent to which the marginal distributions of $X$ and $\varepsilon$ deviate from rotational invariance in a suitable sense.  This provides a sense in which the $F$-test is doubly robust\footnote{In the semiparametric and causal inference literatures \citep[e.g.][]{robins1995semiparametric,scharfstein1999adjusting,hirano2003efficient,kang2007demystifying,rotnitzky2012improved,imbens2015causal,smucler2019unifying}, double robustness refers to the property that average effects can be estimated consistently if at least one of two models defining the data-generating mechanism are well-specified (usually the propensity score and outcome regression models). Our notion of double robustness is related but not identical, as we work within a standard linear model but establish that the $F$-test is valid whenever at least one of the covariate and error distributions has a rotationally invariant distribution.}: its size is close to its nominal level as soon as either the distribution of the normalised error vector is close to uniform on the unit sphere, or the design matrix, after applying a suitable column space-preserving orthogonalisation scheme, is close to being uniformly distributed.

In practice, the null hypothesis is rejected when the $F$-statistic exceeds a given critical value in the right tail of the relevant $F$-distribution.  This motivates the definition of a more liberal `local' Kolmogorov distance that measures only the discrepancy between distribution functions to the right of a given threshold, specifically the appropriate upper $\alpha$-quantile. Theorem~\ref{thm:dK-alpha} establishes an upper bound depending on $\alpha$ with improved H\"older exponents in some regimes, which are also tight for some joint distributions of independent and row-wise exchangeable $(X,\varepsilon)$.

For simplicity of exposition, we present our results in the main text in a way that ignores the potential to orthogonalise covariates with respect to others that are present in the null model.  However, taking advantage of this opportunity leads to a stronger theory, so in the Appendix, we state and prove these more general results, from which the corresponding statements in the main text follow immediately. 

\subsection{Notation}

For $m,n \in \N$, let $[m:n] := [m,n] \cap \N$, $(m:n] := (m,n] \cap \N$ and $[n] := [1:n]$. For $d \in \N$ and $q \in [1,\infty)$, the \emph{$\ell_q$ norm} on $\R^d$ is given by $\norm{v}_q := \bigl(\sum_{j=1}^d |v_j|^q\bigr)^{1/q}$, and the \emph{unit Euclidean sphere} is denoted by $\mathcal{S}^{d-1} := \{v \in \R^d : \norm{v}_2 = 1\}$. We write $e_1,\dotsc,e_d \in \R^d$ for the standard basis vectors and $I_d$ for the $d \times d$ identity matrix. For $n,d \in \N$, denote by $A_{\cdot j} = (A_{1j},\dotsc,A_{nj})$ the $j$th column of a matrix $A = (A_{ij}) \in \R^{n \times d}$, and define $A_S := (A_{\cdot j})_{j \in S} \in \R^{n \times |S|}$ for $S \subseteq [d]$. The \emph{image} (i.e.~column space) of $A$ is the subspace $\mathrm{Im}(A) := \{Av : v \in \R^d\}$, and its \emph{Frobenius norm} is $\norm{A}_\mathrm{F} := (\sum_{i=1}^n\sum_{j=1}^d A_{ij}^2)^{1/2}$. For distributions $\mathscr{P},\mathscr{P}'$ on $\R^{n \times d}$ with finite $q$th moments, the $L^q$-\textit{Wasserstein distance}
\[
\mathrm{W}_q(\mathscr{P},\mathscr{P}') := \inf_{(Z,Z') \sim (\mathscr{P},\mathscr{P}')}\E(\norm{Z - Z'}_\mathrm{F}^q)^{1/q}
\]
is defined as an infimum over all random pairs $(Z,Z')$ with $Z \sim \mathscr{P}$ and $Z' \sim \mathscr{P}'$. For convenience, we slightly abuse notation to write $\mathrm{W}_q(Z,Z') = \mathrm{W}_q(\mathscr{P},\mathscr{P}')$ when $Z \sim \mathscr{P}$ and $Z' \sim \mathscr{P}'$. When $n \geq d$, define $\mathcal{O}_{n \times d}$ to be the set of $n \times d$ matrices with orthonormal columns. A random $n \times d$ matrix $A$ is said to have a \textit{(left) rotationally invariant} distribution if $Q_0 A \eqd A$ for every deterministic $Q_0 \in \mathcal{O}_{n \times n}$. There is a unique distribution on $\mathcal{O}_{n \times d}$, denoted by $\Unif(\mathcal{O}_{n \times d})$, such that $Q_0 A \eqd A$ for every fixed $Q_0 \in \mathcal{O}_{n \times n}$ when $A \sim \Unif(\mathcal{O}_{n \times d})$ \citep[Theorem~6.3 and Proposition~7.2]{eaton2007multivariate}. Finally, a function $f \colon \R \to \R$ is said to be $(\gamma,L)$-H\"older for $\gamma \in (0,1]$ and $L > 0$ if $|f(x) - f(y)| \leq L|x - y|^\gamma$ for all $x,y \in \R$.

\section{Main results}
\label{sec:main-results}

Suppose throughout that $p_0,p,n \in \mathbb{N}_0$ satisfy $0 \leq p_0 < p < n$.  Our results will hold in a more general framework than that discussed in the introduction; in fact, we only require the parameter $\beta^0 \in \mathbb{R}^p$ to be identifiable.  More precisely, let $\mathcal{P}$ denote a set of joint distributions of random pairs $(X,Y)$, where $X$ is an $n \times p$ design matrix and $Y$ is an $n$-dimensional response vector.  Given a function $\beta:\mathcal{P} \rightarrow \mathbb{R}^p$ and $(X,Y) \sim \mathscr{P} \in \mathcal{P}$, we write $\beta^0 := \beta(\mathscr{P})$ and $\varepsilon := Y - X\beta^0$, so that
\[
Y = X\beta^0 + \varepsilon.
\]
To reconcile this general formulation with more familiar settings, suppose that $\mathcal{P}$ consists of the set of distributions of $(X,Y)$ with $\mathbb{E}(X^\top X)$ positive definite and $\mathbb{E}(\|Y\|_2^2) < \infty$.  Then for $(X,Y) \sim \mathscr{P} \in \mathcal{P}$, we may define
\[
\beta(\mathscr{P}) := \argmin_{b \in \mathbb{R}^p}\,\mathbb{E}\bigl(\|Y - Xb\|_2^2) = \bigl\{\mathbb{E}(X^\top X)\bigr\}^{-1}\mathbb{E}(X^\top Y),
\]
which ensures that $\mathbb{E}(X^\top \varepsilon) = 0$.  Often, $\mathcal{P}$ is further restricted by the stronger assumption that $\mathbb{E}(Y \, | \, X) = X\beta^0$, so that $\mathbb{E}(\varepsilon \, | \, X) = 0$.  As a second example, and one that avoids moment conditions, we can take $\mathcal{P}$ to be the set of joint distributions of $(X,Y)$ for which there exists $\beta^0 \in \R^p$ such that $X$ and $\varepsilon = (\varepsilon_1,\dotsc,\varepsilon_n) = Y - X\beta^0$ are independent, $X$ has full column rank with positive probability and $\varepsilon_i \eqd -\varepsilon_i$ for all $i \in [n]$. Finally, if in this example we require $Xv$ to be non-deterministic for all fixed $v \in \R^p \setminus \{0\}$ instead of the errors being symmetric, then identifiability of $\beta^0$ is preserved; see Lemma~\ref{lem:linear-model-identifiable}. Henceforth, we will assume without further comment that $\mathcal{P}$ and the function $\beta$ are such that the error vector~$\varepsilon$ satisfies $\Pr(\varepsilon = 0) = 0$ for all $\mathscr{P} \in \mathcal{P}$, and moreover that $X$ has full column rank almost surely whenever $(X,Y) \sim \mathscr{P} \in \mathcal{P}$.  

Now suppose that there exists $\mathcal{P}_0  \subseteq \mathcal{P}$, called the \emph{null hypothesis parameter space}, such that $\beta_1^0 := \beta(\mathscr{P})_{(p_0:p]} = 0$ for all $\mathscr{P} \in \mathcal{P}_0$.  Our goal is to study the $F$-test of $H_0:\mathscr{P} \in \mathcal{P}_0$ against $H_1: \mathscr{P} \in \mathcal{P} \setminus \mathcal{P}_0$.  Recall the standard decomposition 
\[
\mathrm{RSS}_0 = \norm{(I_n - P_0)Y}_2^2 = \norm{(I_n - P)Y}_2^2 + \norm{(P - P_0)Y}_2^2 = \mathrm{RSS} + \norm{(P - P_0)Y}_2^2,
\]
where the $n \times n$ matrices $P := X(X^\top X)^{-1}X^\top$ and $P_0 := X_0(X_0^\top X_0)^{-1}X_0^\top$ represent the orthogonal projections onto the column spaces of $X$ and $X_0 \equiv X_{[p_0]}$ respectively. We adopt the convention that $\mathrm{Im}(X_0) = \{0\}$ and $P_0 = 0 \in \mathbb{R}^{n \times n}$ when $p_0 = 0$. Since $(P - P_0)X_0 = (I_n - P)X_0 = 0$, it follows that under $H_0$ where $Y = X_0\beta_0^0 + \varepsilon$, we can write
\begin{equation}
\label{eq:F-stat}
\mathsf{F} = \frac{\frac{1}{p - p_0}\norm{(P - P_0)Y}_2^2}{\frac{1}{n - p}\norm{(I_n - P)Y}_2^2} = \frac{\frac{1}{p - p_0}\norm{(P - P_0)\varepsilon}_2^2}{\frac{1}{n - p}\norm{(I_n - P)\varepsilon}_2^2}= \frac{T_2/(p - p_0)}{T_3/(n - p)}
\end{equation}
in terms of
\[
T = (T_1,T_2,T_3) := \frac{\bigl(\norm{P_0\varepsilon}_2^2,\; \norm{(P - P_0)\varepsilon}_2^2,\;\norm{(I_n - P)\varepsilon}_2^2\bigr)}{\norm{\varepsilon}_2^2}.
\]
The convenience of working with $T$ arises partly from the fact that $T$ depends on $\varepsilon$ only through the self-normalised error vector $\varepsilon/\norm{\varepsilon}_2$, so moment assumptions on $\varepsilon$ will not be required in our subsequent results. Moreover, $T$ has non-negative components satisfying $T_1 + T_2 + T_3 = 1$; in other words, $T$ takes values in the two-dimensional unit simplex in $\mathbb{R}^3$.

\subsection{\texorpdfstring{Sufficient conditions for exactness of the $F$-test}{Sufficient conditions for exactness of the F-test}}
\label{sec:exact-F}

\begin{proposition}
\label{prop:spherical-sym-A}
\begin{enumerate}[label=(\alph*)]
\item Suppose that
\[
T \sim \mathrm{Dirichlet}\Bigl(\frac{p_0}{2},\frac{p - p_0}{2},\frac{n - p}{2}\Bigr).
\]
Then
\[
\frac{\norm{(P - P_0)\varepsilon}_2^2}{\norm{(I_n - P_0)\varepsilon}_2^2} = \frac{T_2}{T_2 + T_3} \sim \mathrm{Beta}\Bigl(\frac{p - p_0}{2},\frac{n - p}{2}\Bigr),
\]
which is equivalent to  $\mathsf{F} \sim F_{p - p_0, n - p}$.
\item Suppose that $X$ and $\varepsilon$ are independent, and that at least one of them has a rotationally invariant distribution. Then $\mathsf{F} \sim F_{p - p_0, n - p}$ under~$H_0$. In particular, this holds if the rows of $X$ are independent $N_p(0,\Sigma_X)$ random vectors for some positive definite $\Sigma_X \in \R^{p \times p}$.
\end{enumerate}
\end{proposition}

Proposition~\ref{prop:spherical-sym-A}\textit{(a)} is a standard probabilistic result that motivates the form of Theorem~\ref{thm:dK-W1-A} below, while our main statistical interest lies in part~\textit{(b)}, whose extension in Proposition~\ref{prop:spherical-sym}\textit{(b)} generalises \citet[][Lemma~1]{wen2025residual}. To provide some intuition for this result, it is instructive to consider the simple linear regression case $p = 1$. Here,
\[
\mathsf{F} = (n - 1)\frac{(\tilde{X}^\top\tilde{\varepsilon})^2}{1 - (\tilde{X}^\top\tilde{\varepsilon})^2}
\]
is proportional to the square of the cotangent of the angle between $X$ and $\varepsilon$, where $\tilde{X} := X/\norm{X}_2$ and $\tilde{\varepsilon} := \varepsilon/\norm{\varepsilon}_2$. When $\varepsilon \sim N_n(0,\sigma^2 I_n)$, it is well-known that $\mathsf{F} \sim F_{1,n-1}$. More generally, $\mathsf{F} \sim F_{1,n-1}$ whenever $\varepsilon/\norm{\varepsilon}_2 \sim \Unif(\mathcal{S}^{n-1})$, and therefore by the symmetry of the expression for $\mathsf{F}$ in $X$ and $\varepsilon$, we see that $\mathsf{F} \sim F_{1,n-1}$ when $X/\norm{X}_2 \sim \Unif(\mathcal{S}^{n-1})$, i.e.~when $X$ has a rotationally invariant distribution. Proposition~\ref{prop:spherical-sym-A}\textit{(b)} shows that the exactness continues to hold even when $p \geq 2$ and the definition~\eqref{eq:F-stat} of the $F$-statistic is no longer symmetric in $X$ and~$\varepsilon$.

\subsection{\texorpdfstring{H\"older continuity bounds on the size of the $F$-test}{H\"older continuity bounds on the size of the F-test}}
\label{sec:upper-bds}

While it is interesting to see that the $F$-test remains exact in wider generality than typically presented in textbooks, these settings remain limited.  Our main results therefore seek to study the robustness of the $F$-test to departures of the joint distribution of $(X,\varepsilon)$ from cases where $\mathsf{F} \sim F_{p-p_0,n-p}$ under $H_0$.  To this end, the Kolmogorov distance 
\[
d_\mathrm{K}(\mathsf{F},\mathsf{F}^*) := \sup_{t \in \R} |\Pr(\mathsf{F} \leq t) - \Pr(\mathsf{F}^* \leq t)|
\]
is a measure of distributional proximity between the $F$-statistic $\mathsf{F}$ and the notional $\mathsf{F}^* \sim F_{p-p_0, n-p}$. This measure is particularly relevant for our purposes because if for a given $\alpha \in (0,1)$ we calibrate the $F$-test using the upper $\alpha$-quantile $q_\alpha^*$ of the $F_{p-p_0, n-p}$ distribution, then the size of the test satisfies
\begin{equation}
\label{eq:size-dK}
\Pr(\mathsf{F} > q_\alpha^*) \leq \Pr(\mathsf{F}^* > q_\alpha^*) + d_\mathrm{K}(\mathsf{F},\mathsf{F}^*) = \alpha + d_\mathrm{K}(\mathsf{F},\mathsf{F}^*).
\end{equation}
Thus, $d_\mathrm{K}(\mathsf{F},\mathsf{F}^*)$ controls the discrepancy between the actual Type~I error of the $F$-test and its nominal level. Theorem~\ref{thm:dK-W1-A} below establishes the H\"older continuity of the Kolmogorov distance above in terms of the $L^1$-Wasserstein distance between the distributions of $T$ and $T^* \sim \mathrm{Dirichlet}\bigl(\frac{p_0}{2},\frac{p - p_0}{2},\frac{n - p}{2}\bigr)$.

Recalling the definition of the beta function $\mathrm{B}(a,b) := \Gamma(a)\Gamma(b)/\Gamma(a+b)$ for $a,b > 0$, define
\[
B := \mathrm{B}\Bigl(\frac{p - p_0}{2}, \frac{n - p}{2}\Bigr) \quad \text{and} \quad B_j := \mathrm{B}\Bigl(\frac{p_0}{2}, \frac{n - p_0 - j}{2}\Bigr),
\]
for $j \in \{0,1,2\}$, when $p_0$ and $n - p_0 - j$ are positive integers. It will also be convenient to define
\[
D_1 \equiv D_1(n,p_0) := \frac{B_1}{B_0},
\]
with the interpretation that $D_1(n,0) := \lim_{x \searrow 0} D_1(n,x) = 1$. Further, write $g_{a,b}$ for the density of the $\mathrm{Beta}(a,b)$ distribution, so that 
\begin{equation}
\label{Eq:BetaDensity}
g_{a,b}(x) := \frac{x^{a-1}(1-x)^{b-1}}{\mathrm{B}(a,b)}
\end{equation}
for $x \in (0,1)$.  For $a,b \geq 1$, define 
\[
\mathrm{M}(a,b) := \sup_{x \in (0,1)} g_{a,b}(x) < \infty,
\]
and let $M := \mathrm{M}\bigl(\frac{p - p_0}{2},\frac{n - p}{2}\bigr)$ when $\min(p-p_0,n-p) \geq 2$.

\begin{theorem}
\label{thm:dK-W1-A}
Under $H_0$, we have
\begin{equation}
\label{eq:dK-W1}
\dK(\mathsf{F},\mathsf{F}^*) \leq C(n,p,p_0) \cdot
\begin{cases}
\W_1(nT,nT^*)^{1/3} \;\;&\text{if }k = 1 \\
\W_1(nT,nT^*)^{1/2} \;\;&\text{if }k \geq 2,
\end{cases}
\end{equation}
where $k := \min(p - p_0, n - p)$ and
\begin{align}
\label{eq:W1-constant}
C(n,p,p_0) &:= 
\begin{cases}
3 \cdot 2^{-1/2}\Bigl(\dfrac{\pi D_1}{n^{1/2}B}\Bigr)^{2/3} \leq \dfrac{3\pi^{1/3}}{2^{5/6}} \;\; &\text{if $k = 1$} \\[8pt]
2^{5/4}\Bigl(\dfrac{(n - 2)M}{n(n - p_0 - 2)}\Bigr)^{1/2} < 2^{3/4} \;\; &\text{if $k \geq 2$.}
\end{cases}
\end{align}
\end{theorem}

Lemma~\ref{lem:dK-W1-constant} shows that 
\[
\frac{1}{9}\Bigl(\frac{n - p_0}{(p - p_0)(n - p)}\Bigr)^{1/4} < C(n,p,p_0) < 3\Bigl(\frac{n - p_0}{(p - p_0)(n - p)}\Bigr)^{1/4}
\]
in all cases, where
\[
\frac{1}{k} = \max\Bigl(\frac{1}{p - p_0},\frac{1}{n - p}\Bigr) \leq \frac{n - p_0}{(p - p_0)(n - p)} = \frac{1}{p - p_0} + \frac{1}{n - p} \leq 2.
\]
Thus, for all values of $n$, $p$ and $p_0$, the universal constant upper bound on $C(n,p,p_0)$ in~\eqref{eq:W1-constant} is tight up to a universal constant multiple of $k^{1/4} \leq (p - p_0)^{1/4}$.

Theorem~\ref{thm:dK-W1-A} immediately implies Proposition~\ref{prop:spherical-sym-A}\emph{(b)}, because if $T \sim \mathrm{Dirichlet}\bigl(\frac{p_0}{2},\frac{p - p_0}{2},\frac{n - p}{2}\bigr)$, then $\W_1(nT,nT^*) = 0$, so $\dK(\mathsf{F},\mathsf{F}^*) = 0$.  But the result also goes much further, controlling the maximal inflation of the Type~I error as the distribution of $T$ moves away from this Dirichlet distribution.  On the one hand, this provides a notion of robustness of the $F$-test to departures from the canonical setting of independent, homoscedastic Gaussian errors, but on the other, the H\"older continuity exponents in Theorem~\ref{thm:dK-W1-A} are relatively small, and Proposition~\ref{prop:dK-W1-lb} below establishes that they are unimprovable in general. Thus, as a very rough guide, even a perturbation of the joint distribution of $(X,\varepsilon)$ leading to $\W_1(nT,nT^*) \approx 10^{-4}$ could potentially yield a non-trivial Type~I error inflation of order $0.01$.  As further intuition, albeit that the settings are very different, sample paths of Brownian motion (which are generally thought of as very rough), are almost surely locally H\"older continuous with exponent $\gamma$ for every $\gamma \in (0,1/2)$, but not for $\gamma > 1/2$ \citep[e.g.][Corollary~1.20 and Exercise~1.9]{morters2010brownian}. 

To give some explanation of the difference between the exponents in the cases $k = 1$ and $k \geq 2$, recall that if $\mathscr{P},\mathscr{Q}$ are Borel probability measures on $\R$, with distribution functions $F,G$ respectively, then
\[
\dK(\mathscr{P},\mathscr{Q}) = \sup_{x \in \R} |F(x) - G(x)| = \norm{F - G}_\infty, \quad\;\; \W_1(\mathscr{P},\mathscr{Q}) = \int_{-\infty}^\infty |F(x) - G(x)|\,dx = \norm{F - G}_1.
\]
In Lemma~\ref{lem:holder-L1-Linfty}, we verify that if $G$ is $(\gamma,L)$-H\"older for some $\gamma \in (0,1]$ and $L > 0$, then
\begin{equation}
\label{eq:holder}
\norm{F - G}_\infty \leq L^{1/(\gamma + 1)}\Bigl(\frac{\gamma + 1}{\gamma} \norm{F - G}_1\Bigr)^{\gamma/(\gamma + 1)}.
\end{equation}
Moreover, this bound can be attained with equality, so the exponent $\gamma/(\gamma + 1)$ is the best possible. In the context of Theorem~\ref{thm:dK-W1-A}, it turns out that the relevant reference probability distribution is $\mathscr{Q} = \mathrm{Beta}(\frac{p - p_0}{2},\frac{n - p}{2})$. When $k = 1$, its distribution function $G^*$ is $\gamma$-H\"older with $\gamma = 1/2$ and hence $\gamma/(\gamma + 1) = 1/3$. On the other hand, if $k \geq 2$, then $\mathscr{Q}$ has a bounded density on $(0,1)$ and $G^*$ is therefore Lipschitz, so $\gamma = 1$ and $\gamma/(\gamma + 1) = 1/2$. These two cases are illustrated in Figure~\ref{fig:dK-W1}.  See \citet{gaunt2023bounding} for further inequalities relating Kolmogorov and integral probability metrics.

\begin{figure}
\centering
\resizebox{0.8\linewidth}{!}{%
\begin{tikzpicture}[
x=0.85cm,
y=0.85cm,
>={Stealth[length=3pt, width=4pt]},
line cap=round,
line join=round,
every node/.style={font=\small}
]

\def\wth{4}
\def\hgt{4}
\def\dlt{0.50}
\def\dlts{0.90}

\definecolor{softblue}{RGB}{220,235,250}
\definecolor{arrowblue}{RGB}{47,111,183}
\definecolor{arrowred}{RGB}{233,84,129}
\begin{scope}[xshift=-3.5cm]
\draw[->] (-0.15,0) -- (\wth+0.35,0);
\draw[->] (0,-0.15) -- (0,\hgt+0.35);

\draw[black, line width=1.25pt]
plot[smooth, domain=0.001:0.999, samples=160]
({\wth*(\x*\x)/(\x*\x+(1-\x)*(1-\x))},{\hgt*\x});

\pgfmathsetmacro{\uend}{sqrt(\dlt/\wth)/(1 + sqrt(\dlt/\wth))}
\pgfmathsetmacro{\yend}{\hgt*\uend}

\fill[softblue, opacity=0.75]
(0,0)
plot[smooth, domain=0.001:\uend, samples=100]
({\wth*(\x*\x)/(\x*\x+(1-\x)*(1-\x))},{\hgt*\x})
-- (\dlt,\yend)
-- (\dlt,0)
-- cycle;

\draw[gray!75, dashed]
(\dlt,0) -- (\dlt,\yend);

\pgfmathsetmacro{\uendR}{1 - sqrt(\dlt/\wth)/(1 + sqrt(\dlt/\wth))}
\pgfmathsetmacro{\yendR}{\hgt*\uendR}
\pgfmathsetmacro{\xendR}{\wth-\dlt}

\fill[arrowred, opacity=0.4]
(\wth,\hgt)
plot[smooth, domain=0.999:\uendR, samples=100]
({\wth*(\x*\x)/(\x*\x+(1-\x)*(1-\x))},{\hgt*\x})
-- (\xendR,\yendR)
-- (\wth,\yendR)
-- cycle;

\draw[gray!75, dashed]
(\xendR,\yendR) -- (\wth,\yendR);


\draw[<->, arrowblue]
(0,-0.9) -- (\dlt,-0.9)
node[midway, below=2pt] {$\delta$};

\draw[<->, arrowblue]
(-0.45,0) -- (-0.45,\yend)
node[midway, left=2pt] {$\asymp \delta^{1/2}$};

\draw[<->, arrowred]
({\wth-\dlt},{\hgt+0.2}) -- (\wth,{\hgt+0.2})
node[midway, above=2pt] {$\delta$};

\draw[<->, arrowred]
({\wth+0.3},{\yendR}) -- ({\wth+0.3},\hgt)
node[midway, right=2pt] {$\asymp \delta^{1/2}$};

\node[below=3pt] at (0,0) {$0$};
\node[below=3pt] at (\wth,0) {$1$};
\node[left=3pt] at (0,\hgt) {$1$};

\node[above=20pt] at (\wth/2,\hgt+0.15)
{{\color{arrowblue}{$p - p_0 = 1$}} \:or\: {\color{arrowred}$n - p = 1$}};
\end{scope}

\begin{scope}[xshift=3.5cm]
\draw[->] (-0.15,0) -- (\wth+0.35,0);
\draw[->] (0,-0.15) -- (0,\hgt+0.35);

\draw[black, line width=1.25pt]
plot[smooth, domain=0:1, samples=160]
({\wth*\x},{\hgt*(3*\x*\x-2*\x*\x*\x)});

\pgfmathsetmacro{\xL}{\wth/2}
\pgfmathsetmacro{\xR}{\wth/2+\dlt}
\pgfmathsetmacro{\uL}{\xL/\wth}
\pgfmathsetmacro{\uR}{\xR/\wth}
\pgfmathsetmacro{\yL}{\hgt*(3*\uL*\uL-2*\uL*\uL*\uL)}
\pgfmathsetmacro{\yR}{\hgt*(3*\uR*\uR-2*\uR*\uR*\uR)}

\fill[gray, opacity=0.3]
(\xL,\yL)
plot[smooth, domain=\uL:\uR, samples=100]
({\wth*\x},{\hgt*(3*\x*\x-2*\x*\x*\x)})
-- (\xR,\yL)
-- cycle;

\draw[gray!75, dashed]
(\xR,\yL) -- (\xR,\yR);

\draw[gray!75, dashed]
(\xL,\yL) -- (\xR,\yL);

\draw[<->]
(\xL,\yL-0.3) -- (\xR,\yL-0.3)
node[midway, below=2pt] {$\delta$};

\draw[<->]
(\xR+0.3,\yL) -- (\xR+0.3,\yR)
node[midway, right=2pt] {$\asymp \delta$};

\node[below=3pt] at (0,0) {$0$};
\node[below=3pt] at (\wth,0) {$1$};
\node[left=3pt] at (0,\hgt) {$1$};


\node[above=20pt] at (\wth/2,\hgt+0.15)
{$\min(p-p_0,n-p)\geq 2$};
\end{scope}
\end{tikzpicture}}
\caption{Plots illustrating the distribution function $G^*$ of $\mathrm{Beta}\bigl(\frac{p - p_0}{2},\frac{n - p}{2}\bigr)$ in the cases $p - p_0 = 1$ or $n - p = 1$ (\emph{left}), and $\min(p - p_0, n - p) \geq 2$ (\emph{right}), and the types of perturbations of $G^*$ yielding the optimal exponents of $1/3$ and $1/2$ respectively in~\eqref{eq:holder}.}
\label{fig:dK-W1}
\end{figure}
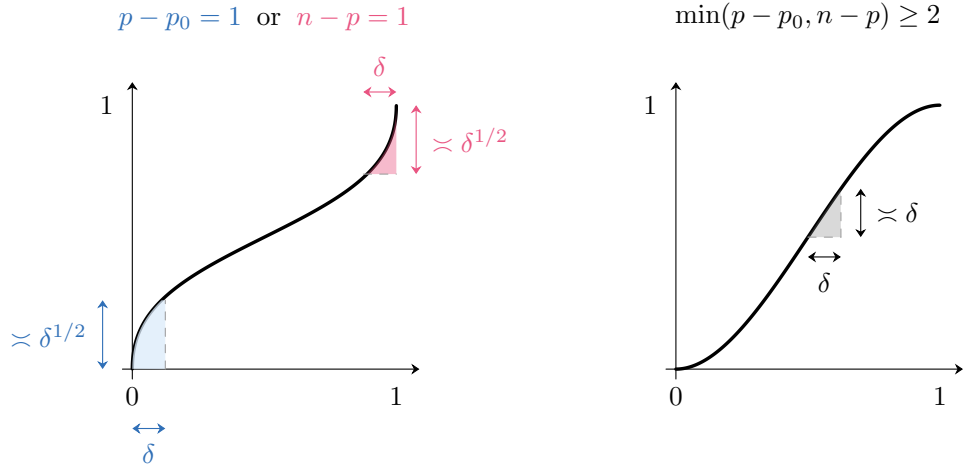

Although it is not the main focus of this work, Theorem~\ref{thm:dK-W1-A} can also be used to establish asymptotic Type~I error control of the $F$-test under second moment conditions and independence of errors and covariates; see Proposition~\ref{Prop:Asymptotic} and Corollary~\ref{cor:asymptoticF}.

The lower bounds in the following result confirm that the H\"older exponents obtained in Theorem~\ref{thm:dK-W1-A} cannot be improved in general, even when the rows $X_1,\dotsc,X_n$ of the design matrix~$X$ and the coordinates $\varepsilon_1,\dotsc,\varepsilon_n$ of the error vector $\varepsilon$ are required to be exchangeable.

\begin{proposition}
\label{prop:dK-W1-lb}
\begin{enumerate}[label=(\alph*),leftmargin=0.7cm]
\item 
For any $\eta > 0$, there exist an $n \times p$ design matrix $X$ and an $n$-dimensional error vector $\varepsilon$ such that $X$ and $\varepsilon$ are independent, $\W_1(nT,nT^*) < \eta$ and
\begin{equation}
\label{eq:dK-W1-lb}
\dK(\mathsf{F},\mathsf{F}^*) > \frac{C(n,p,p_0)}{9} \cdot
\begin{cases}
	\W_1(nT,nT^*)^{1/3} \;\;&\text{if }k = 1 \\
	\W_1(nT,nT^*)^{1/2} \;\;&\text{if }k \geq 2.
\end{cases}
\end{equation}
Here, $C(n,p,p_0) > 0$ is the multiplicative factor in Theorem~\ref{thm:dK-W1-A}.
\item For any $\eta > 0$, there exist $X$ and $\varepsilon$ satisfying all of the properties in~\textit{(a)} such that $(X_1,\varepsilon_1),\dotsc,(X_n,\varepsilon_n)$ are also exchangeable, and~\eqref{eq:dK-W1-lb} holds with $C(n,p,p_0)$ replaced by some $c > 0$ still depending only on $n,p,p_0$.
\end{enumerate}
In both parts, we can also ensure that $\Pr(\mathsf{F} > q) \geq \Pr(\mathsf{F}^* > q)$ for all $q > 1$.
\end{proposition}

Examples~\ref{ex:exchangeable-lb-1} and~\ref{ex:exchangeable-lb-2} in the appendix present more explicit probabilistic constructions than the geometric construction employed in the proof of Proposition~\ref{prop:dK-W1-lb}\emph{(b)}, but still respect the exchangeability and independence conditions in that result.  For general $n$, these treat the cases $(p_0,p) = (0,1)$ and $(p_0,p) = (0,2)$ respectively, corresponding to examples of the two cases in~\eqref{eq:dK-W1-lb}.  In particular, Example~\ref{ex:exchangeable-lb-1} illustrates the final assertion of Proposition~\ref{prop:dK-W1-lb}, namely that the $F$-test can be anti-conservative, i.e.~the size can exceed the nominal level.

Theorem~\ref{thm:dK-W1-A} and Proposition~\ref{prop:dK-W1-lb} bound a distance between complicated ratios in the definitions of $\mathsf{F}$ and $\mathsf{F}^*$ in terms of a distance between more primitive quantities, but these still depend jointly on $X$ and $\varepsilon$. In Theorem~\ref{thm:doubly-robust-A} below, we therefore provide a more interpretable bound involving only distances between the marginal distributions of appropriately normalised versions of $X$ and $\varepsilon$.

\begin{theorem}
\label{thm:doubly-robust-A}
Suppose that $X$ and $\varepsilon$ are independent, and let $C_{\varepsilon},C_X \geq 1$ be such that
\begin{equation}
\label{eq:exp-proj-eval-S}
\norm{\E(P)}_{\mathrm{op}}
\leq \frac{C_X p}{n}
\quad\text{and}\quad 
\norm{\E(\xi\xi^\top)}_{\mathrm{op}}
\leq \frac{C_\varepsilon}{n},
\end{equation}
where $\xi := \varepsilon/\norm{\varepsilon}_2$. Let $\xi^* \sim \Unif(\mathcal{S}^{n-1})$, $Q^* \sim \Unif(\mathcal{O}_{n \times p})$, $P^* := Q^*{Q^*}^\top$ and $P_0^* := Q_{[p_0]}^*{Q_{[p_0]}^*}^\top$. Writing
\begin{align*}
A := \min\biggl\{(C_X p)^{1/2}\,\W_2(\xi,\xi^*),\,
C_\varepsilon^{1/2}\,\W_2\biggl(\binom{P_0}{P - P_0}, \binom{P_0^*}{P^* - P_0^*}\biggr)\biggr\},
\end{align*}
we have
\begin{equation}
\label{eq:doubly-robust}
\W_1(nT,nT^*) \leq 2^{1/2}(A^2 + 2p^{1/2}A).
\end{equation}
Moreover,
\[
\W_2\biggl(\binom{P_0}{P - P_0}, \binom{P_0^*}{P^* - P_0^*}\biggr) \leq 2^{1/2}\inf_{Q \in \mathcal{Q}} \W_2(Q,Q^*),
\]
where $\mathcal{Q}$ denotes the set of all random $Q$ taking values in $\mathcal{O}_{n \times p}$ such that $\mathrm{Im}(Q_{[p_0]}) = \mathrm{Im}(X_{[p_0]})$ and $\mathrm{Im}(Q) = \mathrm{Im}(X)$.
\end{theorem}

Recall from Proposition~\ref{prop:spherical-sym-A}\textit{(b)} that $\W_1(nT,nT^*) = 0$ when at least one of $X$ or $\varepsilon$ has a rotationally invariant distribution. Theorem~\ref{thm:doubly-robust-A} controls $\W_1(nT,nT^*)$ more generally in terms of the quantity $A$, which is a \emph{minimum} of two $L^2$-Wasserstein distances that quantify departures from rotational invariance for $\varepsilon$ and $X$ individually. Since $\xi,\xi^*$ and all columns of $Q \in \mathcal{Q}$ and $Q^*$ take values in $\mathcal{S}^{n-1}$, we have $\norm{\xi - \xi^*}_2 \leq 2$ and $\norm{Q - Q^*}_{\mathrm{F}} \leq \norm{Q}_{\mathrm{F}} + \norm{Q^*}_{\mathrm{F}} = 2p^{1/2}$ almost surely. Therefore, $A \leq 2p^{1/2}\min(C_X,2C_\varepsilon)^{1/2}$, so Theorem~\ref{thm:doubly-robust-A} reveals that $\W_1(nT,nT^*)$ is Lipschitz in $A$. The main interest lies in settings where $A$ is small, i.e.~in small departures from cases where the $F$-test is exact, in which case the second term of order $A$ on the right-hand side of~\eqref{eq:doubly-robust} is the dominant one. The combination of Theorems~\ref{thm:dK-W1-A} and~\ref{thm:doubly-robust-A} provides a precise sense in which the $F$-test is doubly robust, because under the conditions of these results, the Type~I error of the $F$-test is guaranteed to be close to its nominal level as soon as either the distribution of $\xi$ is close to uniform on~$\mathcal{S}^{n-1}$, or there exists $Q \in \mathcal{Q}$ whose distribution is close to that of $Q^* \sim \Unif(\mathcal{O}_{n \times p})$. Here, the distribution of the projection matrix $P^* = Q^*{Q^*}^\top$ is characterised by the invariance property $Q_0 P^*Q_0^\top \eqd P^*$ for all $Q_0 \in \mathcal{O}_{n \times n}$ \citep[see e.g.][Section~6.3]{eaton2007multivariate}.

\begin{remark*}
For any $Q,Q' \in \mathcal{Q}$, we have $Q = Q'\Lambda$ for some matrix of the form $\Lambda = 
\bigl(\begin{smallmatrix}
\Lambda_1 & 0 \\
0 & \Lambda_2
\end{smallmatrix}
\bigr)$
with $\Lambda_1 \in \mathcal{O}_{p_0 \times p_0}$ and $\Lambda_2 \in \mathcal{O}_{(p - p_0) \times (p - p_0)}$. One element of $\mathcal{Q}$ corresponds to applying the Gram--Schmidt procedure to the columns of $X$, namely the matrix $Q \in \mathcal{O}_{n \times p}$ in the QR decomposition $X = QR$, which is unique (when $R$ has positive diagonal entries) and satisfies $\mathrm{Im}(Q_{[j]}) = \mathrm{Im}(X_{[j]})$ for every $j \in [p]$ \citep[e.g.][Proposition~5.2]{eaton2007multivariate}.

Since $1 = \E(\norm{\xi}_2^2) = \tr\E(\xi\xi^\top) \leq n\norm{\E(\xi\xi^\top)}_{\mathrm{op}}$, we must have $C_\varepsilon \geq 1$ in~\eqref{eq:exp-proj-eval-S}, with equality if and only if all eigenvalues of $\E\bigl(\xi\xi^\top)$ are equal to $1/n$, i.e.~$\E\bigl(\xi\xi^\top) = I_n/n$. In particular, $C_\varepsilon = 1$ when $\xi \eqd \xi^* \sim \Unif(\mathcal{S}^{n-1})$, i.e.~$\varepsilon$ has a rotationally invariant distribution. Similarly, $p = \E\tr(P) = \tr\E(P) \leq n\norm{\E(P)}_{\mathrm{op}}$, so $C_X \geq 1$ and equality holds if and only if $\E(P) = pI_n/n$, which occurs when $P \eqd P^*$. This is the case if $Q \sim \Unif(\mathcal{O}_{n \times p})$ for some $Q \in \mathcal{Q}$ (whose columns form an orthonormal basis of $\mathrm{Im}(X)$ by definition).

In general, $\norm{\E(P)}_{\mathrm{op}} \leq \E(\norm{P}_{\mathrm{op}}) = 1$. Equality holds if and only if there exists some deterministic $v \in \mathrm{Im}(X)$, since in this case $Pv = v$ almost surely and hence $\E(P)v = v$, or equivalently $\norm{\E(P)}_{\mathrm{op}} = 1$. In particular, this occurs when $X$ has a deterministic column. In this case, we must take $C_X \geq n/p$, which may inflate the value of $A$ and hence the bound~\eqref{eq:doubly-robust}. We obtain a tighter bound in Theorem~\ref{thm:doubly-robust} by first orthogonalising out some of the deterministic covariates (e.g.~an intercept term) in the null model.
\end{remark*}

Finally in this section, we remark that quantifying departures from exactness of the $F$-test in terms of a Wasserstein distance between relevant vectors such as $\xi$ and $\xi^*$ is tighter and more natural than alternatives such as the total variation distance $\mathrm{TV}(\xi,\xi^*)$. Since $\xi$ and $\xi^*$ take values in $\mathcal{S}^{n-1}$ as mentioned above, we always have $\W_1(\xi,\xi^*) \leq 2\mathrm{TV}(\xi,\xi^*)$ \citep[Theorem~4]{gibbs2002choosing}, but this bound can be extremely loose in certain cases. Indeed, the total variation distance is 1 as soon as $\xi$ is discrete (even if it is a very fine quantisation of~$\xi^*$, for instance, so that $\W_1(\xi,\xi^*)$ can be arbitrarily small). Similar remarks show that other $f$-divergences such as Hellinger distance, Kullback--Leibler divergence or chi-squared divergence are also inappropriate, owing to standard inequalities relating these notions to the total variation distance \citep[][Exercise~8.2\emph{(a)}]{samworth24modern}. By contrast, Wasserstein distances have a natural interpretation in terms of mass transportation and are known to have several attractive properties \citep{villani2008optimal}.  We also note that for random vectors in $\mathcal{S}^{n-1}$, the Wasserstein distance $\W_1$ is closely related to the bounded Lipschitz distance $d_{\mathrm{BL}}$; indeed, $d_{\mathrm{BL}}(\xi,\xi^*) \leq \W_1(\xi,\xi^*) \leq 2d_{\mathrm{BL}}(\xi,\xi^*)$.

\subsection{Bounds on the local Kolmogorov distance}
\label{sec:local-dK}

Recall that the Kolmogorov distance controls the Type~I error inflation of the $F$-test via the inequality~\eqref{eq:size-dK}.  On the other hand, since we reject the null hypothesis for large values of $\mathsf{F}$, practitioners are primarily interested in departures from the nominal $F$-distribution in the right tail.  In this sense, the inequality~\eqref{eq:size-dK} may be loose when departures occur elsewhere in the support of the distribution. With a view to strengthening the bound, we now introduce a \emph{local Kolmogorov pseudometric}\footnote{A pseudometric $d$ has the same properties as a metric except that we allow $d(x,y) = 0$ even when $x \neq y$.} between the $F$-statistic $\mathsf{F}$ and the notional $\mathsf{F}^* \sim F_{p - p_0, n - p}$, namely
\[
d_{\mathrm{K},\alpha}(\mathsf{F},\mathsf{F}^*) := \sup_{t \geq q_\alpha^*}|\Pr(\mathsf{F} > t) - \Pr(\mathsf{F}^* > t)| = \sup_{\alpha' \in (0,\alpha]}|\Pr(\mathsf{F} > q_{\alpha'}^*) - \alpha'|.
\]
As before, $q_{\alpha'}^*$ denotes the $(1 - \alpha')$-quantile of the $F_{p-p_0, n-p}$ distribution. Theorem~\ref{thm:dK-alpha} below establishes the H\"older continuity of the local Kolmogorov pseudometric above in terms of the $L^1$-Wasserstein distance between the distributions of $T$ and $T^* \sim \mathrm{Dirichlet}\bigl(\frac{p_0}{2},\frac{p - p_0}{2},\frac{n - p}{2}\bigr)$. This yields a more refined bound
\[
\Pr(\mathsf{F} > q_\alpha^*) \leq \alpha + d_\mathrm{K,\alpha}(\mathsf{F},\mathsf{F}^*)
\]
on the Type I error of the $F$-test.

\begin{theorem}
\label{thm:dK-alpha}
For $\alpha \in (0,1/2)$, we have
\begin{equation}
\label{eq:dK-alpha}
d_{\mathrm{K},\alpha}(\mathsf{F},\mathsf{F}^*) \lesssim_{n,p,p_0}
\begin{cases}
\W_1(nT,nT^*)^{1/3} \;\; &\text{if }n - p = 1 \\
\max\bigl\{\alpha^{\frac{n - p - 2}{2(n - p)}}\W_1(nT,nT^*)^{1/2},\,\W_1(nT,nT^*)^{\frac{n - p}{n - p + 2}}\bigr\} \;\; &\text{if }n - p \geq 2.
\end{cases}
\end{equation}
\end{theorem}

Here, the notation $a \lesssim_{n,p,p_0} b$ or $b \gtrsim_{n,p,p_0} a$ means that $a \leq Cb$ for some $C > 0$ depending only on $n,p,p_0$.  We defer discussion of the conclusion of Theorem~\ref{thm:dK-alpha} until after the following refinement of Proposition~\ref{prop:dK-W1-lb}\textit{(b)}, which shows that~\eqref{eq:dK-alpha} can be tight.

\begin{proposition}
\label{prop:dK-alpha-lb}
Fix $\alpha \in (0,1/2)$.  For any $\eta > 0$, there exist an $n \times p$ design matrix $X$ and an $n$-dimensional error vector $\varepsilon$ such that $X$ and $\varepsilon$ are independent, $(X_1,\varepsilon_1),\dotsc,(X_n,\varepsilon_n)$ are exchangeable, $\W_1(nT,nT^*) < \eta$ and
\begin{equation}
\label{eq:dK-alpha-lb}
d_{\mathrm{K},\alpha}(\mathsf{F},\mathsf{F}^*) \gtrsim_{n,p,p_0}
\begin{cases}
\W_1(nT,nT^*)^{1/3} \;\; &\text{if }n - p = 1 \\
\max\bigl\{\alpha^{\frac{n - p - 2}{2(n - p)}}\W_1(nT,nT^*)^{1/2},\,\W_1(nT,nT^*)^{\frac{n - p}{n - p + 2}}\bigr\} \;\; &\text{if }n - p \geq 2.
\end{cases}
\end{equation}
Moreover, we can ensure that $\Pr(\mathsf{F} > q) \geq \Pr(\mathsf{F}^* > q)$ for all $q > 1$.
\end{proposition}

When $p_0 = 0$, $p = 1$ and $n \geq 3$, Example~\ref{ex:exchangeable-lb-1} in the appendix explicitly constructs a joint distribution of $(X,\varepsilon)$ for which the $F$-statistic attains the lower bound in Proposition~\ref{prop:dK-alpha-lb} and yields an anti-conservative test.  When $n - p = 1$, the bounds in~\eqref{eq:dK-alpha} and~\eqref{eq:dK-alpha-lb} of order $\W_1(nT,nT^*)^{1/3}$ coincide with those in the $k = 1$ case of Theorem~\ref{thm:dK-W1-A} and Proposition~\ref{prop:dK-W1-lb} (up to a multiplicative factor depending on $n,p,p_0$ but not $\alpha$). The main difference lies in the case $n - p \geq 2$, where
\begin{align}
&\max\bigl\{\alpha^{\frac{n - p - 2}{2(n - p)}}\W_1(nT,nT^*)^{1/2},\,\W_1(nT,nT^*)^{\frac{n - p}{n - p + 2}}\bigr\} \notag \\
\label{eq:dK-alpha-bd-2}
&\hspace{2cm}= 
\begin{cases}
\W_1(nT,nT^*)^{\frac{n - p}{n - p + 2}} \;\;&\text{if }\alpha \leq \W_1(nT,nT^*)^{\frac{n - p}{n - p + 2}} \\
\alpha^{\frac{n - p - 2}{2(n - p)}}\W_1(nT,nT^*)^{1/2} \;\;&\text{if }\alpha > \W_1(nT,nT^*)^{\frac{n - p}{n - p + 2}},
\end{cases}
\end{align}
in contrast to $\W_1(nT,nT^*)^{1/3}$ in our earlier results when $p - p_0 = 1 < n - p$, and $\W_1(nT,nT^*)^{1/2}$ when $\min(p - p_0, n - p) \geq 2$. In the first case in~\eqref{eq:dK-alpha-bd-2}, the H\"older exponent $\frac{n - p}{n - p + 2}$ is always at least $1/2$ when $n - p \geq 2$, and approaches 1 as $n - p$ increases. Thus, when $n-p \geq 2$, the bound on $d_{\mathrm{K},\alpha}(\mathsf{F},\mathsf{F}^*)$ in~\eqref{eq:dK-alpha} is indeed tighter than the earlier bounds on the global Kolmogorov distance $d_{\mathrm{K}}(\mathsf{F},\mathsf{F}^*)$ (up to a multiplicative factor depending on $n,p,p_0$), and scales with $\alpha$ in the regime $\alpha > \W_1(nT,nT^*)^{\frac{n - p}{n - p + 2}}$ when $n - p \geq 3$.

The two terms in Theorem~\ref{thm:dK-alpha} and Proposition~\ref{prop:dK-alpha-lb} reflect different perturbation regimes. When~$\alpha$ is sufficiently large relative to the Wasserstein discrepancy, the corresponding critical value for the $F$-test (once transformed onto the appropriate beta distribution scale) lies away from the right-hand endpoint of the beta distribution.  A local perturbation then yields a H\"older exponent of $1/2$, adjusted by the density at the $(1-\alpha)$-level quantile. For very small~$\alpha$, however, the critical value approaches this beta distribution boundary, where a horizontal perturbation of $h$ in the upper tail of the beta distribution function results in a vertical perturbation of order $h^{(n-p)/2}$.  Transporting this mass incurs a Wasserstein cost of order $h^{(n-p+2)/2}$, and therefore leads to a H\"older exponent of $\frac{n - p}{n - p + 2}$.

\section{Related work and discussion}
\label{sec:relatedWork}

There has been a long line of work since the 1920s studying the performance of the $F$-test under non-normal errors. Relevant early papers include those by \citet{pearson1929notes}, \citet{pearson1931analysis}, \citet{bartlett1935effect}, \citet{geary1947testing}, \citet{david1951effect}. On the other hand, more refined theoretical and empirical analyses in subsequent decades revealed a more nuanced picture, which is encapsulated by the following quote from \citet{ali1996robustness} in their exploration of the robustness of regression $F$-tests:
\begin{quote}
Besides the sample size and the degrees of freedom of error sum of squares, the major determinant of the sensitivity [of the $F$-statistic] to nonnormality [of the errors] is the extent of the ‘nonnormality’ of the regressors or the extent of presence of ‘leveraged’ (influential) observations.
\end{quote}
We now highlight some key strands of the more recent literature that address the different aspects mentioned above. First, works including those by \citet{efron1969student}, \citet{benjamini1983t} and \citet{pinelis1994extremal} specifically studied the one-sample $t$-statistic and its multivariate extension, Hotelling's $T^2$ \citep{hotelling31generalization}. The $t$-test was seen to be conservative when the errors are orthant-symmetric, i.e.~$(\varepsilon_1,\dotsc,\varepsilon_n) \eqd (\eta_1\varepsilon_1,\dotsc,\eta_n\varepsilon_n)$ for any fixed $\eta_1,\dotsc,\eta_n \in \{-1,1\}$. In this scenario, the most refined finite-sample bounds obtained by \citet{pinelis1994extremal} showed that all tail probabilities of the $t$-statistic are bounded above by a universal constant multiple of their chi-squared counterparts. The robustness of the $t$-test more generally can be attributed in part to the regularising effect of self-normalisation \citep[e.g.][]{logan1973limit,bentkus1996berry,shao1997self,gine1997student,david2005asymptotic}.

As for the general linear regression model, early works focused more on fixed design settings.  For instance, \citet{zellner1976bayesian} studied the case where the distribution of the error vector is a scale mixture of isotropic Gaussians (e.g.~a centred multivariate $t$-distribution with isotropic scale matrix), while \citet{kariya1977robust} considered the setting of spherically symmetric errors with a finite second moment.  \citet{jensen1979linear} established the general exactness of the $F$-test under spherical symmetry of the error distribution (without additional moment conditions).

Under the key assumption that the maximum \textit{leverage} (namely $\max_{i \in [n]} X_i^\top(X^\top X)^{-1}X_i = \max_{i \in [n]} P_{ii}$) tends to 0 in a triangular array setting, \citet{huber1973robust} established the asymptotic normality of all one-dimensional contrasts of the OLS estimator when $p/n = \tr(P)/n \to 0$. It follows that the $F$-test has a limiting chi-squared distribution in this regime \citep{arnold1980asymptotic}. The leverage condition says roughly that no covariate $X_i$ can have a disproportionate asymptotic influence on the fitted estimator, which is essentially why it fulfils the hypotheses of the Lindeberg--Feller central limit theorem \citep[Example~2.28]{vdV1998asymptotic}.

The effect of the design matrix on the sensitivity of the $F$-test to non-Gaussian errors was previously recognised by \citet{box1962robustness}. They quantified this phenomenon using a `correction factor' obtained by approximating the distribution of the test statistic under random permutations of the error vector. The two illustrative fixed-design examples in their Section~3 are particularly pertinent to this discussion. First, consider the $p$-group one-way ANOVA model with group sizes $n_1,\dotsc,n_p$ summing to $n$, for which
\[
X^\top X = \mathrm{diag}(n_1,\dotsc,n_p) \quad\text{and}\quad
P = \mathrm{diag}\Bigl(\frac{I_{n_1}}{n_1},\dotsc,\frac{I_{n_p}}{n_p}\Bigr)
\]
and hence the maximum leverage is $1/\min_{j \in [p]} n_j$. Many subsequent works studied the contrast between balanced ANOVA designs where $\min_{j \in [p]} n_j \asymp n/p$, as opposed to unbalanced groups where those with small $n_j$ have excessive influence. (The numerical calculations of \citet{box1955permutation}, later reproduced by \citet{scheffe1959analysis}, already indicated substantial robustness of balanced ANOVA $F$-tests to moderate skewness and kurtosis.) In particular, the one-dimensional $t$-test corresponds to the special case of a single group of size $n_1 = n$, where the aforementioned robustness can be viewed through the lens of the maximum leverage taking its smallest possible value $1/n$.

In the second example given by \citet{box1962robustness},
\[
X = 
\begin{pmatrix}
I_p & 0 \\
0 & 0
\end{pmatrix}
\in \R^{n \times p}
\quad\text{and}\quad
P = 
\begin{pmatrix}
I_p & 0 \\
0 & 0
\end{pmatrix}
\in \R^{n \times n},
\]
so the maximum leverage attains its largest possible value of 1. Here, the $F$-statistic under the global null ($p_0 = 0$) is
\[
\mathsf{F} = \frac{\frac{1}{p}\norm{PY}_2^2}{\frac{1}{n-p}\norm{(I_n - P)Y}_2^2} = \frac{n - p}{p} \cdot \frac{\sum_{i=1}^p Y_i^2}{\sum_{i=p+1}^n Y_i^2}.
\]
In a different context, this is used as a two-sample test for equality of variances. The considerations above support the assertion of  \citet{box1962robustness} that `by different choices of the~$x$ vectors, almost the same regression model
can be made to reproduce on the one hand a test to compare means which is little affected by non-normality and on the other a comparison of variances test which is notoriously sensitive to non-normality'. \citet{ali1996robustness} developed this theory further by deriving an expansion of the Type~I error deviation of the $F$-test from its nominal level, in terms of quantities depending on both the covariates and errors. 

The relevance of this discussion to our theory is that the Wasserstein bound~\eqref{eq:doubly-robust} in Theorem~\ref{thm:doubly-robust-A} (controlling $\dK(\mathsf{F},\mathsf{F}^*)$ through Theorem~\ref{thm:dK-W1-A}) depends monotonically on the quantity $C_X$ appearing in the condition $\norm{\E(P)}_{\mathrm{op}} \leq C_X p/n$. Here, $\norm{\E(P)}_{\mathrm{op}}$ is our random design analogue of the maximum leverage, which is $p/n$ in the most favourable rotationally invariant (`balanced') setting where we can take $C_X = 1$. On the other hand, $\norm{\E(P)}_{\mathrm{op}}$ can attain its largest possible value $1$ (so we require $C_X \geq n/p$) in unbalanced designs such as the second example above.

On a related note, \citet{wen2025residual} not only establish Proposition~\ref{prop:spherical-sym-A}\textit{(b)} for the regression $t$-test specifically ($p - p_0 = 1$) in their Section~3, but also present simulations demonstrating the invalidity of this test when either the covariates or errors (which are independent) have some heavy-tailed entries. One plausible explanation is that heavy-tailed covariates generate highly variable leverage scores, allowing a few observations to exert disproportionate influence on both the numerator and denominator of the test statistic, especially when $p$ is comparable to $n$.

The performance of the $F$-test in high-dimensional regimes has been studied more intensively in more recent years, beginning with papers by \citet{boos1995anova}, \citet{akritas2000asymptotics}, \citet{akritas2004heteroscedastic} and \citet{orme2006asymptotic} on ANOVA. Here, we discuss some illuminating results of \citet{calhoun2011hypothesis} and \citet{anatolyev2012inference}. In Section~3 of the former, the classical $F$-test (calibrated with $F$-quantiles) achieves relatively good empirical Type~I error control in some non-Gaussian examples (e.g.~with Cauchy covariates and $t$-distributed errors) when $p - p _0 = 1$ and $p/n \in \{0.1, 0.5\}$. However, its performance deteriorates when $p - p_0$ is larger (e.g.~equal to $p - 1$). \citet{anatolyev2012inference} studied the asymptotic regime where $p/n$ converges to a non-zero constant and the errors have finite fourth moments. Under leverage conditions on $P$ and $P_0$, he proved that $q\mathsf{F} \cvd \chi_q^2$ when $q := p - p_0$ is fixed, while on the other hand $\sqrt{q}(\mathsf{F} - 1) \to N\bigl(0,\lim_{n \rightarrow \infty} \frac{2(n - p_0)}{n - p}\bigr)$ when $q$ also grows proportionally to $n$. When calibrated with $\chi_q^2$ quantiles, the asymptotic size of the $F$-test is inflated in the second regime, but interestingly the size of the classical $F$-test based on $F_{p-p_0, n-p}$ quantiles converges to the nominal level in both regimes. This perhaps explains the relative robustness of the $F$-test in some high-dimensional settings where the leverage condition is satisfied. However, when this crucial property is violated (e.g.~in ANOVA with unbalanced groups), the classical $F$-test is asymptotically invalid and requires correction \citep{calhoun2011hypothesis}. Additional weighting is also needed when the errors are heteroscedastic \citep[e.g.][]{akritas2004heteroscedastic}.

In summary, previous works show that the $F$-test can be either robust or highly non-robust depending on many factors including characteristics of the covariate matrix $X$, nature of the non-Gaussian error distributions, and sample size relative to the number of regressors. In particular, it is the interaction between design geometry and the error distribution that governs the behaviour of the $F$-statistic under the null.

In this paper, we introduce a novel and complementary approach to studying the robustness of the regression $F$-test, focusing on finite-sample sensitivity analysis via H\"older continuity bounds. Our main Theorems~\ref{thm:dK-W1-A} and~\ref{thm:doubly-robust-A} apply regardless of the relative magnitudes of $n,p,p_0$; together, they explicitly bound the potential Type~I error inflation in terms of the minimum of two interpretable quantities, which measure deviations from rotational invariance of the covariate and error distributions respectively. This framework accommodates a flexible range of data-generating mechanisms, with only mild assumptions imposed on the structure of the design matrix and error vector. Nevertheless, our bounds are not especially pessimistic, in the sense that the upper and lower bounds in Theorem~\ref{thm:dK-W1-A} and Proposition~\ref{prop:dK-W1-lb}, as well as those on the local Kolmogorov distance in Section~\ref{sec:local-dK}, continue to match up to multiplicative factors depending only on $n$, $p$ and $p_0$ even when the pairs $(X_1,Y_1),\ldots,(X_n,Y_n)$ are exchangeable.

\bibliographystyle{apalike}
\bibliography{bib}

\clearpage

\setcounter{section}{0}
\setcounter{equation}{0}
\setcounter{theorem}{0}
\def\theequation{S\arabic{equation}}
\def\thesection{S\arabic{section}}
\def\thetheorem{S\arabic{theorem}}
\def\thefigure{S\arabic{figure}}

\section{General results and proofs}

The results in the main text are special cases of their counterparts below, whose labels are prefixed with the letter `S'. We study a more general setting where analogously to the Frisch--Waugh--Lovell theorem (see \citet{frisch1933partial}, \citet{lovell1963seasonal}, \citet{Lovell01012008} or \citet[][Exercise~2.11]{samworth24modern}), it is convenient to orthogonalise covariates with respect to certain other covariates in the null model that play no role in the distributional properties of the $F$-statistic. For instance, we may have deterministic covariates such as an intercept term, and failure to orthogonalise renders some of our earlier results (e.g.~Theorem~\ref{thm:doubly-robust-A}) suboptimal. To this end, given $p' \in \{0,1,\dotsc,p_0\}$, we can let $P' := X_{[p']}(X_{[p']}^\top X_{[p']})^{-1}X_{[p']}^\top$ and $\varepsilon' := (I_n - P')\varepsilon$. In particular, in the case where~$p' = 1$ and $X_1$ represents an intercept term, premultiplication by $I_n - P'$ has the effect of column centring. Henceforth, suppose that $\Pr(\varepsilon' = 0) = 0$, and write $\tilde{n} := n - p'$, $\tilde{p} := p - p'$ and $\tilde{p}_0 := p_0 - p'$.

Since $(I_n - P)P' = (P - P_0)P' = 0$, it follows that under $H_0$, we have
\[
\mathsf{F} = \frac{\frac{1}{p - p_0}\norm{(P - P_0)\varepsilon}_2^2}{\frac{1}{n - p}\norm{(I_n - P)\varepsilon}_2^2} = \frac{\frac{1}{p - p_0}\norm{(P - P_0)\varepsilon'}_2^2}{\frac{1}{n - p}\norm{(I_n - P)\varepsilon'}_2^2}= \frac{T_2/(p - p_0)}{T_3/(n - p)},
\]
where 
\[
T = (T_1,T_2,T_3) := \frac{\bigl(\norm{(P_0 - P')\varepsilon'}_2^2,\; \norm{(P - P_0)\varepsilon'}_2^2,\;\norm{(I_n - P)\varepsilon'}_2^2\bigr)}{\norm{\varepsilon'}_2^2}.
\]

In our first result, we require a notion of rotational invariance that applies to matrices $A$ whose column span is $\mathcal{V} := \mathrm{Im}(I_n - P')$. For $Q_0 \in \mathcal{O}_{n \times n}$, the matrix $Q_0(I_n - P')Q_0^\top$ represents the orthogonal projection onto $Q_0(\mathcal{V})  := \{Q_0 v : v \in \mathcal{V}\}$, so
\begin{equation}
\label{eq:ortho-subspace}
\mathcal{O}' \equiv \mathcal{O}_{P'}' := \{Q_0 \in \mathcal{O}_{n \times n} : Q_0 P'Q_0^\top = P'\} = \{Q_0 \in \mathcal{O}_{n \times n} : Q_0(\mathcal{V}) = \mathcal{V}\}.
\end{equation}
We say that a random $n \times d$ matrix $A$ with $d \in [\tilde{p}]$ has a \emph{rotationally invariant distribution with respect to $\mathcal{V}$} if, given $P'$, we have $Q_0 A \eqd A$ for every $Q_0 \in \mathcal{O}_{P'}'$. There is a unique distribution on $\mathcal{V}_{\tilde{p}} := \{A \in \mathcal{O}_{n \times \tilde{p}} : \mathrm{Im}(A) \subseteq \mathcal{V}\}$ that is rotationally invariant with respect to $\mathcal{V}$ \citep[e.g.][Theorem~6.3]{eaton2007multivariate}.

\begin{proposition}
\label{prop:spherical-sym}
\begin{enumerate}[label=(\alph*)]
\item Suppose that
\[
T \sim \mathrm{Dirichlet}\Bigl(\frac{\tilde{p}_0}{2},\frac{p - p_0}{2},\frac{n - p}{2}\Bigr).
\]
Then
\[
\frac{\norm{(P - P_0)\varepsilon}_2^2}{\norm{(I_n - P_0)\varepsilon}_2^2} = \frac{T_2}{T_2 + T_3} \sim \mathrm{Beta}\Bigl(\frac{p - p_0}{2},\frac{n - p}{2}\Bigr),
\]
which is equivalent to  $\mathsf{F} \sim F_{p - p_0, n - p}$.
\item Suppose that, conditionally on $P'$, the matrix $X' := (I_n - P')X_{(p':p]}$ and vector~$\varepsilon' = (I_n - P')\varepsilon$ are independent, and that at least one of them has a rotationally invariant distribution with respect to $\mathcal{V}$. Then $\mathsf{F} \sim F_{p - p_0, n - p}$ under~$H_0$. In particular, this latter condition holds if conditionally on $P'$, the rows of $X_{(p':p]}$ are independent $N_{\tilde{p}}(0,\Sigma_X)$ random vectors for some positive definite $\Sigma_X \in \R^{\tilde{p} \times \tilde{p}}$.
\end{enumerate}
\end{proposition}

The hypotheses of \textit{(b)} are satisfied if $X_{(p':p]}$ and~$\varepsilon$ are conditionally independent and at least one has a rotationally invariant distribution (on $\R^{n \times \tilde{p}}$ and $\R^n$ respectively) given~$P'$. This follows readily from the above definition of rotational invariance with respect to~$\mathcal{V}$.

\begin{proof}[Proof of Proposition~\ref{prop:spherical-sym}]
\emph{(a)} If $T \sim \mathrm{Dirichlet}\bigl(\frac{\tilde{p}_0}{2},\frac{p - p_0}{2},\frac{n - p}{2}\bigr)$, then letting $S_1 \sim \chi_{\tilde{p}_0}^2$, $S_2 \sim \chi_{p - p_0}^2$ and $S_3 \sim \chi_{n - p}^2$ be independent, we have
\[
T = (T_1,T_2,T_3) \eqd \Bigl(\frac{S_1}{S_1 + S_2 + S_3}, \frac{S_2}{S_1 + S_2 + S_3}, \frac{S_3}{S_1 + S_2 + S_3}\Bigr)
\]
\citep[e.g.][Proposition~G.3]{ghosal2017fundamentals}. Then
\begin{equation}
\label{eq:Fstar}
\mathsf{F} = \frac{T_2/(p - p_0)}{T_3/(n - p)} \eqd \frac{S_2/(p - p_0)}{S_3/(n - p)} \sim F_{p - p_0, n - p}.
\end{equation}
We have $\mathsf{F} \sim F_{p - p_0, n - p}$ if and only if
\begin{align*}
\frac{\norm{(P - P_0)\varepsilon}_2^2}{\norm{(I_n - P_0)\varepsilon}_2^2} = \frac{T_2}{T_2 + T_3} = \frac{(p - p_0)\mathsf{F}}{n - p + (p - p_0)\mathsf{F}} &\eqd \frac{S_2}{S_2 + S_3} \sim \mathrm{Beta}\Bigl(\frac{p - p_0}{2},\frac{n - p}{2}\Bigr),
\end{align*}
noting that $t \mapsto (p - p_0)t/\bigl(n - p + (p - p_0)t\bigr)$ is a bijective function from $(0,\infty)$ to $(0,1)$. 

\medskip
\noindent \textit{(b)} Define $P_1 := P_0 - P'$, $P_2 := P - P_0$ and $P_3 := I_n - P$, which we claim are projection matrices determined by $X' := (I_n - P')X_{(p':p]}$. To see this, first observe that $\mathrm{Im}(P') = \mathrm{Im}(X_{[p']}) \subseteq \mathrm{Im}(X) = \mathrm{Im}(P)$, so $P - P'$ is idempotent and since it is also symmetric, it is a projection matrix. In addition, for $v \in \R^n$, we can write $Pv = X_{[p']}u + X_{(p':p]}w$ for some $u \in \R^{p'}$ and $w \in \R^{\tilde{p}}$, so
\[
(P - P')v = (P - P')Pv = (P - P')X_{(p':p]}w = (I_n - P')X_{(p':p]}w.
\]
This shows that $\mathrm{Im}(P - P') \subseteq \mathrm{Im}(X')$. On the other hand, for $w \in \R^{\tilde{p}}$, we have
\[
(P - P')(I_n - P')X_{(p':p]}w = P(I_n - P')X_{(p':p]}w = (I_n - P')X_{(p':p]}w,
\]
so $\mathrm{Im}(X') \subseteq \mathrm{Im}(P - P')$, and hence $P_1 + P_2 = P - P'$ represents the orthogonal projection onto $\mathrm{Im}(X')$. Similarly, $P_1 = P_0 - P'$ represents the orthogonal projection onto the image of $(I_n - P')X_{(p':p_0]} = X'_{[\tilde{p}_0]}$, and $P_3 = I_n - P' - (P_1 + P_2)$, so $P_1,P_2,P_3$ are indeed functions of $X'$.

Suppose first that given $P'$, we have that $Q_0\varepsilon' \eqd \varepsilon'$ for all $Q_0 \in \mathcal{O}'$, in which case $\xi := \varepsilon'/\norm{\varepsilon'}_2 \sim \Unif(\mathcal{V} \cap \mathcal{S}^{n-1})$ given $P'$. Then letting $Z \sim N_n(0,I_n)$ be independent of $X'$, we have given $P'$ that
\[
\xi \eqd \frac{(I_n - P')Z}{\norm{(I_n - P')Z}_2} =: \xi^*,
\]
and $\xi,\xi^*$ are conditionally independent of $X'$ (and hence $(P_1,P_2,P_3)$ by the previous paragraph) given $P'$. Since $P_j P_k = 0$ for all distinct $j,k \in \{1,2,3\}$, the random vectors $P_1 Z$, $P_2 Z$ and $P_3 Z$ are independent given $X'$. Now by Cochran's theorem, $\norm{P_j Z}_2^2 \sim \chi_{\tr(P_j)}^2$ independently given $X'$ \citep[][Exercise~2.10]{samworth24modern}. Since $P_j(I_n - P') = P_j$ for all $j$, we have
\begin{align}
T = (\norm{P_1\xi}_2^2, \norm{P_2\xi}_2^2, \norm{P_3\xi}_2^2) &\eqd
\frac{(\norm{P_1 Z}_2^2, \norm{P_2 Z}_2^2, \norm{P_3 Z}_2^2)}{\norm{P_1 Z}_2^2 + \norm{P_2 Z}_2^2 + \norm{P_3 Z}_2^2} \notag \\[3pt]
\label{eq:spherical-sym-dirichlet}
&\sim \mathrm{Dirichlet}\Bigl(\frac{\tilde{p}_0}{2},\frac{p - p_0}{2},\frac{n - p}{2}\Bigr)
\end{align}
conditionally and hence unconditionally on $P'$, so $\mathsf{F} \sim F_{p - p_0, n - p}$ under $H_0$ by~\eqref{eq:Fstar}, as required.

\medskip

Now suppose instead that $X' = (I_n - P')X_{(p':p]}$ satisfies $Q_0 X' \eqd X'$ given $P'$ for every $Q_0 \in \mathcal{O}'$. Then $Q_0(P_0 - P')Q_0^\top$ and $Q_0(P - P')Q_0^\top$ represent the orthogonal projections onto the column spaces of $Q_0 X_{[\tilde{p}_0]}'$ and $Q_0 X'$ respectively, so
\begin{equation}
\label{eq:ortho-conjugation}
(P', P_0 - P', P - P') \eqd \bigl(Q_0 P'Q_0^\top,\,Q_0(P_0 - P')Q_0^\top,\,Q_0(P - P')Q_0^\top\bigr)
\end{equation}
and hence $(P_1,P_2,P_3) \eqd (Q_0 P_1 Q_0^\top, Q_0 P_2 Q_0^\top, Q_0 P_3 Q_0^\top)$ given $P'$. Now for any fixed $v,w \in \mathcal{V} \cap \mathcal{S}^{n-1}$, there exists $Q_0 \in \mathcal{O}'$ such that $Q_0^\top v = w$, so
\begin{align*}
(\norm{P_1 v}_2^2, \norm{P_2 v}_2^2, \norm{P_3 v}_2^2) &\eqd (\norm{Q_0 P_1 Q_0^\top v}_2^2,\norm{Q_0 P_2 Q_0^\top v}_2^2, \norm{Q_0 P_3 Q_0^\top v}_2^2) \\
&= (\norm{P_1 w}_2^2, \norm{P_2 w}_2^2, \norm{P_3 w}_2^2)
\end{align*}
given $P'$. In other words, the conditional distribution of the left-hand side does not depend on~$v$. By assumption, $\xi$ is independent of $X'$ given $P'$ and takes values in $\mathcal{V} \cap \mathcal{S}^{n-1}$, so
\[
T = (\norm{P_1\xi}_2^2, \norm{P_2\xi}_2^2, \norm{P_3\xi}_2^2) \eqd (\norm{P_1\xi^*}_2^2, \norm{P_2\xi^*}_2^2, \norm{P_3\xi^*}_2^2) =: T^*
\]
given $P'$, where $\xi^* \sim \Unif(\mathcal{V} \cap \mathcal{S}^{n-1})$ is also independent of $X'$ given $P'$. By the case where~$\varepsilon'$ has a rotationally invariant distribution studied above, $T^* \sim \mathrm{Dirichlet}\bigl(\frac{\tilde{p}_0}{2},\frac{p - p_0}{2},\frac{n - p}{2}\bigr)$, so the conclusion also holds in this case.

In particular, if $X_{(p':p]}$ has independent $N_{\tilde{p}}(0,\Sigma_X)$ rows, then letting $X^*$ be an $n \times \tilde{p}$ matrix with independent $N(0,1)$ entries given $P'$, we have $(I_n - P')X^* \eqd (I_n - P')Q_0 X^* = Q_0(I_n - P')X^*$ for every $Q_0 \in \mathcal{O}'$. Therefore,
\[
(I_n - P')X_{(p':p]} \eqd (I_n - P')X^*\Sigma_X^{1/2} \eqd Q_0(I_n - P')X^*\Sigma_X^{1/2} \eqd Q_0(I_n - P')X_{(p':p]}
\]
given $P'$ for every fixed $Q_0 \in \mathcal{O}'$, so $\mathsf{F} \sim F_{p-p_0, n-p}$ by the previous paragraph.
\end{proof}

\subsection{Proofs of upper bounds on the Kolmogorov distance in Section~\ref{sec:upper-bds}}

We adopt the same notation as before the statement of Theorem~\ref{thm:dK-W1-A} in the main text, except we instead define
\[
B_j := \mathrm{B}\Bigl(\frac{\tilde{p}_0}{2}, \frac{n - p_0 - j}{2}\Bigr) = \mathrm{B}\Bigl(\frac{\tilde{p}_0}{2}, \frac{\tilde{n} - \tilde{p}_0 - j}{2}\Bigr) ,
\]
for $j \in \{0,1,2\}$, when $\tilde{p}_0$ and $n - p_0 - j$ are positive integers. It will also be convenient to define
\[
D_1 \equiv D_1(\tilde{n},\tilde{p}_0) := \frac{B_1}{B_0},
\]
with the interpretation that $D_1(\tilde{n},0) := \lim_{x \searrow 0} D_1(\tilde{n},x) = 1$. 

\begin{theorem}
\label{thm:dK-W1}
Let $T^* \sim \mathrm{Dirichlet}\bigl(\frac{\tilde{p}_0}{2},\frac{p - p_0}{2},\frac{n - p}{2}\bigr)$.  Under $H_0$, we have
\begin{equation}
\label{eq:dK-W1-tilde}
\dK(\mathsf{F},\mathsf{F}^*) \leq C(\tilde{n},\tilde{p},\tilde{p}_0) \cdot
\begin{cases}
\W_1(\tilde{n}T,\tilde{n}T^*)^{1/3} \;\;&\text{if }k = 1 \\
\W_1(\tilde{n}T,\tilde{n}T^*)^{1/2} \;\;&\text{if }k \geq 2,
\end{cases}
\end{equation}
where $k := \min(p - p_0, n - p)$ and
\begin{align*}
C(\tilde{n},\tilde{p},\tilde{p}_0) &:= 
\begin{cases}
3 \cdot 2^{-1/2}\Bigl(\dfrac{D_1}{\tilde{n}^{1/2}B}\Bigr)^{2/3} \leq \dfrac{3\pi^{1/3}}{2^{5/6}}  \;\; &\text{if $k = 1$} \\[8pt]
2^{5/4}\Bigl(\dfrac{(\tilde{n} - 2)M}{\tilde{n}(\tilde{n} - \tilde{p}_0 - 2)}\Bigr)^{1/2} < 2^{3/4} \;\; &\text{if $k \geq 2$.}
\end{cases}
\end{align*}
\end{theorem}

\begin{proof}[Proof of Theorem~\ref{thm:dK-W1}]
Given any coupling of $T$ and $T^*$, define
\[
\Delta = (\Delta_1,\Delta_2,\Delta_3) := T - T^* \quad\text{and}\quad V^* := \frac{T_2^*}{1 - T_1^*}.
\]
Fix $t > 0$ and let
\[
s := \frac{(p - p_0)t}{n - p + (p - p_0)t} \in (0,1).
\]
Since $T_1 + T_2 + T_3 = T_1^* + T_2^* + T_3^* = 1$, we have
\[
\Pr(\mathsf{F}^* \leq t) = \Pr\Bigl(\frac{T_2^*}{T_3^*} \leq \frac{(p - p_0)t}{n - p}\Bigr) = \Pr\Bigl(\frac{T_2^*}{1 - T_1^*} \leq \frac{(p - p_0)t}{n - p + (p - p_0)t}\Bigr) = \Pr(V^* \leq s),
\]
and
\begin{align*}
\Pr(\mathsf{F} \leq t) &= \Pr\Bigl(\frac{T_2}{T_3} \leq \frac{(p - p_0)t}{n - p}\Bigr) = \Pr\Bigl(\frac{T_2}{1 - T_1} \leq \frac{(p - p_0)t}{n - p + (p - p_0)t}\Bigr) \\[3pt]
&= \Pr\bigl(T_2^* + \Delta_2 \leq s(1 - T_1^* - \Delta_1)\bigr) = \Pr\Bigl(V^* \leq s - \frac{s\Delta_1 + \Delta_2}{1 - T_1^*}\Bigr).
\end{align*}
Now, for any $\delta > 0$, we can write
\begin{align*}
\Pr\Bigl(V^* \leq s - \frac{s\Delta_1 + \Delta_2}{1 - T_1^*}\Bigr) &\leq \Pr\Bigl(V^* \leq s + \frac{\delta}{1 - T_1^*}\Bigr) + \Pr(s\Delta_1 + \Delta_2 < -\delta)
\end{align*}
and
\begin{align*}
\Pr\Bigl(V^* \leq s - \frac{\delta}{1 - T_1^*}\Bigr) &\leq \Pr\Bigl(V^* \leq s - \frac{s\Delta_1 + \Delta_2}{1 - T_1^*}\Bigr) + \Pr(s\Delta_1 + \Delta_2 > \delta).
\end{align*}
Thus, since $|s\Delta_1 + \Delta_2| \leq 2^{1/2}\norm{\Delta}_2$, we have
\begin{align}
\dK(\mathsf{F},\mathsf{F}^*) &= \sup_{t > 0}\;\bigl|\Pr(\mathsf{F} \leq t) - \Pr(\mathsf{F}^* \leq t)\bigr| \notag \\
\label{eq:dK-Delta}
&\leq \max_{\epsilon \in \{-1,1\}}\sup_{s \in (0,1)}\:\biggl|\Pr\Bigl(V^* \leq s + \frac{\epsilon\delta}{1 - T_1^*}\Bigr) - \Pr(V^* \leq s)\biggr| + \Pr(2^{1/2}\norm{\Delta}_2 \geq \delta).
\end{align}
By~\citet[Proposition~G.3]{ghosal2017fundamentals},
\begin{equation}
\label{eq:dirichlet-indep}
T_1^* \sim \mathrm{Beta}\Bigl(\frac{\tilde{p}_0}{2}, \frac{n - p_0}{2}\Bigr) \quad\text{and}\quad V^* = \frac{T_2^*}{1 - T_1^*} \sim \mathrm{Beta}\Bigl(\frac{p - p_0}{2}, \frac{n - p}{2}\Bigr)
\end{equation}
are independent when $\tilde{p}_0 \geq 1$. On the other hand, when $\tilde{p}_0 = 0$, we have $T_1^* = 0$ and $V^* = T_2^* \sim \mathrm{Beta}\bigl(\frac{p - p_0}{2},\frac{n - p}{2}\bigr)$, so these random variables remain independent.

Recall the $\mathrm{Beta}(a,b)$ density $g_{a,b}$ from~\eqref{Eq:BetaDensity}. If $a < 1$, then $\lim_{x \searrow 0} g_{a,b}(x) = \infty$. In this case, $g_{a,b}$ is decreasing on $(0,1)$ when $b \geq 1$, and otherwise decreasing on $\bigl(0,\frac{a-1}{a+b-2}\bigr)$ and symmetric about $1/2$ if $b = a$. Thus for $\eta > 0$, we have
\begin{align}
\label{eq:W-prob}
&\sup_{v \in \R}\,\bigl\{\Pr(V^* \leq v + \eta) - \Pr(V^* \leq v)\bigr\} \notag \\
&\hspace{2cm}
\begin{cases}
\,\leq \displaystyle\int_0^{\eta \wedge 1} \frac{x^{-1/2}(1 - x)^{-1/2}}{B}\,dx
\leq \frac{\pi\eta^{1/2}}{B}
\quad &\text{if }k = 1 \\[6pt]
\,\leq M\eta
\quad &\text{if }k \geq 2.
\end{cases}
\end{align}

By the tower property of expectation and the independence of $T_1^*$ and $V^*$, we deduce from~\eqref{eq:W-prob} that for every $\delta > 0$, and when $\tilde{p}_0 \geq 1$,
\begin{align}
\max_{\epsilon \in \{-1,1\}} &\sup_{s \in (0,1)}\:\biggl|\Pr\Bigl(V^* \leq s + \frac{\epsilon\delta}{1 - T_1^*}\Bigr) - \Pr(V^* \leq s)\biggr| \notag \\
&\leq \int_0^1 \sup_{v \in \R}\,\Bigl\{\Pr\Bigl(V^* \leq v + \frac{\delta}{1 - w}\Bigr) - \Pr(V^* \leq v)\Bigr\} \cdot \frac{w^{\frac{\tilde{p}_0}{2} - 1}(1 - w)^{\frac{n - p_0}{2} - 1}}{B_0}\,dw \notag \\[3pt]
&\leq
\begin{cases}
\,\displaystyle\frac{\pi\delta^{1/2}}{B_0 B} \int_0^1 w^{\frac{\tilde{p}_0}{2} - 1}(1 - w)^{\frac{n - p_0 - 1}{2} - 1}\,dw = \frac{\pi D_1}{B} \cdot \delta^{1/2} \quad &\text{if }k = 1 \\[12pt]
\,\displaystyle\frac{M\delta}{B_0} \int_0^1 w^{\frac{\tilde{p}_0}{2} - 1}(1 - w)^{\frac{n - p_0 - 2}{2} - 1}\,dw = \frac{MB_2}{B_0} \cdot \delta = \frac{(\tilde{n} - 2)M}{n - p_0 - 2} \cdot \delta \quad &\text{if }k \geq 2.
\end{cases}
\notag
\end{align}
Observe that by our definition of $D_1$ when $\tilde{p}_0 = 0$, these final bounds hold even when $\tilde{p}_0 = 0$. By applying Markov's inequality and optimising over $\delta > 0$ in~\eqref{eq:dK-Delta}, we deduce when $k = 1$ that
\begin{align*}
\dK(\mathsf{F},\mathsf{F}^*) &\leq \inf_{\delta > 0}\:\biggl(\frac{\pi D_1}{B} \cdot \delta^{1/2} + \frac{2^{1/2}\,\E(\norm{\Delta}_2)}{\delta}\biggr) = 3 \cdot 2^{-1/2}\Bigl(\frac{\pi D_1}{B}\Bigr)^{2/3} \bigl\{\E(\norm{T - T^*}_2)\bigr\}^{1/3},
\end{align*}
and when $k \geq 2$ that
\begin{align*}
\dK(\mathsf{F},\mathsf{F}^*) &\leq \inf_{\delta > 0}\:\biggl(\frac{(\tilde{n} - 2)M}{n - p_0 - 2} \cdot \delta + \frac{2^{1/2}\,\E(\norm{\Delta}_2)}{\delta}\biggr) = 2^{5/4}\,\biggl(\frac{(\tilde{n} - 2)M}{n - p_0 - 2}\biggr)^{1/2} \bigl\{\E(\norm{T - T^*}_2)\bigr\}^{1/2}.
\end{align*}
These bounds hold for any coupling of $T$ and $T^*$, so on the right-hand side, we may take an infimum over all couplings to obtain the desired conclusion~\eqref{eq:dK-W1-tilde}.

It remains to bound $D_1/B$ and $M$. When $k = 1$, we have
\begin{align*}
\frac{D_1}{B}
=\frac{\mathrm{B}\bigl(\frac{\tilde{p}_0}{2}, \frac{n - p_0 - 1}{2}\bigr)}
{\mathrm{B}\bigl(\frac{\tilde{p}_0}{2}, \frac{n - p_0}{2}\bigr)\,
\mathrm{B}\bigl(\frac{p - p_0}{2}, \frac{n - p}{2}\bigr)} = \frac{\Gamma\bigl(\frac{n - p_0}{2}\bigr)}{\Gamma\bigl(\frac{p - p_0}{2}\bigr)\Gamma\bigl(\frac{n - p}{2}\bigr)} \cdot \frac{\Gamma\bigl(\frac{n - p_0 - 1}{2}\bigr)}{\Gamma\bigl(\frac{n - p_0}{2}\bigr)} \cdot \frac{\Gamma\bigl(\frac{\tilde{n}}{2}\bigr)}{\Gamma\bigl(\frac{\tilde{n} - 1}{2}\bigr)} = \frac{1}{\pi^{1/2}} \frac{\Gamma\bigl(\frac{\tilde{n}}{2}\bigr)}{\Gamma\bigl(\frac{\tilde{n} - 1}{2}\bigr)}.
\end{align*}
Thus, by Wendel's inequality \citep{wendel1948note},
\begin{equation}
\label{eq:wendel-bound}
\frac{\tilde{n} - 1}{(2\pi \tilde{n})^{1/2} }\leq \frac{D_1}{B}\leq \Bigl(\frac{\tilde{n} - 1}{2\pi}\Bigr)^{1/2},
\end{equation}
so if $k = 1$, then
\begin{align*}
\dK(\mathsf{F},\mathsf{F}^*) \leq \frac{3\pi^{1/3}}{2^{5/6}} \W_1(\tilde{n}T,\tilde{n}T^*)^{1/3}.
\end{align*}
When $k \geq 2$, define $a := (p - p_0)/2$ and $b := (n - p)/2$. Consider $c := a + b = (n - p_0)/2$ as fixed and define $\ell(a) := \log\mathrm{M}(a,b) = \log\mathrm{M}(a,c-a)$.  If $a + b > 2$, then
\begin{equation}
\label{eq:max-beta}
M = \mathrm{M}(a,b) = g_{a,b}\Bigl(\frac{a - 1}{a + b - 2}\Bigr) = \frac{(a-1)^{a-1}(b-1)^{b-1}\Gamma(a+b)}{(a+b-2)^{a+b-2}\Gamma(a)\Gamma(b)}.
\end{equation}
Writing $\psi = (\log\Gamma)''$ for the trigamma function, we have for $x > 1$ that
\[
\psi(x) = \sum_{i=0}^\infty \frac{1}{(x+i)^2} < \int_{-1}^\infty \frac{1}{(x+z)^2}\,dz = \frac{1}{x-1}.
\]
Hence, for $a \in (1,c - 1)$, we have
\begin{align*}
\ell''(a) &= -\psi(a)-\psi(c-a)+\frac{1}{a-1}+\frac{1}{c-a-1}\\
&>-\frac{1}{a-1}-\frac{1}{c-a-1}+\frac{1}{a-1}+\frac{1}{c-a-1}=0.
\end{align*}
Since $\ell$ is also continuous on $[1,c-1]$, it is convex on $[1,c-1]$, so
\[
\max_{a \in [1,c-1]} \mathrm{M}(a,c-a) = \max\bigl\{\mathrm{M}(1,c-1),\mathrm{M}(c-1,1)\bigr\} = c-1 = \frac{n - p_0 - 2}{2}.
\]
On the other hand, if $a + b = 2$, then $M = \mathrm{M}(1,1) = 1 = (n - p_0 - 2)/2$ as before. Thus, when $k \geq 2$, we have
\[
\dK(\mathsf{F},\mathsf{F}^*) \leq 2^{3/4}\Bigl(\frac{\tilde{n} - 2}{\tilde{n}}\Bigr)^{1/2}\W_1(\tilde{n}T,\tilde{n}T^*)^{1/2} < 2^{3/4}\W_1(\tilde{n}T,\tilde{n}T^*)^{1/2},
\]
as required.
\end{proof}

\begin{lemma}
\label{lem:dK-W1-constant}
Let $C(\tilde{n},\tilde{p},\tilde{p}_0) > 0$ be taken from Theorem~\ref{thm:dK-W1}. When $k = 1$,
\[
\frac{3\pi^{1/3}}{2^{3/2}} \leq C(\tilde{n},\tilde{p},\tilde{p}_0) < \frac{3\pi^{1/3}}{2^{5/6}}.
\]
When $k \geq 2$, 
\[
\frac{2^{1/4}e^{-25/12}}{\pi^{1/4}}\Bigl(\frac{n - p_0}{(p - p_0)(n - p)}\Bigr)^{1/4} < C(\tilde{n},\tilde{p},\tilde{p}_0) \leq \frac{2^{5/4}}{\pi^{1/4}}\Bigl(\frac{n - p_0}{(p - p_0)(n - p)}\Bigr)^{1/4}.
\]
\end{lemma}

\begin{proof}
When $k = 1$, it follows from~\eqref{eq:wendel-bound} that
\[
\frac{1}{2\pi^{1/3}} \leq \Bigl(\frac{\tilde{n} - 1}{(2\pi)^{1/2}\,\tilde{n}}\Bigr)^{2/3} \leq \frac{C(\tilde{n},\tilde{p},\tilde{p}_0)}{3 \cdot 2^{-1/2}\pi^{2/3}} = \Bigl(\frac{D_1}{\tilde{n}^{1/2}B}\Bigr)^{2/3} \leq \Bigl(\frac{\tilde{n} - 1}{2\pi \tilde{n}}\Bigr)^{1/3} < \frac{1}{(2\pi)^{1/3}}
\]
since $\tilde{n} \geq 2$. Now consider $k \geq 2$. For $x > 0$, \citet[Lemma~10]{Duembgen2021} states that
\begin{equation}
\label{eq:log-gamma}
\log\Gamma(x) = \frac{1}{2}\log(2\pi) - x + \Bigl(x - \frac{1}{2}\Bigr)\log x + R(x),
\end{equation}
where $1/(12x + 1) < R(x) < 1/(12x)$.  As in the proof of Theorem~\ref{thm:dK-W1}, let $a = (p - p_0)/2 \geq 1$, $b = (n - p)/2 \geq 1$ and $c = a + b = (n - p_0)/2 \geq 2$.  With a view to approximating $M = \mathrm{M}(a,b)$ in~\eqref{eq:max-beta}, we have by~\eqref{eq:log-gamma} that
\begin{align*}
\log\frac{(a - 1)^{a - 1}}{\Gamma(a)} &= (a - 1)\log(a - 1) - \Bigl(a - \frac{1}{2}\Bigr)\log a + a - \frac{1}{2}\log(2\pi) - R(a) \\
&= -\frac{\log a}{2} + (a - 1)\log\Bigl(1 - \frac{1}{a}\Bigr) + a - \frac{1}{2}\log(2\pi) - R(a)
\end{align*}
when $a > 1$, where $-1 < (a - 1)\log(1 - 1/a) < 0$. On the other hand, the left-hand side is equal to 0 when $a = 1$, so 
\[
- 1 - \frac{1}{12a} \leq \log\frac{(a - 1)^{a - 1}}{\Gamma(a)} + \frac{\log a}{2} - a + \frac{1}{2}\log(2\pi) \leq -\frac{1}{12a + 1}
\]
in all cases. Similarly,
\[
- 1 - \frac{1}{12b} \leq \log\frac{(b - 1)^{b - 1}}{\Gamma(b)} + \frac{\log b}{2} - b + \frac{1}{2}\log(2\pi) \leq -\frac{1}{12b + 1}.
\]
Moreover, 
\begin{align*}
&\log\frac{\Gamma(c)}{(c - 2)^{c - 2}(n - p_0 - 2)} \\[3pt]
&\hspace{1.5cm}= \frac{1}{2}\log(2\pi) - c + \Bigl(c - \frac{1}{2}\Bigr)\log c + R(c) - (c - 2)\log(c - 2) - \log(2c - 2) \\[3pt]
&\hspace{1.5cm}= \frac{\log c}{2} - c + (c - 2)\log\Bigl(1 - \frac{2}{c}\Bigr) - \log\Bigl(1 - \frac{1}{c}\Bigr) + \frac{1}{2}\log\frac{\pi}{2} + R(c),
\end{align*}
where $-2 < (c - 2)\log(1 - 2/c) - \log(1 - 1/c) \leq \log 2$. Therefore, by~\eqref{eq:max-beta} and the bounds above,
\[
\log\frac{M}{n - p_0 - 2} \leq \frac{1}{2}\log\frac{c}{ab} - \frac{1}{2}\log(2\pi) + \frac{1}{12c} - \frac{1}{12a + 1} - \frac{1}{12b + 1} \leq \frac{1}{2}\log\frac{c}{ab} - \frac{1}{2}\log(2\pi)
\]
and
\begin{align*}
\log\frac{M}{n - p_0 - 2} > \frac{1}{2}\log\frac{c}{ab} - \frac{1}{2}\log(8\pi) - 4 + \frac{1}{12c + 1} - \frac{1}{12a} - \frac{1}{12b} > \frac{1}{2}\log\frac{c}{ab} - \frac{1}{2}\log(8\pi) - \frac{25}{6}.
\end{align*}
We conclude that
\[
C(\tilde{n},\tilde{p},\tilde{p}_0) = 2^{5/4}\Bigl(\frac{(\tilde{n} - 2)M}{\tilde{n}(n - p_0 - 2)}\Bigr)^{1/2} \leq 2^{5/4}\Bigl(\frac{n - p_0}{\pi(p - p_0)(n - p)}\Bigr)^{1/4}
\]
and
\[
C(\tilde{n},\tilde{p},\tilde{p}_0) > 2^{5/4} \cdot \frac{e^{-25/12}}{2^{1/2}} \Bigl(\frac{2(n - p_0)}{8\pi(p - p_0)(n - p)}\Bigr)^{1/4} = \frac{2^{1/4}e^{-25/12}}{\pi^{1/4}}\Bigl(\frac{n - p_0}{(p - p_0)(n - p)}\Bigr)^{1/4},
\]
since $\tilde{n} \geq 2k \geq 4$, as required.
\end{proof}

Next, we establish a double robustness result that naturally extends Theorem~\ref{thm:doubly-robust-A} for a general $p' \in \{0,1,\dotsc,p_0\}$. Since the $F$-statistic $\mathsf{F}$ depends on $X$ only through the matrices $P_0$ and $P$, we first derive a bound based on the proximity between projection matrices before relating it more directly to the notion of rotational invariance featured in Proposition~\ref{prop:spherical-sym}\textit{(b)}. Formally, write $\mathcal{P}$ for the set of pairs of orthogonal projection matrices $(P_1,P_2)$ such that $\tr(P_1) = p_0 - p'$, $\tr(P_2) = p - p_0$, $P_1 P_2 = 0$, $P'P_1 = 0$, $P'P_2 = 0$; in other words, $P',P_1,P_2$ represent orthogonal projections onto mutually orthogonal subspaces, of dimensions $p'$, $\tilde{p}_0 = p_0 - p'$ and $p - p_0$ respectively. For random pairs $(P_1,P_2),(\tilde{P}_1,\tilde{P}_2)$ taking values in~$\mathcal{P}$, we can define the metric
\begin{equation}
\label{eq:sin-Theta}
d\bigl((P_1,P_2),(\tilde{P}_1,\tilde{P}_2)\bigr) := (\norm{P_1 - \tilde{P}_1}_{\mathrm{F}}^2 + \norm{P_2 - \tilde{P}_2}_{\mathrm{F}}^2)^{1/2}
\end{equation}
and the 2-Wasserstein distance $\W_2\bigl((P_1,P_2),(\tilde{P}_1,\tilde{P}_2)\bigr)$ between their distributions as the infimum of $\E\bigl\{d\bigl((P_1,P_2),(\tilde{P}_1,\tilde{P}_2)\bigr)^2\bigr\}^{1/2}$ over all couplings. In fact, recalling from e.g.~\citet{davis1970rotation} and~\citet{yu2015useful} the definition of the \emph{$\sin\theta$ distance} between matrices with orthonormal columns,\footnote{Recall that if $V,V' \in \mathcal{O}_{n \times d}$, then the vector of $d$ principal angles between their column spaces is given by $(\arccos \sigma_1,\dotsc,\arccos\sigma_d)$, where $\sigma_1 \geq \cdots \geq \sigma_d$ are the singular values of $V^\top V'$. If $\Theta(V,V')$ denotes the $d \times d$ diagonal matrix whose $j$th diagonal entry is the $j$th principal angle, then $\sin\Theta(V,V')$ is defined entrywise.} we have
\[
d^2\bigl((P_1,P_2),(\tilde{P}_1,\tilde{P}_2)\bigr) = \sum_{j=1}^2 \norm{Q_j Q_j^\top - \tilde{Q}_j\tilde{Q}_j^\top}_{\mathrm{F}}^2 = 2\sum_{j=1}^2 \norm{\sin\Theta(Q_j,\tilde{Q}_j)}_{\mathrm{F}}^2
\]
for any $Q_1,\tilde{Q}_1 \in \mathcal{O}_{n \times \tilde{p}_0}$ and $Q_2,\tilde{Q}_2 \in \mathcal{O}_{n \times \tilde{p}}$ such that $\mathrm{Im}(Q_j) = \mathrm{Im}(P_j)$ and $\mathrm{Im}(\tilde{Q}_j) = \mathrm{Im}(\tilde{P}_j)$ for $j \in \{1,2\}$. Moreover,
\[
0 \leq \norm{Q_j^\top \tilde{Q}_j - I_{\tr(P_j)}}_{\mathrm{F}}^2 = \norm{Q_j^\top \tilde{Q}_j}_{\mathrm{F}}^2 - 2\tr(Q_j^\top \tilde{Q}_j) + \tr(P_j), 
\]
so because $P_j = Q_j Q_j^\top$ and $\tilde{P}_j = \tilde{Q}_j\tilde{Q}_j^\top$, we have $\norm{Q_j}_{\mathrm{F}}^2 = \tr(P_j) = \tr(\tilde{P}_j) = \norm{\tilde{Q}_j}_{\mathrm{F}}^2 $ and hence
\begin{align*}
d\bigl((P_1,P_2),(\tilde{P}_1,\tilde{P}_2)\bigr)^2 = \sum_{j=1}^2 \norm{Q_j Q_j^\top - \tilde{Q}_j \tilde{Q}_j^\top}_{\mathrm{F}}^2 &= \sum_{j=1}^2 (\norm{Q_j}_{\mathrm{F}}^2 - 2\norm{Q_j^\top \tilde{Q}_j}_{\mathrm{F}}^2 + \norm{\tilde{Q}_j}_{\mathrm{F}}^2) \\
&\leq 2\sum_{j=1}^2 \bigl(\norm{Q_j}_{\mathrm{F}}^2 - 2\tr(Q_j^\top \tilde{Q}_j) + \norm{\tilde{Q}_j}_{\mathrm{F}}^2\bigr) \\
&= 2\bigl(\norm{Q_1 - \tilde{Q}_1}_{\mathrm{F}}^2 + \norm{Q_2 - \tilde{Q}_2}_{\mathrm{F}}^2\bigr) = 2\norm{Q - \tilde{Q}}_{\mathrm{F}}^2.
\end{align*}
Therefore, denoting by $\mathcal{Q} \equiv \mathcal{Q}(P_1,P_2)$ the set of all matrices $Q$ with the above properties, and defining $\tilde{\mathcal{Q}} \equiv \mathcal{Q}(\tilde{P}_1,\tilde{P}_2)$ analogously, we have
\begin{equation}
\label{eq:subspace-metrics}
d\bigl((P_1,P_2),(\tilde{P}_1,\tilde{P}_2)\bigr) \leq 2^{1/2}\inf_{Q \in \mathcal{Q},\,\tilde{Q} \in \tilde{\mathcal{Q}}} \norm{Q - \tilde{Q}}_{\mathrm{F}} =: 2^{1/2}d'\bigl((P_1,P_2),(\tilde{P}_1,\tilde{P}_2)\bigr).
\end{equation}
Writing $\tilde{\mathcal{O}}_{\tilde{p} \times \tilde{p}}$ for the set of
$\Lambda \in \mathcal{O}_{\tilde{p} \times \tilde{p}}$ of the form 
$\Lambda = \bigl(\begin{smallmatrix}
\Lambda_1 & 0 \\
0 & \Lambda_2
\end{smallmatrix}\bigr)$, we see that for any $Q,\tilde{Q} \in \mathcal{Q}(P_1,P_2)$, there exists $\Lambda \in \tilde{\mathcal{O}}_{\tilde{p} \times \tilde{p}}$ such that $Q = \tilde{Q}\Lambda$. The same is true of $\mathcal{Q}(\tilde{P}_1,\tilde{P}_2)$, so since $\tilde{\mathcal{O}}_{\tilde{p} \times \tilde{p}}$ is a group and the Frobenius norm is orthogonally invariant,
\[
d'\bigl((P_1,P_2),(\tilde{P}_1,\tilde{P}_2)\bigr) = \inf_{\Lambda \in \tilde{\mathcal{O}}_{\tilde{p} \times \tilde{p}}} \norm{Q\Lambda - \tilde{Q}}_{\mathrm{F}}
\]
and hence $d'$ is a metric on~$\mathcal{P}$; in fact, it is a variant of the Procrustes distance \citep[e.g.][]{schonemann1966generalized,zhu2022high} with the infimum taken over the subset $\tilde{\mathcal{O}}_{\tilde{p} \times \tilde{p}}$ of $\mathcal{O}_{\tilde{p} \times \tilde{p}}$.  We can therefore define $\W_2'\bigl((P_1,P_2),(\tilde{P}_1,\tilde{P}_2)\bigr)$ as the infimum of $\E\bigl\{d'\bigl((P_1,P_2),(\tilde{P}_1,\tilde{P}_2)\bigr)^2\bigr\}^{1/2}$ over all couplings. It follows from~\eqref{eq:subspace-metrics} that for any $\tilde{Q}$ taking values in $\tilde{\mathcal{Q}}$, we have
\begin{equation}
\label{eq:wasserstein-metrics}
\W_2\bigl((P_1,P_2),(\tilde{P}_1,\tilde{P}_2)\bigr) \leq 2^{1/2}\W_2'\bigl((P_1,P_2),(\tilde{P}_1,\tilde{P}_2)\bigr) \leq 2^{1/2}\inf_{Q \in \mathcal{Q}(P_1,P_2)} \W_2(Q,\tilde{Q})
\end{equation}
for all random pairs $(P_1,P_2),(\tilde{P}_1,\tilde{P}_2)$ taking values in $\mathcal{P}$. Recall from the paragraph containing~\eqref{eq:ortho-subspace} the definitions of $\mathcal{V}$ and $\mathcal{V}_{\tilde{p}}$.

\begin{theorem}
\label{thm:doubly-robust}
Suppose that $X$ and $\varepsilon$ are independent, with $P'$ being deterministic, and let $C_{\varepsilon},C_X \geq 1$ be such that
\begin{equation}
\label{eq:exp-proj-eval}
\norm{\E(P - P')}_{\mathrm{op}}
\leq \frac{C_X\tilde{p}}{\tilde{n}}
\quad\text{and}\quad \norm{\E(\xi\xi^\top)}_{\mathrm{op}}
\leq \frac{C_\varepsilon}{\tilde{n}},
\end{equation}
where $\xi := \varepsilon'/\norm{\varepsilon'}_2$. Let $\xi^* \sim \Unif(\mathcal{V} \cap \mathcal{S}^{n-1})$ and $Q^* = (Q_1^* \;\: Q_2^*) \sim \Unif(\mathcal{V}_{\tilde{p}})$ with $Q_1^* = Q_{[\tilde{p}_0]}^*$, and define $(P_1^*,P_2^*) := (Q_1^*{Q_1^*}^\top,Q_2^*{Q_2^*}^\top)$.
Writing
\begin{align*}
A := \min\biggl\{(C_X\tilde{p})^{1/2}\,\W_2(\xi,\xi^*),\,C_\varepsilon^{1/2}\,\W_2\biggl(\binom{P_0 - P'}{P - P_0}, \binom{P_1^*}{P_2^*}\biggr)\biggr\},
\end{align*}
we have
\[
\W_1(\tilde{n}T,\tilde{n}T^*) \leq 2^{1/2}(A^2 + 2\tilde{p}^{1/2}A).
\]
Moreover,
\[
\W_2\biggl(\binom{P_0 - P'}{P - P_0}, \binom{P_1^*}{P_2^*}\biggr) \leq 2^{1/2}\inf_{Q \in \mathcal{Q}} \W_2(Q,Q^*),
\]
where $\mathcal{Q}$ denotes the set of all random matrices $Q$ taking values in $\mathcal{V}_{\tilde{p}}$ such that $\mathrm{Im}(Q_{[\tilde{p}_0]}) = \mathrm{Im}(X_{[\tilde{p}_0]}')$ and $\mathrm{Im}(Q) = \mathrm{Im}(X')$.
\end{theorem}

\begin{proof}[Proof of Theorem~\ref{thm:doubly-robust}]
Let $\xi^* \sim \Unif(\mathcal{V} \cap \mathcal{S}^{n-1})$ be independent of $X$ with $(\xi,\xi^*)$ being an optimal coupling attaining $\W_2(\xi,\xi^*)$. Letting $P_1 := P_0 - P'$, $P_2 := P - P_0$ and $P_3 := I_n - P$, we have
\begin{align*}
T &= \bigl(\norm{P_1\xi}_2^2, \norm{P_2\xi}_2^2, \norm{P_3\xi}_2^2\bigr) \quad\text{and}\quad T^* \eqd \bigl(\norm{P_1\xi^*}_2^2, \norm{P_2\xi^*}_2^2, \norm{P_3\xi^*}_2^2\bigr) =: \tilde{T},
\end{align*}
where the latter holds by~\eqref{eq:spherical-sym-dirichlet} in the proof of  Proposition~\ref{prop:spherical-sym}\textit{(b)}. For $j \in \{1,2\}$, we have
\begin{align}
\E|T_j - \tilde{T}_j| &= \E\bigl|\norm{P_j\xi}_2^2 - \norm{P_j\xi^*}_2^2\bigr| = \E\bigl|(\xi + \xi^*)^\top P_j(\xi - \xi^*)\bigr| \notag \\
&\leq \E\bigl(\norm{P_j(\xi - \xi^*)}_2^2\bigr) + 2\,\E\bigl|(\xi - \xi^*)^\top P_j^\top P_j\xi^*\bigr| \notag \\
\label{eq:T-eps-1}
&\leq \E\bigl(\norm{P_j(\xi - \xi^*)}_2^2\bigr) + 2\,\E(\norm{P_j\xi^*}_2^2)^{1/2} \cdot \E\bigl(\norm{P_j(\xi - \xi^*)}_2^2\bigr)^{1/2}.
\end{align}
Since $T^* \sim \mathrm{Dirichlet}\bigl(\frac{\tilde{p}_0}{2},\frac{p - p_0}{2},\frac{n - p}{2}\bigr)$ (with $T^* = (0,T_2^*,T_3^*)$ and $(T_2^*,T_3^*) \sim \mathrm{Beta}\bigl(\frac{p - p_0}{2},\frac{n - p}{2}\bigr)$ when $\tilde{p}_0 = 0$), we have $\E(\norm{P_j\xi^*}_2^2) = \E(T_j^*) = \tr(P_j)/\tilde{n}$. Moreover, $(\xi,\xi^*)$ is independent of $X$ and $P_1 + P_2 = P - P'$, so
\begin{align}
\label{eq:T-eps-2}
\E\bigl(\norm{P_1(\xi - \xi^*)}_2^2 + \norm{P_2(\xi - \xi^*)}_2^2 \bigr) &= \E\bigl\{(\xi - \xi^*)^\top(P_1 + P_2)(\xi - \xi^*)\bigr\} \\
&\leq \norm{\E(P - P')}_{\mathrm{op}}\,\E(\norm{\xi - \xi^*}_2^2) \leq \frac{C_X\tilde{p}}{\tilde{n}}\W_2^2(\xi,\xi^*), \notag
\end{align}
where the final inequality holds by the first assumption in~\eqref{eq:exp-proj-eval}. Since $\sum_{j=1}^3 T_j = \sum_{j=1}^3 \tilde{T}_j = 1$, combining~\eqref{eq:T-eps-1} and~\eqref{eq:T-eps-2} yields
\begin{align}
\W_1(T,T^*) &\leq \E\norm{T - \tilde{T}}_2 = \E\bigl[\bigl\{(T_1 - \tilde{T}_1)^2 + (T_2 - \tilde{T}_2)^2 + (T_1 - \tilde{T}_1 + T_2 - \tilde{T}_2)^2\bigr\}^{1/2}\bigr] \notag \\
&\leq 2^{1/2}\,\E(|T_1 - \tilde{T}_1| + |T_2 - \tilde{T}_2|) \notag \\
&\leq 2^{1/2} \biggl[\E\bigl(\norm{P_1(\xi - \xi^*)}_2^2\bigr) + 2\Bigl(\frac{\tilde{p}_0}{\tilde{n}}\Bigr)^{1/2} \E\bigl(\norm{P_1(\xi - \xi^*)}_2^2\bigr)^{1/2} \notag \\
&\hspace{1.5cm}+ \E\bigl(\norm{P_2(\xi - \xi^*)}_2^2\bigr) + 2\Bigl(\frac{p - p_0}{\tilde{n}}\Bigr)^{1/2} \E\bigl(\norm{P_2(\xi - \xi^*)}_2^2\bigr)^{1/2}\biggr] \notag \\
\label{eq:W2-bd-epsilon}
&\leq \frac{2^{1/2}}{\tilde{n}}\bigl\{C_X\tilde{p}\W_2^2(\xi,\xi^*) + 2C_X^{1/2}\tilde{p}\W_2(\xi,\xi^*)\bigr\},
\end{align}
where the final bound follows from the Cauchy--Schwarz inequality.

Next, consider an optimal coupling of $(P_1,P_2)$ with a pair of projection matrices $(P_1^*,P_2^*)$ that attains $\W_2\bigl((P_1,P_2),(P_1^*,P_2^*)\bigr)$ and is independent of $\xi$. By the definition of $(P_1^*,P_2^*) = (Q_1^*{Q_1^*}^\top,Q_2^*{Q_2^*}^\top)$ in terms of $Q^* = (Q_1^* \; \; Q_2^*) \sim \Unif(\mathcal{V}_{\tilde{p}})$, it follows as in~\eqref{eq:ortho-conjugation} that for any fixed $Q_0 \in \mathcal{O}'$,
\begin{align*}
(P',P_1^*,P_2^*) = (P',Q_1^*{Q_1^*}^\top,Q_2^*{Q_2^*}^\top) &\eqd \bigl(Q_0 P'Q_0^\top, (Q_0 Q_1^*)(Q_0 Q_1^*)^\top,(Q_0 Q_2^*)(Q_0 Q_2^*)^\top\bigr) \\
&= (Q_0 P'Q_0^\top, Q_0 P_1^*Q_0^\top,\,Q_0 P_2^*Q_0^\top).
\end{align*}
Writing $P_3^* := I_n - (P' + P_1^* + P_2^*)$, we have $(P_1^*,P_2^*,P_3^*) \eqd (Q_0 P_1^*Q_0^\top,\,Q_0 P_2^*Q_0^\top,\,Q_0 P_3^*Q_0^\top)$, so
\begin{align*}
T &= \bigl(\norm{P_1\xi}_2^2, \norm{P_2\xi}_2^2, \norm{P_3\xi}_2^2\bigr) \quad\text{and}\quad T^* \eqd \bigl(\norm{P_1^*\xi}_2^2, \norm{P_2^*\xi}_2^2, \norm{P_3^*\xi}_2^2\bigr) =: \breve{T},
\end{align*}
where the distributional equality follows from the proof of Proposition~\ref{prop:spherical-sym}\textit{(b)}, specifically the paragraph after~\eqref{eq:ortho-conjugation}. For $j \in \{1,2\}$, we deduce similarly to~\eqref{eq:T-eps-1} that
\begin{align}
\label{eq:T-proj-1}
\E|T_j - \breve{T}_j| \leq \E\bigl(\norm{(P_j - P_j^*)\xi}_2^2\bigr) + 2\,\E(\norm{P_j^*\xi}_2^2)^{1/2} \cdot \E\bigl(\norm{(P_j - P_j^*)\xi}_2^2\bigr)^{1/2}.
\end{align}
Similarly to the line after~\eqref{eq:T-eps-1}, we have $\E(\norm{P_j^*\xi}_2^2) = \E(\breve{T}_j) = \tr(P_j)/\tilde{n}$. Also, $\{(P_j,P_j^*)\}_{j=1}^2$ is independent of $\xi$, so
\begin{align}
\E\bigl(\norm{(P_j - P_j^*)\xi}_2^2\bigr) &= \E\tr\bigl\{(P_j - P_j^*)\xi\xi^\top(P_j - P_j^*)^\top\bigr\} \notag \\
&= \E\tr\bigl\{(P_j - P_j^*)\,\E(\xi\xi^\top)(P_j - P_j^*)^\top\bigr\} \notag \\
\label{eq:T-proj-2}
&\leq \norm{\E(\xi\xi^\top)}_{\mathrm{op}}\,\E(\norm{P_j - P_j^*}_{\mathrm{F}}^2) \leq \frac{C_\varepsilon}{\tilde{n}}\E(\norm{P_j - P_j^*}_{\mathrm{F}}^2),
\end{align}
where the penultimate inequality follows from the fact that $\tr(A^\top BA) \leq \tr(A^\top A)\norm{B}_{\mathrm{op}} = \norm{A}_{\mathrm{F}}^2\norm{B}_{\mathrm{op}}$ for $A,B \in \R^{n \times n}$ with $B$ being non-negative definite, while the final inequality holds by the second assumption in~\eqref{eq:exp-proj-eval}. Arguing as in~\eqref{eq:W2-bd-epsilon} and applying the Cauchy--Schwarz inequality, we deduce from~\eqref{eq:T-proj-1} and~\eqref{eq:T-proj-2} that
\begin{align*}
\W_1(T,T^*) &\leq \E\norm{T - \breve{T}}_2 \leq 2^{1/2}\,\E(|T_1 - \breve{T}_1| + |T_2 - \breve{T}_2|) \\
&\leq 2^{1/2} \biggl[\E\bigl(\norm{(P_1 - P_1^*)\xi}_2^2\bigr) + 2\Bigl(\frac{\tilde{p}_0}{\tilde{n}}\Bigr)^{1/2} \E\bigl(\norm{(P_1 - P_1^*)\xi}_2^2\bigr)^{1/2} \\
&\hspace{1.5cm}+ \E\bigl(\norm{(P_2 - P_2^*)\xi}_2^2\bigl) + 2\Bigl(\frac{p - p_0}{\tilde{n}}\Bigr)^{1/2}\E\bigl(\norm{(P_2 - P_2^*)\xi}_2^2\bigr)^{1/2}\biggr] \\
&\leq \frac{2^{1/2}}{\tilde{n}} \bigl\{C_\varepsilon\W_2^2\bigl((P_1,P_2),(P_1^*,P_2^*)\bigr) + 2(C_\varepsilon\tilde{p})^{1/2}\W_2\bigl((P_1,P_2),(P_1^*,P_2^*)\bigr)\bigr\} \\
&\leq \frac{2^{3/2}}{\tilde{n}} \bigl\{C_\varepsilon\inf_{Q \in \mathcal{Q}} \W_2^2(Q,Q^*) + 2^{1/2}(C_\varepsilon\tilde{p})^{1/2} \inf_{Q \in \mathcal{Q}} \W_2(Q,Q^*)\bigr\},
\end{align*}
where the final inequality follows from~\eqref{eq:wasserstein-metrics}. Combining this with~\eqref{eq:W2-bd-epsilon} yields the result.
\end{proof}

\begin{remark*}
Similarly to the remark after Theorem~\ref{thm:doubly-robust-A}, we always have $1/\tilde{n} \leq \norm{\E(\xi\xi^\top)}_{\mathrm{op}} \leq 1$ and $\tilde{p}/\tilde{n} \leq \norm{\E(P - P')}_{\mathrm{op}} \leq 1$, so we may assume without loss of generality that $C_X \in [1,\tilde{n}/\tilde{p}]$ and $C_\varepsilon \in [1,\tilde{n}]$ in Theorem~\ref{thm:doubly-robust}. Indeed, we can write $I_n - P' = Q'Q'^\top$ for some $Q' \in \mathcal{O}_{n \times \tilde{n}}$. Then $\eta := Q'^\top\xi$ takes values in $\mathcal{S}^{\tilde{n} - 1}$ and we can write $\E(\xi\xi^\top) = Q'DQ'^\top$ for some diagonal $D \in \R^{\tilde{n} \times \tilde{n}}$, so
\begin{align}
\label{eq:Cepsilon}
1 = \E(\norm{\xi\xi^\top}_{\mathrm{op}}) &\geq \norm{\E(\xi\xi^\top)}_{\mathrm{op}} = \|Q'DQ'^\top\|_{\mathrm{op}} = \|D\|_{\mathrm{op}} \notag \\
&= \norm{Q'^\top\E(\xi\xi^\top)Q'}_{\mathrm{op}} = \norm{\E(\eta\eta^\top)}_{\mathrm{op}} \geq \frac{\tr\E(\eta\eta^\top)}{\tilde{n}} = \frac{\E(\norm{\eta}_2^2)}{\tilde{n}} = \frac{1}{\tilde{n}}. 
\end{align}
In particular, if $\xi \eqd \xi^*$, then $\xi \eqd Q'U$ for $U \sim \Unif(\mathcal{S}^{\tilde{n} - 1})$. Then $\eta = Q'^\top\xi \eqd U$, so $\E(\eta\eta^\top) = I_{\tilde{n}}/\tilde{n}$ and hence equality holds in the second line of~\eqref{eq:Cepsilon}. Thus, we may take $C_\varepsilon = 1$ here.

On the other hand, we always have
\begin{align}
1 = \E(\norm{P - P'}_{\mathrm{op}}) &\geq \norm{\E(P - P')}_{\mathrm{op}} = \bigl\|Q'^\top\E(P - P')Q'\bigr\|_{\mathrm{op}} \notag \\ 
\label{eq:CX}
&\geq \frac{\tr\bigl\{Q'^\top\E(P - P')Q'\bigr\}}{\tilde{n}} = \frac{\tr\E\{(P - P')(I_n - P')\}}{\tilde{n}} = \frac{\tr\E(P - P')}{\tilde{n}} = \frac{\tilde{p}}{\tilde{n}}.
\end{align}
In the rotationally invariant setting where $P - P' \eqd Q^*{Q^*}^\top = P_1^* + P_2^*$, then since $Q^* \eqd Q'R$ with $R \sim \Unif(\mathcal{O}_{\tilde{n} \times \tilde{p}})$, we have
\[
\E(P - P') = \E(Q^*{Q^*}^\top) = Q'\E(RR^\top)Q'^\top = Q'\Bigl(\frac{\tilde{p}I_{\tilde{n}}}{\tilde{n}}\Bigr)Q'^\top = \frac{\tilde{p}(I_n - P')}{\tilde{n}},
\]
so equality holds in the second line of~\eqref{eq:CX}. Hence, we may take $C_X = 1$ in this case.
\end{remark*}

\subsection{Proofs of lower bounds on the Kolmogorov distance in Section~\ref{sec:upper-bds}}

The proof of Proposition~\ref{prop:dK-W1-lb} is divided into two parts: we begin by proving part~\textit{(a)} before presenting a more sophisticated construction for \textit{(b)} demonstrating the optimality of the exponents in Theorem~\ref{thm:dK-W1} under the additional requirement that $(X_1,\varepsilon_1),\dotsc,(X_n,\varepsilon_n)$ be exchangeable.

\begin{proof}[Proof of Proposition~\ref{prop:dK-W1-lb}\textit{(a)}]
Let $X := (e_1\,\cdots\,e_p)$ and $\varepsilon^* = (\varepsilon_1^*,\dotsc,\varepsilon_n^*) \sim \Unif(\mathcal{S}^{n-1})$. Then $X$ has orthonormal columns with $P_0 = \sum_{j=1}^{p_0} e_j e_j^\top$ and $P = \sum_{j=1}^p e_j e_j^\top$, so
\[
T_2^* = \frac{\norm{(P - P_0)\varepsilon^*}_2^2}{\norm{\varepsilon^*}_2^2} = \sum_{j=p_0+1}^p (\varepsilon_j^*)^2
\quad\text{and}\quad T_3^* = \frac{\norm{(I_n - P)\varepsilon^*}_2^2}{\norm{\varepsilon^*}_2^2} = \sum_{j=p+1}^n (\varepsilon_j^*)^2.
\]
Moreover, $V^* := T_2^*/(T_2^* + T_3^*) \sim \mathrm{Beta}(\frac{p - p_0}{2}, \frac{n - p}{2})$. First, we consider the case $n - p \leq p - p_0$ and let
\[
v_0 :=
\begin{cases}
1 \;\;&\text{if }n - p \in \{1,2\} \\
(p - p_0 - 2)/(n - p_0 - 4) \geq 1/2 \;\;&\text{otherwise},
\end{cases}
\]
so that the density $g^*$ of $V^*$ satisfies $\lim_{v \nearrow v_0} g^*(v) = \sup_{v \in (0,1)} g^*(v)$. Given $\delta \in (0,v_0)$, let
\begin{align}
V &:= V^* + (v_0 - V^*)\Ind_{\{v_0 - \frac{\delta}{1 - T_1^*} < V^* \leq v_0\}}, \notag \\
\label{eq:lbd-epsilon}
\varepsilon &:= \sum_{j=1}^{p_0} \varepsilon_j^*e_j + \sqrt{\frac{V}{V^*}} \sum_{j=p_0+1}^p \varepsilon_j^*e_j + \sqrt{\frac{1 - V}{1 - V^*}} \sum_{j=p+1}^n \varepsilon_j^*e_j.
\end{align}
Then $V \geq V^*$, so $\mathsf{F} = \frac{n - p}{p - p_0} \cdot V/(1 - V) \geq \frac{n - p}{p - p_0} \cdot V^*/(1 - V^*) = \mathsf{F}^*$ and hence
\[
\Pr(\mathsf{F} > q) \geq \Pr(\mathsf{F}^* > q)
\]
for all $q > 0$.  Moreover,
\begin{align*}
\norm{\varepsilon}_2^2 &= T_1^* + V\frac{T_2^*}{V^*} + (1 - V)\frac{T_3^*}{1 - V^*} = T_1^* + \bigl(V + (1 - V)\bigr)(T_2^* + T_3^*) = 1, \\[3pt]
T_1 &= \frac{\norm{P_0\varepsilon}_2^2}{\norm{\varepsilon}_2^2} = \sum_{j=1}^{p_0} \varepsilon_j^2 = T_1^*, \qquad T_2 = \frac{\norm{(P - P_0)\varepsilon}_2^2}{\norm{\varepsilon}_2^2} = \sum_{j=p_0+1}^p \varepsilon_j^2 = V\frac{T_2^*}{V^*}
\end{align*}
and $V = T_2/(T_2 + T_3)$. Since $T_2^*/V^* = 1 - T_1^*$, we have
\begin{align}
\W_1(T,T^*) \leq \E\norm{T - T^*}_2 = 2^{1/2}\,\E|T_2 - T_2^*| &= 2^{1/2}\,\E\biggl(\frac{T_2^*(v_0 - V^*)}{V^*} \Ind_{\{v_0 - \frac{\delta}{1 - T_1^*} < V^* \leq v_0\}}\biggr) \notag \\
\label{eq:lbd-W1}
&\leq 2^{1/2}\delta\,\Pr\Bigl(v_0 - \frac{\delta}{1 - T_1^*} < V^* \leq v_0\Bigr)
\end{align}
and
\begin{align}   
\dK(\mathsf{F},\mathsf{F}^*) = \sup_{x \in \R}\:\Bigl|\Pr\Bigl(\frac{T_2}{T_3} \geq x\Bigr) - \Pr\Bigl(\frac{T_2^*}{T_3^*} \geq x\Bigr)\Bigr| &= \sup_{v \in (0,1)}\,\{\Pr(V \geq v) - \Pr(V^* \geq v)\} \notag \\
&\geq \Pr(V \geq v_0) - \Pr(V^* \geq v_0) \notag \\
&= \Pr\Bigl(v_0 - \frac{\delta}{1 - T_1^*} < V^* \leq v_0\Bigr).
\end{align}
Now let $h_0 := v_0$ if $p - p_0 \in \{1,2\}$ and $h_0 := v_0(1 - 2^{-2/(p - p_0 - 2)})$ otherwise; now fix $h \in (0,h_0)$. If $n - p = 1$, then $v_0 = 1$ and
\begin{equation}
\label{Eq:BetaLowerBound1}
\Pr(v_0 - h < V^* \leq v_0) = \int_{1 - h}^1 \frac{t^{\frac{p - p_0}{2} - 1}(1 - t)^{-1/2}}{B}\,dt \geq \frac{2h^{1/2}}{B} (1 - h)^{(\frac{p - p_0}{2} - 1)_+} > \frac{h^{1/2}}{B},
\end{equation}
where $B = \mathrm{B}(\frac{p - p_0}{2}, \frac{n - p}{2})$. On the other hand, if $n - p \geq 2$, then 
\begin{align}
\label{Eq:BetaLowerBound2}
\Pr(v_0 - h < V^* \leq v_0) = \int_{v_0 - h}^{v_0} \frac{t^{\frac{p - p_0}{2} - 1}(1 - t)^{\frac{n - p}{2} - 1}}{B}\,dt &\geq \frac{h(v_0 - h)^{\frac{p - p_0}{2} - 1}(1 - v_0)^{\frac{n - p}{2} - 1}}{B} \notag \\
&> \frac{hv_0^{\frac{p - p_0}{2} - 1}(1 - v_0)^{\frac{n - p}{2} - 1}}{2B} = \frac{Mh}{2},
\end{align}
where $M = \mathrm{M}\bigl(\frac{p - p_0}{2},\frac{n - p}{2}\bigr) = \sup_{v \in (0,1)} g^*(v)$, as defined before the statement of Theorem~\ref{thm:dK-W1}. For $j \in \{1,2\}$ and when $p_0 \geq 1$, define $m_j$ to be the median of the $\mathrm{Beta}(\frac{p_0}{2},\frac{n - p_0 - j}{2})$ distribution. Moreover, in that case, $V^*$ is independent of $T_1^* \sim \mathrm{Beta}\bigl(\frac{p_0}{2}, \frac{n - p_0}{2}\bigr)$; see~\eqref{eq:dirichlet-indep} in the proof of Theorem~\ref{thm:dK-W1}. Thus, by~\eqref{Eq:BetaLowerBound1} and~\eqref{Eq:BetaLowerBound2}, if $m := m_1 \vee m_2$ and $\delta < (1 - m)h_0$, then for $p_0 \geq 1$,
\begin{align}
&\Pr\Bigl(v_0 - \frac{\delta}{1 - T_1^*} < V^* \leq v_0\Bigr) = \int_0^1 \Pr\Bigl(v_0 - \frac{\delta}{1 - t} < V^* \leq v_0\Bigr) \cdot \frac{t^{\frac{p_0}{2} - 1}(1 - t)^{\frac{n - p_0}{2} - 1}}{B_0}\,dt \notag \\
\label{eq:lbd-dK}
&\hspace{1cm} \geq 
\begin{cases}
\,\displaystyle \int_0^{m_1} \frac{\delta^{1/2}}{B(1 - t)^{1/2}} \cdot \frac{t^{\frac{p_0}{2} - 1}(1 - t)^{\frac{n - p_0}{2} - 1}}{B_0}\,dt = \frac{D_1}{2B} \cdot \delta^{1/2} 
&\text{if }n - p = 1 \\[12pt]
\,\displaystyle\int_0^{m_2} \frac{M\delta}{2(1 - t)} \cdot \frac{t^{\frac{p_0}{2} - 1}(1 - t)^{\frac{n - p_0}{2} - 1}}{B_0}\,dt = \frac{MB_2}{4B_0} \cdot \delta \quad &\text{if }n - p \geq 2,
\end{cases} \nonumber \\
&\hspace{1cm}=
\begin{cases}
\,\displaystyle \frac{C(n,p,p_0)^{3/2}}{2^{1/4} \cdot 3^{3/2}\pi} \cdot (\delta n)^{1/2} 
\quad &\text{if }n - p = 1 \\[10pt]
\dfrac{C(n,p,p_0)^2}{2^{9/2}} \cdot \delta n 
\quad &\text{if }n - p \geq 2,
\end{cases}
\end{align}
where $B_j = \mathrm{B}(\frac{p_0}{2},\frac{n - p_0 - j}{2})$ for $j \in \{0,1,2\}$ and $D_1 = B_1/B_0$ are as in Theorem~\ref{thm:dK-W1}.  Moreover, by~\eqref{Eq:BetaLowerBound1} and~\eqref{Eq:BetaLowerBound2}, the final bound in~\eqref{eq:lbd-dK} holds even when $p_0 = 0$. Therefore, for any $\eta > 0$, it follows from~\eqref{eq:lbd-W1}--\eqref{eq:lbd-dK} that whenever $0 < \delta < \min\{\eta/(2^{1/2}n),(1 - m)h_0\}$, we have $\W_1(nT,nT^*) \leq 2^{1/2}\delta n < \eta$ and
\begin{align*}
\dK(\mathsf{F},\mathsf{F}^*) &\geq 
\begin{cases}
\dfrac{\W_1(T,T^*)^{1/3}}{(2^{1/2}\delta)^{1/3}}\Pr\Bigl(v_0 - \dfrac{\delta}{1 - T_1^*} < V^* \leq v_0\Bigr)^{2/3} \;\;&\text{if }n - p = 1, \\[12pt]
\dfrac{\W_1(T,T^*)^{1/2}}{(2^{1/2}\delta)^{1/2}}\Pr\Bigl(v_0 - \dfrac{\delta}{1 - T_1^*} < V^* \leq v_0\Bigr)^{1/2} \;\;&\text{if }n - p \geq 2,
\end{cases} \\
&\geq
C(n,p,p_0) \cdot
\begin{cases}
2^{-1/3} \cdot 3^{-1} \cdot \pi^{-2/3} \W_1(nT,nT^*)^{1/3} \;\;&\text{if }n - p = 1 \\
2^{-5/2} \W_1(nT,nT^*)^{1/2} \;\;&\text{if }n - p \geq 2,
\end{cases}
\end{align*}
which yields the result when $n - p \leq p - p_0$. On the other hand, if $n - p > p - p_0$, then we can argue similarly with $\varepsilon$ in~\eqref{eq:lbd-epsilon} defined in terms of
\[
v_0 :=
\begin{cases}
0 \;\;&\text{if }p - p_0 \in \{1,2\} \\
(p - p_0 - 2)/(n - p_0 - 4) \leq 1/2 \;\;&\text{otherwise}
\end{cases}
\]
and $V := V^* + (v_0 - V^*)\Ind_{\{v_0 \leq V^* < v_0 + \frac{\delta}{1 - T_1^*}\}}$ instead.
\end{proof}

Our proof of Proposition~\ref{prop:dK-W1-lb}\textit{(b)} below requires some further notation and preliminaries. We write $\triangle_n := \{(u_1,\dotsc,u_n) \in [0,\infty)^n : \sum_{i=1}^n u_i = 1\}$ for the standard $(n - 1)$-dimensional simplex and $\mathrm{Perm}(n)$ for the set of all permutation matrices $\Pi \in \{0,1\}^{n \times n}$. Given $0 \leq p_0 < p < n$ and $u = (u_1,\ldots,u_n) \in (0,\infty)^n$, let
\[
r(u) := \frac{\sum_{i=p_0+1}^p u_i}{\sum_{i=p_0+1}^n u_i}.
\]
For $\Pi \in \mathrm{Perm}(n)$ and $c,h > 0$ such that $0 \leq c - h < c + h \leq 1$, define 
\begin{equation}
\label{eq:simplex-regions}
\triangle_{n,\Pi} := \{u \in \triangle_n \cap (0,\infty)^n : |r(\Pi u) - c| < h\}, \qquad \triangle_n^\circ := \bigcup_{\Pi \in \mathrm{Perm}(n)} \triangle_{n,\Pi}, \qquad \triangle_n' := \triangle_n \setminus \triangle_n^\circ.
\end{equation}
Then $\triangle_n$ and $\triangle_n'$ are closed subsets of $\bigl\{(x_1,\dotsc,x_n) \in \R^n : \sum_{i=1}^n x_i = 1\bigr\}$ that are permutation-invariant, i.e.~$\triangle_n = \{\Pi u : u \in \triangle_n\} =: \Pi(\triangle_n)$ for all $\Pi \in \mathrm{Perm}(n)$, and similarly for $\triangle_n'$. We now define a function $\mathcal{T}$ on $\triangle_n$ that projects outwards onto $\triangle_n'$ along rays emanating from $u^* := (1/n,\dotsc,1/n) \in \triangle_n$: let $\mathcal{T}(u^*) := u^*$ and
\begin{equation}
\label{eq:Gn-simplex}
\mathcal{T}(u) := u + \lambda^*(u) \cdot (u - u^*)
\end{equation}
for $u \in \triangle_n \setminus \{u^*\}$, where $\lambda^*(u) := \inf\{\lambda \geq 0 : u + \lambda(u - u^*) \notin \triangle_n^\circ\}$. See Figure~\ref{fig:simplex-proj} below for an illustration of this geometric configuration.

\begin{lemma}
\label{lem:simplex-proj}
Suppose that $(p - p_0)/(n - p_0) \notin (c - h, c + h)$. Then $\mathcal{T}$ is a well-defined permutation-equivariant projection from $\triangle_n$ to $\triangle_n'$. In other words, for $u \in \triangle_n$, we have $\mathcal{T}(u) \in \triangle_n'$ and $\mathcal{T}(\Pi u) = \Pi \mathcal{T}(u)$ for all $\Pi \in \mathrm{Perm}(n)$, with $\lambda^*(u) < \infty$ if $u \neq u^*$ and $\mathcal{T}(u) = u$ for all $u \in \triangle_n'$. Moreover, $\mathcal{T}$ is Borel measurable on $\triangle_n$ with
\begin{equation}
\label{eq:Gn-proj-dist}
\sup_{u \in \triangle_n} \norm{\mathcal{T}(u) - u}_1 \leq 2\binom{n}{p}\binom{p}{p_0}Kh,
\end{equation}
where we can take
\begin{equation}
\label{eq:simplex-K}
K := \Bigl(\frac{n}{n - p_0}\Bigr)^2 \cdot \frac{2}{\inf_{r \in (c-h,c+h)} |r - \frac{p - p_0}{n - p_0}|}.
\end{equation}
In addition, if $u \in \triangle_n$ and $\Pi \in \mathrm{Perm}(n)$ are such that $r(\Pi u) > (p - p_0)/(n - p_0)$, then $r\bigl(\Pi\mathcal{T}(u)\bigr) \geq r(\Pi u)$.
\end{lemma}

\begin{proof}[Proof of Proposition~\ref{prop:dK-W1-lb}\textit{(b)}]
Let $Y^* = (Y_1^*,\dotsc,Y_n^*) \sim \mathrm{Dirichlet}(1/2,\dotsc,1/2)$ and $Y := \mathcal{T}(Y^*)$, where $\mathcal{T}$ is defined as in~\eqref{eq:Gn-simplex} in terms of $c,h \in (0,1)$ that we will specify later. Then by the permutation equivariance of $\mathcal{T}$ established in Lemma~\ref{lem:simplex-proj} as well as the exchangeability of $Y^*$, we have
\[
\Pi Y = \Pi\mathcal{T}(Y^*) = \mathcal{T}(\Pi Y^*) \eqd \mathcal{T}(Y^*) = Y
\]
for every $\Pi \in \mathrm{Perm}(n)$, so $Y$ is also exchangeable. Now let $\eta_1,\dotsc,\eta_p,\xi_1,\dotsc,\xi_n$ be independent Rademacher random variables that are independent of $\Pi^* \sim \Unif\bigl(\mathrm{Perm}(n)\bigr)$ and $Y^*$. Define
\[
X := (\eta_1\Pi^* e_1 \;\;\cdots\;\; \eta_p\Pi^* e_p), \quad \varepsilon := (\xi_1\sqrt{Y_1},\dotsc,\xi_n\sqrt{Y_n}), \quad \varepsilon^* := (\xi_1\sqrt{Y_1^*},\dotsc,\xi_n\sqrt{Y_n^*}),
\]
so that $X$ is an $n \times p$ matrix with exchangeable zero-mean rows, which is independent of random vectors $\varepsilon$ and $\varepsilon^*$ with exchangeable zero-mean components. Since $Y^* \eqd (Z_1^2,\dotsc,Z_n^2)/\norm{Z}_2^2$ with $Z = (Z_1,\dotsc,Z_n) \sim N_n(0,I_n)$, we have $\varepsilon^* \sim \Unif(\mathcal{S}^{n-1})$. Moreover, $X$ has orthonormal columns with $P_0 = \sum_{j=1}^{p_0} (\Pi^*e_j)(\Pi^*e_j)^\top$ and $P = \sum_{j=1}^p (\Pi^*e_j)(\Pi^*e_j)^\top$, so
\[
T_2^* = \frac{\norm{(P - P_0)\varepsilon^*}_2^2}{\norm{\varepsilon^*}_2^2} = \sum_{j=p_0+1}^p Y_{J_j}^* = \sum_{j=p_0+1}^p ({\Pi^*}^\top Y^*)_j
\]
and
\[
T_3^* = \frac{\norm{(I_n - P)\varepsilon^*}_2^2}{\norm{\varepsilon^*}_2^2} = \sum_{j=p+1}^n Y_{J_j}^* = \sum_{j=p+1}^n ({\Pi^*}^\top Y^*)_j,
\]
where $J_1,\dotsc,J_n$ constitute a uniform random permutation of $[n]$. Similarly,
\[
T_2 = \frac{\norm{(P - P_0)\varepsilon}_2^2}{\norm{\varepsilon}_2^2} = \sum_{j=p_0+1}^p Y_{J_j}
\quad\text{and}\quad T_3 = \frac{\norm{(I_n - P)\varepsilon}_2^2}{\norm{\varepsilon}_2^2} = \sum_{j=p+1}^n Y_{J_j},
\]
where $J_1,\dotsc,J_n$ are the same indices as above. We have $V^* := T_2^*/(T_2^* + T_3^*) = r({\Pi^*}^\top Y^*) \sim \mathrm{Beta}(\frac{p - p_0}{2},\frac{n - p}{2})$ and
\[
V := \frac{T_2}{T_2 + T_3} = \frac{\sum_{j=p_0+1}^p Y_{J_j}}{\sum_{j=p_0+1}^n Y_{J_j}} = r({\Pi^*}^\top Y).
\]
By the final assertion of Lemma~\ref{lem:simplex-proj}, if $V^* > (p - p_0)/(n - p_0)$, then $V \geq V^*$, so
\[
\Pr(V > t) \geq \Pr(V^* > t)
\]
for $t > (p - p_0)/(n - p_0)$. Since $\mathsf{F} = \frac{n - p}{p - p_0} \cdot V/(1 - V)$ and similarly $\mathsf{F}^* = \frac{n - p}{p - p_0} \cdot V^*/(1 - V^*)$, it follows for $q > 1$ that
\begin{equation}
\label{eq:anticonservative}
\Pr(\mathsf{F} > q) = \Pr\Bigl(V > \frac{(p - p_0)q}{n - p + (p - p_0)q}\Bigr) \geq \Pr\Bigl(V^* > \frac{(p - p_0)q}{n - p + (p - p_0)q}\Bigr) = \Pr(\mathsf{F}^* > q).
\end{equation}
Next, by the definition~\eqref{eq:simplex-regions} of $\triangle_{n,\Pi}$ for each $\Pi \in \mathrm{Perm}(n)$, the fact that $\Pi Y^* \eqd Y^*$ implies that
\begin{equation}
\label{eq:Ystar-exch}
\Pr(Y^* \in \triangle_{n,\Pi}) = \Pr\bigl(|r(\Pi Y^*) - c| < h\bigr) = \Pr(|V^* - c| < h) =: \eta.
\end{equation}
There exists $\mathrm{Perm}(n)' \subseteq \mathrm{Perm}(n)$ of cardinality $\binom{n}{p}\binom{p}{p_0}$ such that $\triangle_n^\circ = \bigcup_{\Pi \in \mathrm{Perm}(n)'} \triangle_{n,\Pi}$, and by Lemma~\ref{lem:simplex-proj}, $\mathcal{T}(u) \neq u$ if and only if $u \in \triangle_n^\circ$. Therefore, by a union bound,
\begin{equation}
\label{eq:Yfixed}
\Pr(Y \neq Y^*) = \Pr(Y^* \in \triangle_n^\circ) \leq \binom{n}{p}\binom{p}{p_0} \cdot \eta.
\end{equation}
If $Y^*$ lies in $\triangle_n' \cap (0,\infty)^n$, then $Y = \mathcal{T}(Y^*) = Y^*$ and hence $|r(\Pi Y) - c| \geq h$ for all $\Pi \in \mathrm{Perm}(n)$, so $|V - c| \geq h$. Moreover, by~\eqref{eq:Gn-proj-dist}, $\max_{i \in [n]} |Y_i - Y_i^*| \leq \norm{\mathcal{T}(Y^*) - Y^*}_1 \leq ah$ with $a := 2\binom{n}{p}\binom{p}{p_0}K$. Thus, if $\min_{i \in [n]} Y_i^* > ah$, then $Y \in \triangle_n \cap (0,\infty)^n$ does not lie on the boundary of $\triangle_n$, so again $|V - c| \geq h$. Then by a union bound, the exchangeability of $Y^*$ and Lemma~\ref{lem:dirichlet-conditional} and~\eqref{eq:Ystar-exch}, there exists $b > 0$ depending only on $n,p,p_0$ such that if $h \leq b\min(c, 1 - c)/K$, then
\begin{align}
\label{eq:dirichlet-boundary}
\Pr(|V - c| < h) &\leq \Pr\Bigl(\min_{i \in [n]} Y_i^* \leq ah,\;Y^* \in \triangle_n^\circ\Bigr) \\
&\leq \sum_{\Pi \in \mathrm{Perm}(n)'} \Pr\biggl(\min_{i \in [n]} Y_i^* \leq ah \biggm| Y^* \in \triangle_{n,\Pi}\biggr) \cdot \Pr(Y^* \in \triangle_{n,\Pi}) \notag \\[3pt]
&\leq \binom{n}{p} \binom{p}{p_0}\,\Pr\biggl(\min_{i \in [n]} Y_i^* \leq ah \biggm| |r(Y^*) - c| < h\biggr) \cdot \Pr(|V^* - c| < h) \leq \frac{\eta}{2}. \notag
\end{align}
Consequently,
\begin{align}
\dK(\mathsf{F},\mathsf{F}^*) = \sup_{x \in \R}\,\Bigl|\Pr\Bigl(\frac{T_2}{T_3} > x\Bigr) - \Pr\Bigl(\frac{T_2^*}{T_3^*} > x\Bigr)\Bigr| &= \sup_{v \in (0,1)}\,|\Pr(V > v) - \Pr(V^* > v)|
\notag \\
&\geq \frac{\bigl|\Pr(V^* - c| \leq h) - \Pr(|V - c| \leq h)\bigr|}{2} \notag \\[3pt]
\label{eq:exch-lb-dK}
&\geq \frac{\Pr(|V^* - c| < h)}{4} =\frac{\eta}{4}.
\end{align}
Next, since $\sum_{j=1}^3 T_j = \sum_{j=1}^3 T_j^* = 1$, we have
\[
\norm{T - T^*}_2^2 = (T_1 - T_1^*)^2 + (T_2 - T_2^*)^2 + \bigl((T_1 - T_1^*) + (T_2 - T_2^*)\bigr)^2 \leq 2(|T_1 - T_1^*| + |T_2 - T_2^*|)^2.
\]
Now $\{I_1,\dotsc,I_p\}$ is a uniformly chosen random subset of $[n]$ of cardinality $p$, which contains any given $i \in [n]$ with probability $p/n$, so by~\eqref{eq:Gn-proj-dist} in Lemma~\ref{lem:simplex-proj} and~\eqref{eq:Yfixed},
\begin{align}
\W_1(nT,nT^*) \leq n\,\E\norm{T - T^*}_2 &\leq 2^{1/2}n\,\E(|T_1 - T_1^*| + |T_2 - T_2^*|) \leq 2^{1/2}n\sum_{j=1}^p \E(|Y_{J_j} - Y_{J_j}^*|) \notag \\
&= 2^{1/2}n \cdot \frac{p}{n} \sum_{i=1}^n \E|Y_i - Y_i^*| = 2^{1/2}p\,\E(\norm{\mathcal{T}(Y^*) - Y^*}_1 \cdot \Ind_{\{Y \neq Y^*\}}) \notag \\
\label{eq:exch-lb-W1}
&\leq 2^{3/2}p \binom{n}{p}\binom{p}{p_0} Kh \cdot \Pr(Y \neq Y^*) \leq 2^{3/2}p\binom{n}{p}^2\binom{p}{p_0}^2 Kh \cdot \eta,
\end{align}
where
\[
K = \Bigl(\frac{n}{n - p_0}\Bigr)^2 \cdot \frac{2}{\inf_{r \in (c-h,c+h)} |r - \frac{p - p_0}{n - p_0}|}.
\]
Let $B := \mathrm{B}(\frac{p - p_0}{2}, \frac{n - p}{2})$ and $B_j := \mathrm{B}(\frac{p_0}{2}, \frac{n - p_0 - j}{2})$ for $j \in \{0,1,2\}$. Consider the following cases:

\begin{itemize}
\item When $n - p = 1$, let $c,h$ be such that $c - h = 1 - \delta$, where
\[
0 < \delta <
\frac{1}{2(n - p_0)}
\wedge \bigl(1 - 2^{-2/(p - p_0 - 2)_+}\bigr)
\]
and hence
\[
K = \Bigl(\frac{n}{n - p_0}\Bigr)^2 \cdot \frac{2}{1 - \delta - \frac{p - p_0}{n - p_0}} < \frac{4n^2}{n - p_0}.
\]
Taking $b \in (0,1)$ from the line above~\eqref{eq:dirichlet-boundary}, let $c'' := b(n - p_0)/(4n^2) \in (0,1)$ and $h := c''\delta/2$. Then
\[
\eta = \int_{1 - \delta}^{1 - \delta + 2h} \frac{t^{\frac{p - p_0}{2} - 1}(1 - t)^{-1/2}}{B}\,dt \geq \frac{(1 - \delta)^{\frac{(p - p_0 - 2)_+}{2}}}{B} \int_{(1 - c'')\delta}^\delta t^{-1/2}\,dt > \frac{c''\delta^{{1/2}}}{2B}. 
\]
Since $h \leq c''(\delta - h) < (b/K) \cdot (1 - c) = b\min(c, 1 - c)/K$, it follows from~\eqref{eq:exch-lb-dK} and~\eqref{eq:exch-lb-W1} that
\begin{equation}
\label{eq:exch-lb-1}
\dK(\mathsf{F},\mathsf{F}^*) \geq \frac{\eta}{4} > \frac{(\eta{c''}^2\delta)^{1/3}}{2^{8/3}B^{2/3}} \gtrsim_{n,p,p_0} \W_1(nT,nT^*)^{1/3}.
\end{equation}
\item Similarly, if $p - p_0 = 1$, then let $h := c''\delta/2$ and $c := \delta - h$, where
\[
0 < \delta <
\frac{1}{2(n - p_0)}
\wedge \bigl(1 - 2^{-2/(n - p - 2)_+}\bigr).
\]
Then $K < 4n^2/(n - p_0)$ and $\eta > c''\delta^{1/2}/(2B)$ as above, so we obtain the same bound~\eqref{eq:exch-lb-1}.
\item When $n - p > p - p_0 \geq 2$, let $c,h$ be such that $c - h = v_0$, where $v_0 := (p - p_0 - 2)/(n - p_0 - 4)$ if $p - p_0 > 2$ and $v_0 := 1/n$ otherwise. In both cases, $v_0 < (p - p_0)/(n - p_0) < 1/2$. For $h = \delta/2$ with $0 < \delta < \min\bigl\{\frac{p - p_0}{n - p_0} - v_0,\,(1 - v_0)(1 - 2^{-2/(n - p - 2)})\bigr\}/2$, we have
\[
\frac{p - p_0}{n - p_0} - c - h = \frac{p - p_0}{n - p_0} - c + h - \delta \geq \frac{1}{2}\Bigl(\frac{p - p_0}{n - p_0} - v_0\Bigr) =: D_0
\]
so
\[
K = \Bigl(\frac{n}{n - p_0}\Bigr)^2 \cdot \frac{2}{\frac{p - p_0}{n - p_0} - c - h} \leq \frac{2n^2}{(n - p_0)^2 D_0} =: K_0.
\]
Moreover, 
\begin{align*}
\eta = \int_{v_0}^{v_0 + \delta} \frac{t^{\frac{p - p_0}{2} - 1}(1 - t)^{\frac{n - p}{2} - 1}}{B}\,dt &\geq \frac{\delta v_0^{\frac{p - p_0}{2} - 1}(1 - v_0 - \delta)^{\frac{n - p}{2} - 1}}{B} \\
&> \frac{\delta v_0^{\frac{p - p_0}{2} - 1}(1 - v_0)^{\frac{n - p}{2} - 1}}{2B} = \frac{M\delta}{2},
\end{align*}
where $M = \mathrm{M}\bigl(\frac{p - p_0}{2},\frac{n - p}{2}\bigr) = \sup_{v \in (0,1)} g^*(v)$, as defined before the statement of Theorem~\ref{thm:dK-W1}. Therefore, if in addition $h \leq bv_0/(2K_0)$, then $h \leq b\min(c, 1 - c)/K$, so by~\eqref{eq:exch-lb-dK} and~\eqref{eq:exch-lb-W1},
\begin{equation}
\label{eq:exch-lb-2}
\dK(\mathsf{F},\mathsf{F}^*) \geq \frac{\eta}{4} > \Bigl(\frac{\eta M\delta}{32}\Bigr)^{1/2} \gtrsim_{n,p,p_0} \W_1(nT,nT^*)^{1/2}.
\end{equation}
\item If $p - p_0 > n - p \geq 2$, then similarly we choose $c,h$ such that $c + h = v_0$ and $h = \delta/2$ with $0 < \delta < \min\bigl\{v_0 - \frac{p - p_0}{n - p_0},\,v_0(1 - 2^{-2/(p - p_0 - 2)})\bigr\}/2$, so that as in the previous case, $\eta > M\delta/2$ and
\[
K = \Bigl(\frac{n}{n - p_0}\Bigr)^2 \cdot \frac{2}{c - h - \frac{p - p_0}{n - p_0}} \leq K_0
\]
for some $K_0 > 0$ depending only on $n,p,p_0$. If in addition $h \leq bv_0/(2K_0)$, then we obtain the same bound~\eqref{eq:exch-lb-2}.
\item Finally, if $n - p = p - p_0 \geq 2$, then choose $c,h$ such that
\[
c - h = \frac{1}{2} + \frac{1}{2(n - p)^{1/2}} =: v_1
\]
and $h = \delta/2 > 0$, with $\delta < (1 - v_1)(1 - 2^{-2/(n - p - 2))_+})$. Since $M = 1/(2^{n - p - 2}B)$, we have
\begin{align*}
\eta = \int_{v_1}^{v_1 + \delta} \frac{t^{\frac{p - p_0}{2} - 1}(1 - t)^{\frac{n - p}{2} - 1}}{B}\,dt &\geq \frac{\delta v_1^{\frac{p - p_0}{2} - 1}(1 - v_1 - \delta)^{\frac{n - p}{2} - 1}}{B} \\
&\geq \frac{\delta v_1^{\frac{p - p_0}{2} - 1}(1 - v_1)^{\frac{n - p}{2} - 1}}{2B} \\
&\geq \frac{\delta}{2 \cdot 2^{n - p - 2}B} \Bigl(1 - \frac{1}{n - p}\Bigr)^{\frac{n - p}{2} - 1} > \frac{M\delta}{4}.
\end{align*}
Moreover,
\[
K = \Bigl(\frac{n}{n - p_0}\Bigr)^2 \cdot \frac{2}{v_1 - \frac{p - p_0}{n - p_0}} = 4\Bigl(\frac{n}{n - p_0}\Bigr)^2 (n - p)^{1/2}.
\]
Thus, if in addition $h \leq b\min(c, 1 - c)/K$, then by~\eqref{eq:exch-lb-dK} and~\eqref{eq:exch-lb-W1},
\[
\dK(\mathsf{F},\mathsf{F}^*) \geq \frac{\eta}{4} > \Bigl(\frac{\eta M\delta}{64}\Bigr)^{1/2} \gtrsim_{n,p,p_0} \W_1(nT,nT^*)^{1/2}.
\]
\end{itemize}
Recalling~\eqref{eq:exch-lb-W1}, we deduce in all cases that for any $\eta' > 0$, there exist $c,h \in (0,1)$ under the constructions defined above, $\W_1(nT,nT^*) < \eta'$ and the desired bound~\eqref{eq:dK-W1-lb} holds for some multiplicative factor $c' > 0$ depending only on $n,p,p_0$.
\end{proof}

We now present two concrete probabilistic constructions that achieve the lower bound in Proposition~\ref{prop:dK-W1-lb}\textit{(b)} in the cases $p_0 = 0$, $p = 1$ and $p_0 = 0$, $p = 2$ respectively.

\begin{example}
\label{ex:exchangeable-lb-1}
Suppose that $p_0 = 0$, $p = 1$ and $n \geq 3$, and write $1_n = (1,\dotsc,1) \in \R^n$. Let~$\xi$ be a Rademacher random variable and let $\varepsilon^* \sim \mathrm{Unif}(\mathcal{S}^{n-1})$ be independent of $\xi$. Then $X := \xi 1_n$ and~$\varepsilon^*$ are independent, and both have exchangeable coordinates. Letting $u_0 := n^{-1/2}1_n \in \mathcal{S}^{n-1}$, we have $P_0 = 0 \in \R^{n \times n}$ and $P = u_0 u_0^\top$, so 
\[
T_2^* = \frac{\|P\varepsilon^*\|_2^2}{\|\varepsilon^*\|_2^2} \eqd (u_0^\top \varepsilon^*)^2,
\quad
T^*=(0, T_2^*, 1 - T_2^*), \quad\text{and}\quad \mathsf{F}^* = (n - 1)\frac{T_2^*}{1  - T_2^*}.
\]
Since $\Pr(|u_0^\top \varepsilon^*| = 1) = 0$, we can define almost surely
\[
V^{**} := \frac{\varepsilon^*-(u_0^\top \varepsilon^*)u_0}{\sqrt{1-T^*_2}},
\]
which takes values in $\mathcal{S}^{n-2}(u_0^\perp) := \{u \in \mathcal{S}^{n-1}:u_0^\top u = 0\}$.  Then
\[
\varepsilon^* = \mathrm{sgn}(u_0^\top \varepsilon^*) \sqrt{T_2^*}\,u_0 + \sqrt{1-T_2^*}\,V^{**},
\]
where $\mathrm{sgn}(u_0^\top \varepsilon^*)$ is a Rademacher random variable, $V^{**}$ is uniformly distributed on $\mathcal{S}^{n-2}(u_0^\perp)$ and $\bigl(\mathrm{sgn}(u_0^\top \varepsilon^*),V^{**}\bigr)$ is independent of $T^*_2$. 

For $\alpha \in (0,1)$, write $q_\alpha^*$ and $s_\alpha := q_\alpha^*/(n - 1 + q_\alpha^*) \in (0,1)$ for the $(1 - \alpha)$-quantiles of~$\mathsf{F}^*$ and $T_2^*$ respectively, as in the proof of Theorem~\ref{thm:dK-alpha}. For fixed $\alpha' \in (0,\alpha)$, let $s := s_{\alpha'}$ and $h_0 := \min(s - s_\alpha, 1 - s)/2$. For $h \in (0,h_0]$, define
\[
T_{2,h} :=
\begin{cases}
T_2^* + h &\text{if }s - h < T_2^* < s, \\
T_2^* &\text{otherwise}.
\end{cases}
\]
Then the perturbed error vector
\[
\varepsilon_h := \mathrm{sgn}(u_0^\top \varepsilon^*) \sqrt{T_{2,h}}\,u_0 + \sqrt{1 - T_{2,h}}\,V^{**}
\]
takes values in $\mathcal{S}^{n-1}$, is independent of $X$ and has exchangeable coordinates. Moreover,
\[
T_{2,h}=\frac{\|P\varepsilon_h\|_2^2}{\|\varepsilon_h\|_2^2} = (u_0^\top \varepsilon_h)^2 \quad\text{and}\quad \mathsf{F}_h := (n - 1)\frac{T_{2,h}}{1 - T_{2,h}} \geq (n - 1)\frac{T_2^*}{1  - T_2^*} = \mathsf{F}^*,
\] 
so
\[
\Pr(\mathsf{F}_h > q) \geq \Pr(\mathsf{F}^* > q)
\]
for all $q > 0$. In particular, since 
$(n - 1)s/(1 - s) = q_{\alpha'}^*$ is the $(1 - \alpha')$-quantile of $\mathsf{F}^*$, we have
\begin{align}
\label{eq:tail-inflation-bounds}
\Pr(\mathsf F_h > q_{\alpha'}^*) - \alpha' = \Pr(\mathsf F_h > q_{\alpha'}^*) - \Pr(\mathsf F^* > q_{\alpha'}^*) &= \Pr(T_{2,h} > s) - \Pr(T_2^* > s) \nonumber \\
&= \Pr(s - h < T_2^* < s) =: \eta.
\end{align}
By the calculations leading up to~\eqref{eq:eta-lb} in the proof of Proposition~\ref{prop:dK-alpha-lb}, if $s = (1 + s_\alpha)/2$ and $h < (1 - s)/2$, then
\[
\eta \gtrsim_n (\alpha^{1 - \frac{2}{n - 1}}h) \vee h^{\frac{n - 1}{2}}.
\]
Moreover, by construction, $|T_{2,h}- T^*_2| = h\Ind_{\{s - h < T^*_2 < s\}}$, so $\xi := \W_1(T_{2,h},T_2^*) \leq h\eta$ and hence
\[
\eta \gtrsim_n (\alpha^{1 - \frac{2}{n - 1}}\xi/\eta) \vee (\xi/\eta)^{\frac{n - 1}{2}}.
\]
Together with \eqref{eq:tail-inflation-bounds}, this yields
\begin{align*}
d_{\mathrm{K},\alpha}(\mathsf{F},\mathsf{F}^*) \geq \eta \gtrsim_n \max\bigl\{\alpha^{\frac{n - 3}{2(n - 1)}}\W_1(nT,nT^*)^{1/2},\,\W_1(nT,nT^*)^{\frac{n - 1}{n + 1}}\bigr\}.
\end{align*}
For any $\eta' > 0$, we can choose a sufficiently small $h$ in this example so that $\W_1(nT,nT^*) < \eta'$. In conclusion, we have constructed a joint distribution for $(X,\varepsilon)$ where $X$ and $\varepsilon = \varepsilon_h$ are independent and each have exchangeable entries, and satisfying the lower bound in Proposition~\ref{prop:dK-alpha-lb}.  Moreover, the size of the $F$-test is always at least the nominal level, and is strictly anti-conservative at level $\alpha'$.
\end{example}

\begin{example}
\label{ex:exchangeable-lb-2}
Fix $n \geq 6$, set $p_0=0$ and $p=2$, and let
\[
Y^*=(Y_1^*,\ldots,Y_n^*)\sim \mathrm{Dirichlet}\Bigl(\frac{1}{2},\ldots,\frac{1}{2}\Bigr).
\]
Then $Y^*$ has exchangeable components, $\sum_{i=1}^n Y_i^*=1$ almost surely, and for every distinct 
$r,s \in [n]$,
\[
Y_r^*+Y_s^*\sim \mathrm{Beta}\Bigl(1,\frac{n-2}{2}\Bigr).
\]
Let $\Pi=\{I',J'\}$ and $\Lambda=\{I,J\}$ be independent, distinct random unordered pairs, each uniformly distributed on the $N:=\binom n2$ two-element subsets of $[n]$, and assume that $(\Pi,\Lambda)$ is independent of $Y^*$. Define $S:=Y_{I'}^*+Y_{J'}^*$, and for $\delta\in(0,1/n]$, define $Y=(Y_1,\ldots,Y_n)$ by
\[
Y_i := Y_i^*\mathbbm{1}_{\{S\geq \delta\}} + \frac{\delta}{S}\,Y_{I'}^*\mathbbm{1}_{\{S < \delta,i=I'\}} + \frac{\delta}{S}\,Y_{J'}^*\mathbbm{1}_{\{S < \delta,i=J'\}} + \frac{1-\delta}{1-S}\,Y_i^*\mathbbm{1}_{\{S < \delta,i \notin\{I',J'\}\}}
\]
for $i \in [n]$.  Then $Y$ has exchangeable components and $\sum_{i=1}^n Y_i=1$ almost surely.

Let $\xi_1,\dots,\xi_n,\xi_{X_1},\xi_{X_2}$ be independent Rademacher random variables that are independent of $(Y^*,Y,\Pi,\Lambda)$, and define
\[
\varepsilon :=(\xi_1\sqrt{Y_1},\ldots,\xi_n\sqrt{Y_n}),
\qquad
\varepsilon^* :=(\xi_1\sqrt{Y_1^*},\ldots,\xi_n\sqrt{Y_n^*})
\]
and the $n \times 2$ random matrix
\[
X= \begin{pmatrix}\xi_{X_1}e_I & \xi_{X_2}e_J\end{pmatrix}.
\]
Then $X$ has exchangeable rows, $\varepsilon$ and $\varepsilon^*$ have exchangeable coordinates with $\|\varepsilon\|_2=\|\varepsilon^*\|_2=1$, and
$X$ is independent of $(\varepsilon,\varepsilon^*)$. Hence
\[
T_2:=\|P\varepsilon\|_2^2=\varepsilon_I^2+\varepsilon_J^2=Y_I+Y_J,
\qquad
T_2^*:=\|P\varepsilon^*\|_2^2=(\varepsilon_I^*)^2+(\varepsilon_J^*)^2=Y_I^*+Y_J^*,
\]
and $\dK(\mathsf{F},\mathsf{F}^*)=\dK(T_2,T_2^*)$. Since $T_2^*\sim \mathrm{Beta}\bigl(1,(n-2)/2\bigr)$, we have for every $u\in[0,1]$ that
\[
\Pr(T_2^*\leq u)=1-(1-u)^{(n-2)/2}.
\]

Let $\Delta_i:=Y_i-Y_i^*$ for $i \in [n]$. Then 
\[
\E\bigl(|T_2-T_2^*|\,\big|\,Y,Y^*\bigr)
=
\frac1N\sum_{1\leq r<s\leq n} |\Delta_r+\Delta_s|
\leq
\frac1N\sum_{1\leq r<s\leq n} (|\Delta_r|+|\Delta_s|)
=
\frac{2}{n}\sum_{i=1}^n |\Delta_i|.
\]
If $S\geq \delta$, then $Y=Y^*$, so $\sum_{i=1}^n |\Delta_i|=0$. On the other hand, if $S<\delta$, then $\sum_{i=1}^n |\Delta_i|=2(\delta-S),$ and thus
\begin{align*}
\W_1(T_2,T_2^*)
&\leq \E|T_2-T_2^*|
\leq \frac{4}{n}\E\bigl\{(\delta-S)_+\bigr\} \\
&= \frac{4}{n}
\int_0^\delta \Pr(S\leq u) \, du
=
\frac{4}{n}\int_0^\delta \bigl\{1-(1-u)^{(n-2)/2}\bigr\} \,du
\leq \frac{(n-2)\delta^2}{n},
\end{align*}
as $1-(1-u)^a\leq au$ for $u\in[0,1]$ and $a \geq 1$. Therefore,
\begin{equation}
\label{Eq:W1bound}
\W_1(nT,nT^*)=2^{1/2}n\W_1(T_2,T_2^*)\leq 2^{1/2}(n-2)\delta^2.
\end{equation}

We bound $\dK(T_2,T_2^*)$ from below. Define $t:=\delta/2$ and
\[
A_\delta :=\{\Lambda=\Pi,\ S\leq t\},
\qquad
B_\delta :=\{T_2\leq t<T_2^*\}.
\]
On the event $A_\delta$, we have $T_2^*=S\leq t = \delta/2$ and $T_2=Y_{I'}+Y_{J'}=\delta>t$.  Hence
\begin{equation}
\label{Eq:AdeltaBdelta}
\dK(T_2,T_2^*)\geq F_{T_2^*}(t)-F_{T_2}(t)
=
\Pr(T_2^*\leq t<T_2)-\Pr(T_2\leq t<T_2^*)
\geq \Pr(A_\delta)-\Pr(B_\delta).
\end{equation}
Moreover,
\[
\Pr(A_\delta)=\E\bigl(\Ind_{\{S\leq \delta/2\}}\Pr(\Lambda=\Pi \, | \, \Pi,Y^*)\bigr)=\frac{1}{N}\Pr(S\leq \delta/2)
=\frac{1}{N}\bigl\{1-(1-\delta/2)^{(n-2)/2}\bigr\}.
\]
Since $\delta\leq 1/n$, we have $(n-2)\delta/4\leq 1/6$. Further as $1-e^{-u}\leq u$ for all $u\geq 0$ and $1-e^{-u}\geq u/2$ for $u\in[0,\log 2]$, we have
\[
1-(1-\delta/2)^{(n-2)/2}\geq 1-e^{-(n-2)\delta/4}\geq \frac{(n-2)\delta}{8}.
\]
Thus
\[
\Pr(A_\delta)\geq \frac{(n-2)\delta}{8N}
=
\frac{(n-2)\delta}{4n(n-1)}=:c_n\delta.
\]

On $B_\delta$, we must have $S<\delta$, because on $\{S\geq\delta\}$, we have $Y=Y^*$ and thus $T_2=T^*_2$. Also, $\Lambda\neq \Pi$ because otherwise $T_2=\delta$.
Moreover, for every distinct $r,s \in [n]$,
\[
\bigl|(Y_r+Y_s)-(Y_r^*+Y_s^*)\bigr|\leq \delta-S\leq \delta.
\]
Hence on $B_\delta$,
\[
\frac{\delta}{2}<T_2^*\leq \frac{3\delta}{2}.
\]
Conditional on $\Pi$, there are two possibilities for $\Lambda\neq \Pi$:

{Case 1: $|\Lambda \cap \Pi| = 1$.}  
There are $2(n-2)$ such pairs. In that case, since $S<\delta$ and $T_2^*\leq 3\delta/2$, we have $\sum_{i \in \Lambda \cup \Pi} Y_i^* < 5\delta/2$.  Now, for any three distinct indices $r,s,t \in [n]$,
\[
Y_r^* + Y_s^* + Y_t^* \sim \mathrm{Beta}\Bigl(\frac{3}{2},\frac{n-3}{2}\Bigr).
\]
Since $n\geq 6$, we have $(n-5)/2>0$, so by a union bound,
\begin{align*}
\Pr\bigl(\{|\Lambda \cap \Pi| = 1\} \cap B_\delta \, | \, \Pi\bigr)
&\leq
\frac{2(n-2)}{N\, \mathrm{B}\bigl(\frac{3}{2},\frac{n-3}{2}\bigr)} \int_0^{5\delta/2} x^{1/2}(1-x)^{(n-5)/2} \, dx \\ &\leq \frac{2(n-2)}{N} \,
\frac{(5\delta/2)^{3/2}}{(3/2)\mathrm{B}\bigl(\frac{3}{2},\frac{n-3}{2}\bigr)} =: C_n \delta^{3/2}.
\end{align*}

{Case 2: $|\Lambda \cap \Pi| = 0$.} There are $\binom{n-2}{2}$ such pairs. In this case, since $S<\delta$ and $T_2^*\leq 3\delta/2$, we have $\sum_{i \in \Lambda \cup \Pi} Y_i^* < 5\delta/2$.  Now, for any four distinct indices $r,s,t,u \in [n]$,
\[
Y_r^* + Y_s^* + Y_t^* + Y_u^* \sim \mathrm{Beta}\Bigl(2,\frac{n-4}{2}\Bigr).
\]
Since $n\geq 6$, we have $(n-6)/2\geq 0$, so  
\[
\Pr\bigl(\{|\Lambda \cap \Pi| = 0\} \cap B_\delta \,| \,\Pi)
\leq
\frac{\binom{n-2}{2}}{N}\,
\frac{(5\delta/2)^2}{2 \mathrm{B}\bigl(2,\frac{n-4}{2}\bigr)} =: D_n \delta^2.
\]

Combining the two cases, we see that
\[
\Pr(B_\delta)=\E\bigl(\Pr(B_\delta \, | \, \Pi)\bigr) \leq C_n\,\delta^{3/2}+D_n\delta^2\leq (C_n+D_n)\delta^{3/2}.
\]
Define
\[
\delta_n:=\min\biggl\{\frac{1}{n},\biggl(\frac{c_n}{2(C_n+D_n)}\biggr)^2\biggr\},
\]
so for every $\delta \in (0,\delta_n]$,
\[
\Pr(B_\delta)\leq \frac{c_n}{2}\delta,
\]
and thus from~\eqref{Eq:AdeltaBdelta},
\[
\dK(\mathsf{F},\mathsf{F}^*)=\dK(T_2,T_2^*)\geq \frac{c_n}{2}\delta \geq \frac{c_n}{2^{5/4}(n-2)^{1/2}}\W_1(nT,nT^*)^{1/2}.
\]
Therefore, the exponent $1/2$ can be attained even when $X$ and $\varepsilon$ are independent, and each has exchangeable components.
\end{example}

\subsection{Proofs of results in Section~\ref{sec:local-dK}} 

\begin{proof}[Proof of Theorem~\ref{thm:dK-alpha}]
The result follows directly from Theorem~\ref{thm:dK-W1} if $n - p = 1$, so suppose henceforth that $n - p \geq 2$. Given any coupling of $T$ and $T^*$, define $\Delta := T - T^*$. We argue similarly to the proof of Theorem~\ref{thm:dK-W1} but instead bound the local modulus of continuity~\eqref{eq:G-increments} of the distribution function of $V^* := T_2^*/(1 - T_1^*) \sim \mathrm{Beta}(\frac{p - p_0}{2}, \frac{n - p}{2})$, as opposed to its global modulus of continuity in~\eqref{eq:W-prob}. For fixed $\alpha' \in (0,\alpha]$, the $(1 - \alpha')$-quantile $q_{\alpha'}^*$ of $T^*$ satisfies
\[
1 - \alpha' = \Pr(\mathsf{F}^* \leq q_{\alpha'}^*) = \Pr\Bigl(\frac{T_2^*}{T_3^*} \leq \frac{(p - p_0)q_{\alpha'}^*}{n - p}\Bigr) = \Pr\Bigl(V^* \leq \frac{(p - p_0)q_{\alpha'}^*}{n - p + (p - p_0)q_{\alpha'}^*}\Bigr),
\]
so $s_{\alpha'} := (p - p_0)q_{\alpha'}^*/\bigl(n - p + (p - p_0)q_{\alpha'}^*\bigr)$ is the $(1 - \alpha')$-quantile of $V^* \sim \mathrm{Beta}(\frac{p - p_0}{2},\frac{n - p}{2})$.  The function $\alpha' \mapsto s_{\alpha'}$ is decreasing on $(0,\alpha]$.  It therefore follows similarly to~\eqref{eq:dK-Delta} in the proof of Theorem~\ref{thm:dK-W1} that for $\delta > 0$,
\begin{align}
d_{\mathrm{K},\alpha}(\mathsf{F},\mathsf{F}^*) &= \sup_{t \geq q_\alpha^*}\,\bigl|\Pr(\mathsf{F} \leq t) - \Pr(\mathsf{F}^* \leq t)\bigr| \notag \\
\label{eq:dK-alpha-Delta}
&\leq \max_{\epsilon \in \{-1,1\}} \sup_{s \in (s_\alpha,1)}\biggl|\,\Pr\Bigl(V^* \leq s + \frac{\epsilon\delta}{1 - T_1^*}\Bigr) - \Pr(V^* \leq s)\biggr| + \Pr(2^{1/2}\norm{\Delta}_2 \geq \delta).
\end{align}
First consider the case $p_0 \in \mathbb{N}$. For $s \in (s_\alpha,1)$ and $\epsilon \in \{-1,1\}$, letting $v := 1 - s$ and recalling from~\eqref{eq:dirichlet-indep} that $V^*$ and $T_1^*$ are independent, we have
\begin{align}
&\biggl|\Pr\Bigl(V^* \leq s + \frac{\epsilon\delta}{1 - T_1^*}\Bigr) - \Pr(V^* \leq s)\biggr| \notag \\
&\hspace{1cm}= \int_0^1\,\Bigl|\Pr\Bigl(V^* \leq s + \frac{\epsilon\delta}{1 - w}\Bigr) - \Pr(V^* \leq s)\Bigr| \cdot \frac{w^{\frac{p_0}{2} - 1}(1 - w)^{\frac{n - p_0}{2} - 1}}{B_0}\,dw \notag \\
\label{eq:local-dK-1}
&\hspace{1cm}= \int_0^1\,\Bigl|\Pr\Bigl(1 - V^* \geq v - \frac{\epsilon\delta}{w}\Bigr) - \Pr(1 - V^* \geq v)\Bigr| \cdot \frac{w^{\frac{n - p_0}{2} - 1}(1 - w)^{\frac{p_0}{2} - 1}}{B_0}\,dw,
\end{align}
where we recall that $B_0 = \mathrm{B}(\frac{p_0}{2},\frac{n - p_0}{2})$. The density of $1 - V^* \sim \mathrm{Beta}(\frac{n - p}{2},\frac{p - p_0}{2})$ is given by
\[
g(v) = \frac{v^{\frac{n - p}{2} - 1}(1 - v)^{\frac{p - p_0}{2} - 1}}{B}
\]
for $v \in (0,1)$, where $B = \mathrm{B}(\frac{n - p}{2},\frac{p - p_0}{2})$.  Further, if $p - p_0 \geq 2$, the associated distribution function $G$ satisfies
\begin{equation}
\label{eq:G-increments}
|G(v + h) - G(v)| \leq |h|\sup_{v' \leq v + |h|} g(v') \leq \frac{|h|(v + |h|)^{\frac{n - p}{2} - 1}}{B} \lesssim_{n,p,p_0} v^{\frac{n - p}{2} - 1}|h| +  |h|^{\frac{n - p}{2}}
\end{equation}
for $h \in \R$. Now consider the case $p - p_0 = 1$ and suppose that $v < G^{-1}(1/2)$. If $|h| \leq G^{-1}(3/4) - G^{-1}(1/2)$, then $v + |h| < G^{-1}(3/4)$ and 
\[
|G(v + h) - G(v)| \leq |h| \cdot \frac{(v + |h|)^{\frac{n - p}{2} - 1}\bigl(1 - G^{-1}(3/4)\bigr)^{-1/2}}{B} \lesssim_{n,p,p_0} v^{\frac{n - p}{2} - 1}|h| +  |h|^{\frac{n - p}{2}},
\]
whereas if $|h| > G^{-1}(3/4) - G^{-1}(1/2)$, then $|G(v + h) - G(v)| \leq 1 \lesssim_{n,p,p_0} |h|^{\frac{n - p}{2}}$. Therefore,~\eqref{eq:G-increments} also holds in this case. Next, $v_\alpha := 1 - s_\alpha = G^{-1}(\alpha)$ satisfies 
\[
\alpha = G(v_\alpha) = \int_0^{v_\alpha} \frac{v^{\frac{n - p}{2} - 1}(1 - v)^{\frac{p - p_0}{2} - 1}}{B}\,dv \geq \frac{(2^{- \frac{p - p_0}{2} + 2}\wedge2)\,v_\alpha^{\frac{n - p}{2}}}{(n - p)B},
\]
so
\begin{equation}
\label{eq:beta-quantile}
v_\alpha \lesssim_{n,p,p_0} \alpha^{2/(n - p)}.
\end{equation}
Since $\alpha \in (0,1/2)$ by assumption and hence $v_\alpha < G^{-1}(1/2)$, it follows that if $p_0 \in \mathbb{N}$, then by~\eqref{eq:local-dK-1},~\eqref{eq:G-increments} and~\eqref{eq:beta-quantile},
\begin{align}
\max_{\epsilon \in \{-1,1\}} &\sup_{s \in (s_\alpha,1)}\,\biggl|\,\Pr\Bigl(V^* \leq s + \frac{\epsilon\delta}{1 - T_1^*}\Bigr) - \Pr(V^* \leq s)\biggr| \notag \\
&\leq \max_{\epsilon \in \{-1,1\}} \int_0^1\,\sup_{v \in (0,v_\alpha)} \Bigl|G\Bigl(v - \frac{\epsilon\delta}{w}\Bigr) - G(v)\Bigr| \cdot \frac{w^{\frac{n - p_0}{2} - 1}(1 - w)^{\frac{p_0}{2} - 1}}{B_0}\,dw \notag \\
&\lesssim_{n,p,p_0} \int_0^1\,\sup_{v \in (0,v_\alpha)} \biggl\{\frac{v^{\frac{n - p}{2} - 1}\delta}{w} + \Bigl(\frac{\delta}{w}\Bigr)^{\frac{n - p}{2}}\biggr\} \cdot \frac{w^{\frac{n - p_0}{2} - 1}(1 - w)^{\frac{p_0}{2} - 1}}{B_0}\,dw \notag \\
&\leq v_\alpha^{\frac{n - p}{2} - 1}\delta \int_0^1 \frac{w^{\frac{n - p_0}{2} - 2}(1 - w)^{\frac{p_0}{2} - 1}}{B_0}\,dw + \delta^{\frac{n - p}{2}} \int_0^1 \frac{w^{\frac{p - p_0}{2} - 1}(1 - w)^{\frac{p_0}{2} - 1}}{B_0}\,dw \notag \\
\label{eq:local-dK-2}
&\lesssim_{n,p,p_0} (\alpha^{1 - \frac{2}{n - p}}\delta) \vee \delta^{\frac{n - p}{2}},
\end{align}
where the right-hand side is finite because $p - p_0 \geq 1$, $n - p \geq 2$ and hence $n - p_0 \geq 3$. If instead $p_0 = 0$, then $T_1^* = 0$ and it follows directly from~\eqref{eq:G-increments} and~\eqref{eq:beta-quantile} that
\begin{align*}
\max_{\epsilon \in \{-1,1\}} \sup_{s \in (s_\alpha,1)} |\Pr(V^* \leq s + \epsilon\delta) - \Pr(V^* \leq s)| &= \max_{\epsilon \in \{-1,1\}} \sup_{v \in (0,v_\alpha)} |G(v - \epsilon\delta) - G(v)| \\
&\lesssim_{n,p,p_0} (\alpha^{1 - \frac{2}{n - p}}\delta) \vee \delta^{\frac{n - p}{2}},
\end{align*}
so~\eqref{eq:local-dK-2} remains a valid bound. Finally, taking an optimal coupling of $T$ and $T^*$ for which $\E(\|\Delta\|_2) = \W_1(T,T^*)$, we deduce from~\eqref{eq:dK-alpha-Delta},~\eqref{eq:local-dK-2} and Markov's inequality that
\begin{align*}
d_{\mathrm{K},\alpha}(\mathsf{F},\mathsf{F}^*) &\leq \inf_{\delta > 0}\,\biggl\{\max_{\epsilon \in \{-1,1\}}\sup_{s \in (s_\alpha,1)}\,\biggl|\Pr\Bigl(V^* \leq s + \frac{\epsilon\delta}{1 - T_1^*}\Bigr) - \Pr(V^* \leq s)\biggr| + \Pr(2^{1/2}\norm{\Delta}_2 \geq \delta)\biggr\} \\
&\lesssim_{n,p,p_0} \inf_{\delta > 0}\,\biggl\{(\alpha^{1 - \frac{2}{n - p}}\delta) \vee \delta^{(n - p)/2} + \frac{2^{1/2}}{\delta}\E(\|\Delta\|_2)\biggr\} \\
&\lesssim_{n,p,p_0}
\max\bigl\{\alpha^{\frac{n - p - 2}{2(n - p)}}\W_1(nT,nT^*)^{1/2},\,\W_1(nT,nT^*)^{\frac{n - p}{n - p + 2}}\bigr\}
\end{align*}
when $n - p \geq 2$, as required.
\end{proof}

\begin{proof}[Proof of Proposition~\ref{prop:dK-alpha-lb}]
The result in the case $n-p=1$ follows from the proof of Proposition~\ref{prop:dK-W1-lb}\emph{(b)} (provided we also choose $\delta < v_\alpha$), so we assume henceforth that $n-p \geq 2$.  Using the notation in the proof of Theorem~\ref{thm:dK-alpha}, define
\[
v' := \frac{v_\alpha}{2} \wedge \frac{n - p}{4(n - p_0)},
\]
and fix $c := 1 - v' - \delta/2$ and $h := \delta/2$ for some $\delta \in (0,v')$.  Then $c - h \geq 1 - v_\alpha = s_\alpha$, so similarly to the proof of Proposition~\ref{prop:dK-W1-lb}\textit{(b)}, it suffices to obtain a lower bound on
\[
\eta = \Pr(c - h < V^* < c + h) = G(v' + \delta) - G(v')
\]
in~\eqref{eq:exch-lb-dK}. Then 
\[
c - h - \frac{p - p_0}{n - p_0} = \frac{n - p}{n - p_0} - v' - \delta > \frac{n - p}{4(n - p_0)}, 
\]
so in Lemma~\ref{lem:simplex-proj},
\[
K = \Bigl(\frac{n}{n - p_0}\Bigr)^2 \cdot \frac{2}{\inf_{r \in (c-h,c+h)} |r - \frac{p - p_0}{n - p_0}|} = \Bigl(\frac{n}{n - p_0}\Bigr)^2 \cdot \frac{2}{c - h - \frac{p - p_0}{n - p_0}} \lesssim_{n,p,p_0} 1.
\]
Moreover, 
\[
\alpha = G(v_\alpha) = \int_0^{v_\alpha} \frac{v^{\frac{n - p}{2} - 1}(1 - v)^{\frac{p - p_0}{2} - 1}}{B}\,dv \leq 
\frac{2v_\alpha^{\frac{n - p}{2}}(1 - v_\alpha)^{\frac{p - p_0}{2} - 1}}{(n - p)B},
\]
so $v_\alpha \gtrsim_{n,p,p_0} \alpha^{2/(n - p)}$. By definition, $v' \geq \frac{n - p}{2(n - p_0)}\,v_\alpha \gtrsim_{n,p,p_0} \alpha^{2/(n - p)}$. 
Thus, since $v' + \delta \leq (n - p)/(n - p_0)$ and the $\mathrm{Beta}(\frac{n - p}{2},\frac{p - p_0}{2})$ density $g$ is unimodal,
\begin{align}
\eta = G(v' + \delta) - G(v') &\geq \frac{\delta}{2}\Bigl\{g\Bigl(v' + \frac{\delta}{2}\Bigr) \wedge g(v' + \delta)\Bigr\} \notag \\[5pt]
&\geq \frac{\delta(v' + \delta/2)^{\frac{n - p}{2} - 1}}{2B}\Bigl(1 - \frac{n - p}{n - p_0}\Bigr)^{\frac{p - p_0}{2} - 1} \notag \\[5pt]
\label{eq:eta-lb}
&\gtrsim_{n,p,p_0} \delta(v' \vee \delta)^{\frac{n - p}{2} - 1}
\begin{cases}
= \delta^{\frac{n - p}{2}} = \delta^{1/2} \;\;&\text{if }n - p = 1 \\
\gtrsim_{n,p,p_0} (\alpha^{1 - \frac{2}{n - p}}\delta) \vee \delta^{\frac{n - p}{2}} \;\;&\text{if }n - p \geq 2.
\end{cases}
\end{align}
Now following the construction in the proof of Proposition~\ref{prop:dK-W1-lb}\textit{(b)}, which guarantees the independence of $X,\varepsilon$ and the exchangeability of $(X_1,\varepsilon_1),\dotsc,(X_n,\varepsilon_n)$, but with the above choices of $c,h \in (0,1)$, we see from~\eqref{eq:exch-lb-W1} that $\xi := \W_1(nT,nT^*) \lesssim_{n,p,p_0} \eta\delta$. Therefore, in all cases,
\[
\eta \gtrsim_{n,p,p_0}
\begin{cases}
(\xi/\eta)^{1/2} &\text{if }n - p = 1 \\
(\alpha^{1 - \frac{2}{n - p}}\xi/\eta) \vee (\xi/\eta)^{\frac{n - p}{2}} &\text{if }n - p \geq 2,
\end{cases}
\]
so by~\eqref{eq:exch-lb-dK} and the fact that $c - h \geq 1 - v_\alpha = s_\alpha$, we conclude that
\[
d_{\mathrm{K},\alpha}(\mathsf{F},\mathsf{F}^*) \gtrsim \eta \gtrsim_{n,p,p_0}
\begin{cases}
\xi^{1/3} \;\; &\text{if }n - p = 1 \\
(\alpha^{\frac{n - p - 2}{2(n - p)}}\xi^{1/2}) \vee \xi^{\frac{n - p}{n - p + 2}} \;\; &\text{if }n - p \geq 2.
\end{cases}
\]
This establishes the desired bound, and~\eqref{eq:exch-lb-W1} ensures for any $\eta' > 0$, we can always choose $c,h \in (0,1)$ such that $\xi = \W_1(nT,nT^*) < \eta'$. Recall also from~\eqref{eq:anticonservative} that $\Pr(\mathsf{F} > q) \geq \Pr(\mathsf{F}^* > q)$ for all $q > 1$.
\end{proof}

\subsection{Auxiliary results and proofs}

First, we establish the geometric result used to guarantee exchangeability in the proof of Proposition~\ref{prop:dK-W1-lb}\textit{(b)}.

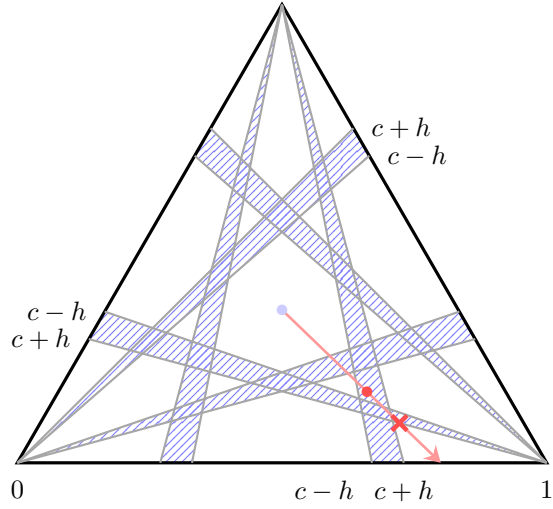
\begin{figure}[t]
\centering
\begin{tikzpicture}[every node/.style={font=\small}]

\def\wth{7.0}
\def\cph{0.73}
\def\cmh{0.67}
\def\cphs{0.27}
\def\cmhs{0.33}

\coordinate (BL) at (0, 0);
\coordinate (BR) at (\wth, 0);
\coordinate (AP) at (\wth/2, \wth*0.866);
\coordinate (C) at (\wth/2, \wth*0.289);
\coordinate (X) at (\wth*0.8, 0);

\coordinate (bot-cmh) at ({\wth*\cmh}, 0);
\coordinate (bot-cph) at ({\wth*\cph}, 0);

\coordinate (bot-cmhs) at ({\wth*\cmhs}, 0);
\coordinate (bot-cphs) at ({\wth*\cphs}, 0);

\coordinate (right-cmh) at (
{\wth + (\wth/2 - \wth)*\cmh},
{\wth*0.866*\cmh}
);
\coordinate (right-cph) at (
{\wth + (\wth/2 - \wth)*\cph},
{\wth*0.866*\cph}
);

\coordinate (right-cmhs) at (
{\wth + (\wth/2 - \wth)*\cmhs},
{\wth*0.866*\cmhs}
);
\coordinate (right-cphs) at (
{\wth + (\wth/2 - \wth)*\cphs},
{\wth*0.866*\cphs}
);

\coordinate (left-cmh) at (
{\wth/2 + (0 - \wth/2)*\cmh},
{\wth*0.866*(1-\cmh)}
);
\coordinate (left-cph) at (
{\wth/2 + (0 - \wth/2)*\cph},
{\wth*0.866*(1-\cph)}
);

\coordinate (left-cmhs) at (
{\wth/2 + (0 - \wth/2)*\cmhs},
{\wth*0.866*(1-\cmhs)}
);
\coordinate (left-cphs) at (
{\wth/2 + (0 - \wth/2)*\cphs},
{\wth*0.866*(1-\cphs)}
);

\path[name path=apex-cph]  (AP) -- (bot-cph);
\path[name path=apex-cmh]  (AP) -- (bot-cmh);
\path[name path=apex-cphs]  (AP) -- (bot-cphs);
\path[name path=apex-cmhs]  (AP) -- (bot-cmhs);
\path[name path=BL-rcph]   (BL) -- (right-cph);
\path[name path=BL-rcmh]   (BL) -- (right-cmh);
\path[name path=BL-rcphs]   (BL) -- (right-cphs);
\path[name path=BL-rcmhs]   (BL) -- (right-cmhs);
\path[name path=BR-lcph]   (BR) -- (left-cph);
\path[name path=BR-lcmh]   (BR) -- (left-cmh);
\path[name path=BR-lcphs]   (BR) -- (left-cphs);
\path[name path=BR-lcmhs]   (BR) -- (left-cmhs);

\path[name path=C-X]   (C) -- (X);
\path[name path=apex-c]  (AP) -- (\wth*0.69, 0);

\path[name intersections={of=apex-cph and BL-rcph,  by=I1}];
\path[name intersections={of=BL-rcph  and BR-lcph,  by=I2}];
\path[name intersections={of=BR-lcph  and apex-cph, by=I3}];

\path[name intersections={of=C-X  and BR-lcph, by=I4}];
\path[name intersections={of=C-X  and apex-c, by=I5}];

\fill[blue!20, pattern=north east lines, pattern color=blue!50]
(AP) -- (bot-cmh) -- (bot-cph) -- cycle;
\fill[blue!20, pattern=north east lines, pattern color=blue!50]
(AP) -- (bot-cmhs) -- (bot-cphs) -- cycle;

\fill[blue!20, pattern=north east lines, pattern color=blue!50]
(BL) -- (right-cmh) -- (right-cph) -- cycle;
\fill[blue!20, pattern=north east lines, pattern color=blue!50]
(BL) -- (right-cmhs) -- (right-cphs) -- cycle;

\fill[blue!20, pattern=north east lines, pattern color=blue!50]
(BR) -- (left-cmh) -- (left-cph) -- cycle;
\fill[blue!20, pattern=north east lines, pattern color=blue!50]
(BR) -- (left-cmhs) -- (left-cphs) -- cycle;

\draw[black, line width=1.2pt] (BL) -- (BR) -- (AP) -- cycle;

\draw[gray!70, line width=0.8pt] (AP) -- (bot-cmh);
\draw[gray!70, line width=0.8pt] (AP) -- (bot-cph);
\draw[gray!70, line width=0.8pt] (AP) -- (bot-cmhs);
\draw[gray!70, line width=0.8pt] (AP) -- (bot-cphs);
\draw[gray!70, line width=0.8pt] (BL) -- (right-cmh);
\draw[gray!70, line width=0.8pt] (BL) -- (right-cph);
\draw[gray!70, line width=0.8pt] (BL) -- (right-cmhs);
\draw[gray!70, line width=0.8pt] (BL) -- (right-cphs);
\draw[gray!70, line width=0.8pt] (BR) -- (left-cmh);
\draw[gray!70, line width=0.8pt] (BR) -- (left-cph);
\draw[gray!70, line width=0.8pt] (BR) -- (left-cmhs);
\draw[gray!70, line width=0.8pt] (BR) -- (left-cphs);

\draw[red!40, line width=1pt, -{Stealth[length=6pt,width=7pt]}] (C) -- (X);

\node[draw=red!70, cross out, line width=1.7pt, minimum size=4pt, inner sep=0pt] at (I4) {};
\fill[red!70] (I5) circle (2pt);
\fill[blue!20] (C) circle (2pt);

\node[below=3pt] at (BL)      {$0$};
\node[below left=3pt] at (bot-cmh) {$c-h$};
\node[below=3pt] at (bot-cph) {$c+h$};
\node[below=3pt] at (BR)      {$1$};

\node[right=3pt] at (right-cmh) {$c-h$};
\node[right=3pt] at (right-cph) {$c+h$};

\node[left=3pt] at (left-cmh) {$c-h$};
\node[left=3pt] at (left-cph) {$c+h$};

\end{tikzpicture}
\caption{Illustration of the configuration in Lemma~\ref{lem:simplex-proj} when $p_0 = 1$, $p = 2$ and $n = 3$, with the shaded region representing the subset $\triangle_n^\circ$ defined in~\eqref{eq:simplex-regions}. The red cross represents the image of the red dot under the projection map $\mathcal{T}$ defined in~\eqref{eq:Gn-simplex} onto the unshaded region $\triangle_n' = \triangle_n \setminus \triangle_n^\circ$.}
\label{fig:simplex-proj}
\end{figure}

\begin{proof}[Proof of Lemma~\ref{lem:simplex-proj}]
For every $\Pi \in \mathrm{Perm}(n)$, we have $r(\Pi u^*) = r(u^*) = (p - p_0)/(n - p_0) \notin (c- h, c + h)$.  Hence $u^* \in \triangle_n'$, and
\[
\mathcal{T}(\Pi u^*) = \mathcal{T}(u^*) = u^* = \Pi u^* = \Pi \mathcal{T}(u^*).
\]
Now fix $u \in \triangle_n \setminus \{u^*\}$. Since $u_\lambda := u + \lambda(u - u^*)/\norm{u - u^*}_1 \notin \triangle_n$ when $\lambda = 2$, we have $\lambda^*(u) \leq 2/\norm{u - u^*}_1 < \infty$. Let $H := \{(x_1,\dotsc,x_n) \in \R^n : \sum_{i=1}^n x_i = 1\}$. Then $\triangle_n^\circ$ and $H \setminus \triangle_n$ are open subsets of $H$ that partition $H \setminus \triangle_n'$, so $\mathcal{T}(u) = u_{\lambda^*(u)} \notin \triangle_n^\circ \cup (H \setminus \triangle_n)$, i.e.~$\mathcal{T}(u) \in \triangle_n'$. If $u \in \triangle_n'$, then $u \notin \triangle_n^\circ$, so $\lambda^*(u) = 0$ and $\mathcal{T}(u) = u$. For $\Pi \in \mathrm{Perm}(n)$ and $u \in \triangle_n$, we have
\[
(\Pi u)_\lambda = \Pi u + \frac{\lambda(\Pi u - u^*)}{\|\Pi u - u^*\|_1} = \Pi\Bigl(u + \frac{\lambda(u - u^*)}{\|u - u^*\|_1}\Bigr) = \Pi u_\lambda
\]
for all $\lambda \geq 0$, so because $\Pi u^* = u^*$ and $\Pi(\triangle_n^\circ) = \triangle_n^\circ$, it follows that $\mathcal{T}(\Pi u) = \Pi\mathcal{T}(u)$.

Next, we show that $u \mapsto \lambda^*(u)$ is lower semicontinuous on $\triangle_n \setminus \{u^*\}$. Given $\lambda_0 > 0$ and $u \in \triangle_n^\circ$ such that $\lambda^*(u) > \lambda_0$, we have $L_{\lambda_0}(u) := \{u_\lambda : \lambda \in [0,\lambda_0]\} \subseteq \triangle_n^\circ$. Since $L_{\lambda_0}(u)$ is a compact line segment and $\triangle_n^\circ$ is open in $\triangle_n$, there exists $\epsilon > 0$ such that $\{v \in \triangle_n : \norm{v - u_\lambda}_2 < \epsilon \text{ for some }\lambda \in [0,\lambda_0]\} \subseteq \triangle_n^\circ$. Therefore, there is an open ball $B \subseteq \triangle_n$ containing $u$ such that $L_{\lambda_0}(v) \subseteq \triangle_n^\circ$ for all $v \in B$, i.e.~$\lambda(v) > \lambda_0$ for such $v$. This shows that $\{u \in \triangle_n : \lambda^*(u) > \lambda_0\}$ is an open subset of $\triangle_n$ for every $\lambda_0 > 0$, so $u \mapsto \lambda^*(u)$ is lower semicontinuous on $\triangle_n \setminus \{u^*\}$ and hence $\mathcal{T}$ is Borel measurable on $\triangle_n$. 

Next, we will prove~\eqref{eq:Gn-proj-dist}. For $u = (u_1,\dotsc,u_n) \in (0,\infty)^n$, define $s(u) := \sum_{i=p_0+1}^p u_i$ and $t(u) := \sum_{i=p_0+1}^n u_i$, so that $r(u) = s(u)/t(u)$. For fixed $\Pi \in \mathrm{Perm}(n)$ and $u \in \triangle_n^\circ$, let $u' := \Pi u$ and $\lambda' := \sup\{\lambda \geq 0 : u_\lambda' \in \triangle_n\}$. Let $a := (p - p_0)/n$ and $b := (n - p_0)/n$. For $\lambda \in (-\|u-u^*\|_1,\lambda')$, we have $s(u_\lambda') = s(u') + \lambda\bigl(s(u') - a\bigr)/\norm{u - u^*}_1$ and $t(u_\lambda') = t(u') + \lambda\bigl(t(u') - b\bigr)/\norm{u - u^*}_1$, which lie in $(0,1)$. Thus,
\[
\theta(\lambda) := r(u_\lambda') = \frac{s(u') + \lambda\bigl(s(u') - a\bigr)/\norm{u - u^*}_1}{t(u') + \lambda\bigl(t(u') - b\bigr)/\norm{u - u^*}_1}
\]
and hence
\[
\theta'(\lambda) = \frac{t(u_\lambda')\bigl(s(u') - a\bigr) - s(u_\lambda')\bigl(t(u') - b\bigr)}{t(u_\lambda')^2 \norm{u - u^*}_1} 
= \frac{t(u')}{t(u_\lambda')^2 \norm{u - u^*}_1} \cdot b \Bigl(r(u') - \frac{a}{b}\Bigr).
\]
If $t(u') \geq b$, then $t(u')/t(u_\lambda')^2 \geq t(u') \geq b$ for all $\lambda \geq 0$. Otherwise, if $t(u') < b$, then $0 \leq t(u_\lambda') \leq t(u') < b$ for all $\lambda \in [0,\lambda')$, so $t(u')/t(u_\lambda')^2 \geq 1/t(u') \geq 1/b \geq b$. Thus, for any $\lambda \in [0,\lambda')$, we have
\[
\frac{t(u')}{t(u_\lambda')^2} \geq b.
\]

Suppose now that $a/b < c - h$, in which case
\[
K = \frac{2/b^2}{(c - h) - a/b}
\]
in~\eqref{eq:simplex-K}. If $u \in \triangle_{n,\Pi}$, then $\theta(0) = r(u') > c - h$ and
\begin{equation}
\label{eq:ratio-increasing}
\inf_{\lambda \in (0,\lambda')} \theta'(\lambda) \geq \frac{b^2}{\norm{u - u^*}_1} \Bigl(c - h - \frac{a}{b}\Bigr) > 0.
\end{equation}
Consequently, if $\lambda > Kh\norm{u - u^*}_1$, then
\[
r(u_\lambda') = \theta(\lambda) \geq \theta(0) + \lambda\inf_{\tilde\lambda \in (0,\lambda')} \theta'(\tilde\lambda) > (c - h) + 2h = c + h,
\]
so $u_\lambda' \notin \triangle_{n,\Pi}$. On the other hand, when $a/b > c + h$, we have
\[
K = \frac{2/b^2}{a/b - (c + h)}.
\]
Arguing similarly, if $u \in \triangle_{n,\Pi}$, then
\[
\sup_{\lambda \in (0,\lambda')} \theta'(\lambda) \leq \frac{b^2}{\|u-u^*\|_1}\Bigl(c + h - \frac{a}{b}\Bigr) < 0.
\]
Hence, if $\lambda > Kh\norm{u - u^*}_1$, then $r(u_\lambda') \leq \theta(0) + \lambda\sup_{\tilde\lambda \in (0,\lambda')} \theta'(\tilde\lambda) < (c + h) - 2h = c - h$, so $u_\lambda' \notin \triangle_{n,\Pi}$. Thus, $u_\lambda' = \Pi u_\lambda \notin \triangle_{n,\Pi}$ whenever $u \in \triangle_{n,\Pi}$ and $\lambda > Kh\norm{u - u^*}_1$. 

Now for $u \in \triangle_n$ and $\Pi \in \mathrm{Perm}(n)$, define
\[
I_\Pi(u) := \{\lambda \geq 0 : \Pi u_\lambda \in \triangle_{n,\Pi}\},
\]
which is an interval by the convexity of $\triangle_{n,\Pi}$; indeed, if $\lambda_1,\lambda_2 \in I_\Pi(u)$ and $t \in (0,1)$, then $\Pi u_{(1 - t)\lambda_1 + t\lambda_2} = \Pi\bigl((1 - t)u_{\lambda_1} + tu_{\lambda_2}\bigr) \in \triangle_{n,\Pi}$. If $I_\Pi(u) \neq \emptyset$ and $\lambda_0 := \inf I_\Pi(u)$, then $I_\Pi(u) = I_\Pi(u_{\lambda_0}) + \lambda_0$, and by the conclusion of the previous paragraph,
\[
\sup I_\Pi(u) - \inf I_\Pi(u) = \sup I_\Pi(u) - \lambda_0 = \sup I_\Pi(u_{\lambda_0}) \leq \norm{u_{\lambda_0} - u^*}_1 Kh \leq 2Kh. 
\]
By the definitions in~\eqref{eq:simplex-regions}, there exists a subset $\mathrm{Perm}(n)' \subseteq \mathrm{Perm}(n)$ of cardinality $\binom{n}{p}\binom{p}{p_0}$ such that $\triangle_n^\circ = \bigcup_{\Pi \in \mathrm{Perm}(n)'} \triangle_{n,\Pi}$ and
\[
\{\lambda \geq 0 : u_\lambda \notin \triangle_n^\circ\} = [0,\infty) \setminus \bigcup_{\Pi \in \mathrm{Perm}(n)'} I_\Pi(u).
\]
Thus, for every $u \in \triangle_n$, we have
\[
\norm{\mathcal{T}(u) - u}_1 =  \lambda^*(u)\norm{u - u^*}_1 = \inf\{\lambda \geq 0 : u_\lambda \notin \triangle_n^\circ\} \leq 2\binom{n}{p}\binom{p}{p_0}Kh,
\]
so~\eqref{eq:Gn-proj-dist} holds. Finally, if $u \in \triangle_n^\circ$ and $\Pi \in \mathrm{Perm}(n)$ satisfy $r(\Pi u) > (p - p_0)/(n - p_0) = a/b$, then by~\eqref{eq:ratio-increasing}, $\lambda \mapsto r(u_\lambda)$ is increasing on $(0,\lambda')$, so $r\bigl(\Pi\mathcal{T}(u)\bigr) \geq r(\Pi u)$. This bound also holds if $u \in \triangle_n'$, since then $\mathcal{T}(u) = u$.
\end{proof}

Next, we justify two conditions mentioned at the start of Section~\ref{sec:main-results}, under which $\beta^0$ is identifiable in a linear model $Y = X\beta^0 + \varepsilon$ for the joint distribution of $(X,Y)$ taking values in $\R^{n \times p} \times \R^n$. 

\begin{lemma}
\label{lem:linear-model-identifiable}
\begin{enumerate}[label=(\alph*), leftmargin=0.7cm]
\item  If $Xv$ is non-deterministic for all fixed $v \in \R^p \setminus \{0\}$, then there is at most one $\beta \in \R^p$ such that $X$ and $Y - X\beta$ are independent.
\item If $X$ has full column rank with positive probability, then there is at most one $\beta \in \R^p$ such that $X$ and $\varepsilon = (\varepsilon_1,\dotsc,\varepsilon_n) := Y - X\beta$ are independent, and $\varepsilon_i \eqd -\varepsilon_i$ for all $i \in [n]$.
\end{enumerate}
\end{lemma}

\begin{proof}
\textit{(a)} Suppose that $X$ and $\varepsilon := Y - X\beta$ are independent. Then for $\beta' \neq \beta$, we write $Y - X\beta' = \varepsilon + V$, where $V := X(\beta - \beta')$ is independent of $\varepsilon$ and non-deterministic by assumption. Therefore, by Lemma~\ref{lem:indep-sum} below, $Y - X\beta'$ is not independent of $V$ and hence $X$, so $\beta$ is the unique vector with the specified property.

\medskip
\noindent \textit{(b)} Suppose that $\varepsilon = Y - X\beta$ satisfies the required conditions for some $\beta \in \R^p$. For $\beta' \neq \beta$, the argument in \textit{(a)} shows that if $\varepsilon' = (\varepsilon_1',\dotsc,\varepsilon_n') := Y - X\beta'$ is independent of $X$, then $\varepsilon' - \varepsilon = X(\beta - \beta') = v$ almost surely for some deterministic $v = (v_1,\dotsc,v_n) \in \R^n$. Since $X$ has full rank with positive probability, we must have $v \neq 0$, i.e.~$v_i \neq 0$ for some $i \in [n]$. Now $\varepsilon_i \eqd -\varepsilon_i$ has median 0, so $\varepsilon_i' = v_i + \varepsilon_i$ has median $v_i \neq 0$, and hence does not have a symmetric distribution. This again establishes the uniqueness of $\beta$.
\end{proof}

The proof of Lemma~\ref{lem:linear-model-identifiable} relies on the following elementary probabilistic fact.

\begin{lemma}
\label{lem:indep-sum}
If $V$ and $W$ are independent $\R^n$-valued random vectors such that $V$ is non-deterministic, then $V$ and $V + W$ are not independent.
\end{lemma}

Here, $V$ and $W$ need not be integrable, so $\Cov(V,W)$ and $\E(W\,|\,V)$ may not be well-defined.

\begin{proof}
The characteristic functions of $V$ and $W$ are given by $\varphi_V(t) := \E(e^{it^\top V})$ and $\varphi_W(t) := \E(e^{it^\top W})$ respectively for $t \in \R^n$. Suppose for a contradiction that $V$ and $V + W$ are independent. Then for $s,t \in \R^n$, we have
\begin{align*}
\varphi_V(s + t)\,\varphi_W(t) = \E(e^{i(s + t)^\top V})\,\E(e^{it^\top W}) &= \E(e^{i(s + t)^\top V} \cdot e^{it^\top W}) = \E(e^{is^\top V} \cdot e^{it^\top(V + W)}) \\
&= \E(e^{is^\top V})\,\E(e^{it^\top(V + W)}) \\
&= \E(e^{is^\top V})\,\E(e^{it^\top V})\,\E(e^{it^\top W}) = \varphi_V(s)\,\varphi_V(t)\,\varphi_W(t),
\end{align*}
where the second and penultimate equalities hold by the independence of $V$ and $W$. Since $\varphi_W(0) = 1$ and $\varphi_W$ is continuous on $\R^n$, there exists $\epsilon > 0$ such that $\varphi_W(t) \neq 0$ whenever $\norm{t}_2 \leq \epsilon$, in which case
\begin{equation}
\label{eq:cauchy-functional-eqn}
\varphi_V(s + t) = \varphi_V(s)\,\varphi_V(t)
\end{equation}
for all $s \in \R^n$. For general $s,t \in \R^n$, choose $m \in \N$ for which $\norm{t}_2/m \leq \epsilon$, so that by~\eqref{eq:cauchy-functional-eqn}, $\varphi_V(t) = \varphi_V\bigl(m \cdot (t/m)\bigr) = \varphi_V(t/m)^m$ and hence
\[
\varphi_V(s + t) = \varphi_V\Bigl(s + m \cdot \frac{t}{m}\Bigr) = \varphi_V(s)\,\varphi_V\Bigl(\frac{t}{m}\Bigr)^m = \varphi_V(s)\,\varphi_V(t).
\]
Thus,~\eqref{eq:cauchy-functional-eqn} holds for all $s,t \in \R^n$. Since the characteristic function $\varphi_V$ is continuous and $\varphi_V(t) \neq 0$ whenever $\norm{t}_2 \leq \epsilon$, there exists $z \in \C^n$ such that $\varphi_V(t) = e^{z^\top t}$ for all $t \in \R^n$; see for instance the proofs of Theorems~5.5.2 and~13.1.4 of \citet{kuczma2009introduction}. Writing $z = u + iv$ with $u,v \in \R^n$, we have
\[
e^{u^\top t} = |e^{z^\top t}| = |\varphi_V(t)| \leq 1
\]
for all $t \in \R^n$, so $u = 0$. We conclude that there exists $v \in \R^n$ such that $\varphi_V(t) = e^{iv^\top t}$ for all $t \in \R^n$, so by the uniqueness of characteristic functions, $V = v$ almost surely. This contradicts the hypothesis that $V$ is non-deterministic, so the proof is complete.
\end{proof}

Next, the appearance of the exponents $1/3$ and $1/2$ in Theorem~\ref{thm:dK-W1} can be understood via the following lemma, in conjunction with the bound~\eqref{eq:W-prob} on the global modulus of continuity of the $\mathrm{Beta}(\frac{p - p_0}{2},\frac{n - p}{2})$ distribution function; see the paragraph containing~\eqref{eq:holder} for further discussion.

\begin{lemma}
\label{lem:holder-L1-Linfty}
Let $F,G \colon \R \to [0,1]$ be increasing functions, and suppose that $G$ is $(\gamma,L)$-H\"older for some $\gamma \in (0,1]$ and $L > 0$. Then
\[
\norm{F - G}_\infty \leq L^{1/(\gamma + 1)}\Bigl(\frac{\gamma + 1}{\gamma} \norm{F - G}_1\Bigr)^{\gamma/(\gamma + 1)}.
\]
Equality holds if $G(x) = Lx^\gamma \wedge 1$ and $F(x) = G(x) \vee d$ for $x \geq 0$, where $d \in (0,1)$, and $F(x) = G(x) = 0$ for $x < 0$. 
\end{lemma}

\begin{proof}
Fix any $x \in \R$ and suppose first that $d := F(x) - G(x) \geq 0$. Then $F(x + h) \geq G(x) + d$ for all $h \geq 0$ and $G(x + h) \leq G(x) + Lh^\gamma$, so
\[
\norm{F - G}_1 \geq \int_x^{x + (d/L)^{1/\gamma}} |F - G| \geq \int_0^{(d/L)^{1/\gamma}} (d - Lh^\gamma)\,dh = \frac{\gamma}{\gamma + 1} \cdot \frac{|F(x) - G(x)|^{(\gamma + 1)/\gamma}}{L^{1/\gamma}}.
\]
Similarly, if $d < 0$, then considering $h \in \bigl(-(|d|/L)^{1/\gamma},0\bigr)$ yields the same conclusion. The first assertion of the result follows by taking a supremum over $x \in \R$.

If $F$ and $G$ are as in the second part of the lemma, then equality holds throughout the previous display because $d = F(0) - G(0)$ and 
\[
F(y) - G(y) = (d - Ly^\gamma)\Ind_{\{y \in [0,(d/L)^{1/\gamma}]\}} \geq 0
\]
for $y \in \R$. Therefore, the conclusion of the lemma holds with equality.
\end{proof}

The following probabilistic fact is used in the proof of Proposition~\ref{prop:dK-W1-lb}\textit{(b)}.

\begin{lemma}
\label{lem:dirichlet-conditional}
Given $n \in \N$, let $Y^* = (Y_1^*,\dotsc,Y_n^*) \sim \mathrm{Dirichlet}(1/2,\dotsc,1/2)$ and define $S_i := \sum_{j=1}^i Y_j^*$ for $i \in [n]$. Let $\ell,m \in [n]$ be such that $\ell < m$. Then for any $a > 0$ and $c,h \in (0,1)$ such that $h < \min(c, 1 - c)/2 =: d$, we have
\[
\Pr\biggl(\min_{i \in [n]} Y_i^* \leq ah \biggm| c - h < \frac{S_\ell}{S_m} < c + h\biggr) \lesssim_{\ell,m,n} \Bigl(\frac{ah}{d}\Bigr)^{1/2}.
\]
\end{lemma}

\begin{proof}
First, if $Z \sim \mathrm{Beta}(1/2,b/2)$ for $b > 0$, then for all $t > 0$, we have
\[
\Pr(Z \leq t) = \int_0^t \frac{x^{-1/2}(1 - x)^{b/2-1} }{B}\,dx \leq \zeta_b t^{1/2},
\]
where $\zeta_b > 0$ depends only on $b$. Moreover, $Y^*$ is an exchangeable random vector with $S_i = \sum_{j=1}^i Y_j^* \sim \mathrm{Beta}(\frac{i}{2},\frac{n - i}{2})$ for each $i \in [n-1]$ and $\E(S_i^{-1/2}) = \mathrm{B}(\frac{i - 1}{2},\frac{n - i}{2})/\mathrm{B}(\frac{i}{2},\frac{n - i}{2}) < \infty$ for $i \geq 2$. We also have $c/2 < c - h < c + h < (1 + c)/2$. Since $S_\ell/S_m \sim \mathrm{Beta}(\frac{\ell}{2},\frac{m - \ell}{2})$, this random variable lies in $(c - h, c + h)$ with positive probability.

Now by \citet[Proposition~G.3]{ghosal2017fundamentals}, $Y_{[\ell]}^*/S_\ell$, $S_\ell/S_m$ and $S_m \sim \mathrm{Beta}(\frac{m}{2},\frac{n - m}{2})$ are independent for $m \geq 2$ and $\ell \in [m-1]$ and $Y_\ell^*/S_\ell \sim \mathrm{Beta}(\frac{1}{2},\frac{\ell-1}{2})$ for $\ell \in \{2,\ldots,m-1\}$. Thus, if $h < d$, then for $\ell \in \{2,\ldots,m-1\}$,
\begin{align}
\Pr\biggl(\min_{i \in [\ell]} Y_i^* \leq ah \,\biggm| \Bigl|\frac{S_\ell}{S_m} - c\Bigr| < h\biggr) &= \Pr\biggl(\min_{i \in [\ell]} \frac{Y_i^*}{S_\ell} \leq \frac{S_m}{S_\ell} \cdot \frac{ah}{S_m} \,\biggm| \Bigl|\frac{S_\ell}{S_m} - c\Bigr| < h\biggr) \notag \\
&\leq \Pr\biggl(\min_{i \in [\ell]} \frac{Y_i^*}{S_\ell} \leq \frac{2ah}{cS_m} \,\biggm| \Bigl|\frac{S_\ell}{S_m} - c\Bigr| < h\biggr) \notag \\
&= \Pr\biggl(\min_{i \in [\ell]} \frac{Y_i^*}{S_\ell} \leq \frac{2ah}{cS_m}\biggr) \notag \\
\label{eq:dirichlet-conditional-1}
&\leq \ell \cdot \zeta_{\ell - 1}\,\E\Bigl\{\Bigl(\frac{2ah}{cS_m}\Bigr)^{1/2}\Bigr\} \lesssim_{\ell,m,n} \Bigl(\frac{ah}{c}\Bigr)^{1/2}.
\end{align}
A slightly simpler argument yields the same conclusion when $\ell=1$.  Similarly, $Y_{(\ell:m]}^*/(S_m - S_\ell)$, $S_\ell/S_m$ and $S_m$ are independent with $Y_m^*/(S_m - S_\ell) \sim \mathrm{Beta}(\frac{1}{2},\frac{m - \ell-1}{2})$ when $\ell \in [m-2]$, so if $h < d$, then
\begin{align}
\Pr\biggl(\min_{i \in (\ell:m]} Y_i^* \leq ah \,\biggm| \Bigl|\frac{S_\ell}{S_m} - c\Bigr| < h\biggr) &= \Pr\biggl(\min_{i \in (\ell:m]} \frac{Y_i^*}{S_m - S_\ell} \leq \frac{S_m}{S_m - S_\ell} \cdot \frac{ah}{S_m} \,\biggm| \Bigl|\frac{S_\ell}{S_m} - c\Bigr| < h\biggr) \notag \\
&\leq \Pr\biggl(\min_{i \in (\ell:m]} \frac{Y_i^*}{S_m - S_\ell} \leq \frac{2ah}{(1 - c)S_m}\biggr) \notag \\
&\leq (m - \ell)\zeta_{m - \ell-1}\,\E\Bigl\{\Bigl(\frac{2ah}{(1 - c)S_m}\Bigr)^{1/2}\Bigr\} \lesssim_{\ell,m,n} \Bigl(\frac{ah}{1 - c}\Bigr)^{1/2}.
\end{align}
Again, the same final conclusion holds when $\ell = m-1$.  Finally, if $m < n$, then $S_\ell/S_m$ is independent of $\bigl(1 - S_m,Y_{(m:n]}^*/(1 - S_m)\bigr)$ and hence $Y_{(m:n]}^*$, so if $h < d$, then
\begin{align}
\Pr\biggl(\min_{i \in (m:n]} Y_i^* \leq ah \,\biggm| \Bigl|\frac{S_\ell}{S_m} - c\Bigr| < h\biggr)
&= \Pr\Bigl(\min_{i \in (m:n]} Y_i^* \leq ah\Bigr) \notag \\
\label{eq:dirichlet-conditional-2}
&\leq (n - m)\,\Pr(Y_n^* \leq ah) \lesssim_{m,n} (ah)^{1/2}.
\end{align}
Thus, for $h < d$, summing~\eqref{eq:dirichlet-conditional-1}--\eqref{eq:dirichlet-conditional-2} yields
\[
\Pr\biggl(\min_{i \in [n]} Y_i^* \leq ah \,\biggm|\, \Bigl|\frac{S_\ell}{S_m} - c\Bigr| < h\biggr) \lesssim_{\ell,m,n} \Bigl(\frac{ah}{d}\Bigr)^{1/2}. \qedhere
\]
\end{proof}

\subsection{Asymptotic results}

\begin{proposition}
\label{Prop:Asymptotic}
Fix $p_0 \in \N_0,p \in \N$ with $p_0 < p$. For each $n>p$, let $\mathcal M_n$ be a
collection of joint distributions of $(X_n,\varepsilon_n)$ on $\R^{n \times p} \times \R^n$, where for every~$n$ and $\nu_n\in\mathcal M_n$, the design matrix~$X_n$ has full column rank almost surely and is independent of
$\varepsilon_n$. Suppose further that the rows
$X_{n1}^\top,\ldots,X_{nn}^\top$ of $X_n$ are independent and identically distributed, with
\[
\Sigma_{X,\nu_n}
:=
\E_{\nu_n}(X_{n1}X_{n1}^\top)
\]
being well-defined, and that $\varepsilon_n$ has independent components $\varepsilon_{n1},\ldots,\varepsilon_{nn}$ satisfying
\[
\E_{\nu_n}(\varepsilon_{ni})=0,
\qquad
\E_{\nu_n}(\varepsilon_{ni}^2)
=
\sigma_{\nu_n}^2\in(0,\infty)
\]
for $i \in [n]$. Assuming that
\begin{align}
\label{Eq:cX} &\ c_X:=\inf_{\nu_n\in\mathcal{M}_n}
\lambda_{\min}(\Sigma_{X,\nu_n})
> 0, \\
\label{Eq:UI}&
\lim_{M \rightarrow \infty} \limsup_{n \rightarrow \infty} \sup_{\nu_n\in\mathcal{M}_n}
\E_{\nu_n}\bigl(
\|X_{n1}\|_2^2
\mathbbm{1}_{\{\|X_{n1}\|_2^2 > M\}}\bigr)
 = 0, \\
\label{Eq:Lindeberg}&\lim_{M\to\infty} \limsup_{n \rightarrow \infty}\sup_{\nu_n\in\mathcal{M}_n}
\frac{1}{n}\sum_{i=1}^n
\E_{\nu_n}\biggl(
\frac{\varepsilon_{ni}^2}{\sigma_{\nu_n}^2}
\mathbbm{1}_{\{|\varepsilon_{ni}|/\sigma_{\nu_n}>M\}}
\biggr)
= 0,
\end{align}
we have under $H_0$ that
\[
\sup_{\nu_n\in\mathcal{M}_n}\W_1(nT,nT^*) \rightarrow 0 \quad \text{and hence} \quad \sup_{\nu_n\in\mathcal{M}_n}\dK(\mathsf{F},\mathsf{F}^*)\rightarrow 0.
\]
\end{proposition}

In the proof below, we write $\E_n:=\E_{\nu_n}$, $\Pr_n:=\Pr_{\nu_n}$, $\sigma_n^2:=\sigma_{\nu_n}^2$ and $\Sigma_{X,n}:=\Sigma_{X,\nu_n}$ for brevity. Moreover, we use the following notation: given a sequence of random vectors $(X_n)$, we write $X_n = o_{\mathcal{M}_n}(1)$ if for every $\epsilon > 0$, we have
\[
\sup_{\nu_n \in \mathcal{M}_n} \Pr_n\bigl(\norm{X_n}_2 > \epsilon\bigr) \to 0.
\]

\begin{proof}
Fix $\eta > 0$. By a triangular array version of the uniform weak law of large numbers \citep[][Theorem~A.11.13]{samworth24modern}, and the uniform integrability condition~\eqref{Eq:UI},
\[
\sup_{\nu_n\in\mathcal M_n}
\Pr_n\biggl(
\biggl\|
\frac{X_n^\top X_n}{n}-\Sigma_{X,n}
\biggr\|_{\mathrm{op}}>\eta
\biggr) \leq \sup_{\nu_n\in\mathcal M_n}
\Pr_n\biggl(
\biggl\|
\frac{X_n^\top X_n}{n}-\Sigma_{X,n}
\biggr\|_{\mathrm{F}}>\eta
\biggr)
\rightarrow 0.
\]
Then by Weyl's inequality \citep[][Theorem~A.10.17]{samworth24modern}, it follows that $\sup_{\nu_n\in\mathcal M_n}
\Pr_n\bigl(\lambda_{\min}(X_n^\top X_n/n)<c_X/2\bigr)\rightarrow 0$. On the other hand, by~\eqref{Eq:UI} again, 
\begin{align*}
\sup_{\nu_n\in\mathcal{M}_n}
\Pr_n\biggl(
\max_{i\in[n]}\frac{\|X_{ni}\|_2^2}{n}\geq\eta
\biggr) &\leq  \sup_{\nu_n\in\mathcal{M}_n}n\Pr_n\bigl(\|X_{n1}\|_2^2\geq n\eta\bigr) \\
&\leq \sup_{\nu_n\in\mathcal{M}_n} \frac{1}{\eta}
\E_n\bigl(
\|X_{n1}\|_2^2
\mathbbm{1}_{\{\|X_{n1}\|_2^2\geq n\eta\}}
\bigr)\rightarrow 0.
\end{align*}
With
\[
P_n:=X_n(X_n^\top X_n)^{-1}X_n^\top,
\qquad
h_n:=\max_{i\in[n]}(P_n)_{ii}=\max_{i\in[n]}X_{ni}^\top(X_n^\top X_n)^{-1}X_{ni},
\]
we deduce for every $\eta > 0$ that
\begin{align*}
\sup_{\nu_n \in \mathcal{M}_n} \mathbb{P}_n(h_n > \eta) &\leq \sup_{\nu_n \in \mathcal{M}_n} \mathbb{P}_n\biggl(\frac{
\max_{i\in[n]}\|X_{ni}\|_2^2/n
}{
\lambda_{\min}(X_n^\top X_n/n)} > \eta\biggr) \\
&\leq \sup_{\nu_n \in \mathcal{M}_n} \mathbb{P}_n\biggl(\lambda_{\min}(X_n^\top X_n/n) < \frac{c_X}{2}\biggr) + \sup_{\nu_n \in \mathcal{M}_n} \mathbb{P}_n\biggl(\max_{i\in[n]}\frac{\|X_{ni}\|_2^2}{n} > \frac{2\eta}{c_X}
\biggr) \rightarrow 0.
\end{align*}

Now consider the QR decomposition $X_n =QR$, where $R$ has positive diagonal entries, so  $P_n=QQ^{\top}$.  Then writing $Q_{i\cdot}$ for the $i$th row of $Q$, we have $\max_{i\in[n]} \|Q_{i\cdot}\|_2= h_n^{1/2} = o_{\mathcal{M}_n}(1)$.  Fix $t\in \mathbb{R}^{p}\setminus \{0\}$, and let $a_n = (a_{n1},\ldots,a_{nn}) :=Qt/\|t\|_2$, so that $\|a_n\|_2 = 1$ and $|a_{ni}| \leq h_n^{1/2}$. As the rows of $X$ are exchangeable, so are the rows of $Q$ and thus $\E_n(a^2_{ni})=1/n$ for $i\in [n]$. With $Y_n = (Y_{n1},\ldots,Y_{np}) :=Q^\top\varepsilon_n/\sigma_n$, we have $t^\top Y_n/\|t\|_2 = a_n^\top \varepsilon_n/\sigma_n$.  We verify the conditional Lindeberg condition as follows: given $\epsilon > 0$, by~\eqref{Eq:Lindeberg}, we can find $\delta > 0$ and $n_0 \in \mathbb{N}$ such that 
 \[
\sup_{\nu_n\in\mathcal M_n}
\frac{1}{n}\sum_{i=1}^n
\E_n\biggl(
\frac{\varepsilon_{ni}^2}{\sigma_n^2}
\mathbbm{1}_{\{|\varepsilon_{ni}|>
 \epsilon \sigma_{n}/\delta\}}
\biggr)
\leq \frac{\epsilon^2}{2}
\]
for $n \geq n_0$.  Now choose $n_1 \in \mathbb{N}$ large enough that $\sup_{\nu_n\in\mathcal M_n}\mathbb{P}_n(h_n^{1/2} > \delta) \leq \epsilon/2$ for $n \geq n_1$.  Then by Markov's inequality, for $n \geq \max(n_0,n_1)$,
\begin{align*}
\sup_{\nu_n\in\mathcal M_n}&\Pr_n\biggl(\frac{1}{\sigma^2_n}\sum_{i=1}^n
\E_n\bigl(
a_{ni}^2\varepsilon_{ni}^2
\mathbbm{1}_{\{
|a_{ni}\varepsilon_{ni}|>
\epsilon \sigma_n\}} \bigm| X_n
\bigr) > \epsilon\biggr) \\
&\leq \sup_{\nu_n\in\mathcal M_n}\Pr_n\biggl(\frac{1}{\sigma^2_n}\sum_{i=1}^n
a_{ni}^2\E_n\bigl(
\varepsilon_{ni}^2
\mathbbm{1}_{\{|\varepsilon_{ni}|>
\epsilon \sigma_n/h^{1/2}_n\}} \bigm| X_n
\bigr) > \epsilon\biggr) \\
&\leq \sup_{\nu_n\in\mathcal M_n}\Pr_n(h_n^{1/2}>\delta)+ \sup_{\nu_n\in\mathcal M_n}\frac{1}{\epsilon}\sum_{i=1}^n \E_n(a^2_{ni})\,\E_n\biggl(
\frac{\varepsilon_{ni}^2}{\sigma_n^2}
\mathbbm 1_{\{|\varepsilon_{ni}|>
\epsilon\sigma_{n}/\delta\}}
\biggr)\leq \epsilon.
\end{align*}
By a triangular array version of the uniform Lindeberg--Feller central limit theorem \citep[][Theorem~A.11.14]{samworth24modern}, we have
\[
\sup_{x \in \mathbb{R}}\;\biggl|\mathbb{P}_n\biggl(\frac{t^\top Y_n}{\|t\|_2} \leq x \biggm| X_n\biggr) - \Phi(x)\biggr| = o_{\mathcal{M}_n}(1).
\]
Since this limiting distribution does not depend on $X_n$, we deduce that 
\[
\sup_{\nu_n \in \mathcal{M}_n} \sup_{x \in \mathbb{R}}\;\biggl|\mathbb{P}_n\biggl(\frac{t^\top Y_n}{\|t\|_2} \leq x\biggr) - \Phi(x)\biggr| \rightarrow 0.
\]
It follows by the Cram\'er--Wold device that under any sequence $(\nu_n)$ with $\nu_n \in \mathcal{M}_n$, we have $Y_n \stackrel{d}{\rightarrow} Z$.  Together with the fact that under any sequence $(\nu_n)$ with $\nu_n \in \mathcal{M}_n$, we have
\[
\E\bigl(\|Y_n\|^2_2\bigr)=\E\bigl(\E(\|Y_n\|^2_2\, | \, X_n)\bigr)=\frac{1}{\sigma_n^2}\E\bigl(
\tr\bigl\{
Q^\top\E(\varepsilon_n\varepsilon_n^\top)Q
\bigr\}\bigr)=p=\E(\|Z\|^2_2),
\]
we conclude that $\W_2(Y_n,Z)\rightarrow 0$, and thus $\W_1(Y_n,Z)\rightarrow 0$. As this conclusion holds for each sequence $(\nu_n)$ with $\nu_n\in \mathcal{M}_n$ for each $n$, 
we have $\sup_{\nu_n\in\mathcal M_n}\W_1(Y_n,Z)\rightarrow 0$.
    
Now define $S:\mathbb{R}^p \rightarrow \mathbb{R}^p$ by $S(x_1,\ldots,x_p) := (x_1^2,\ldots,x_p^2)$, and set $B_n := S(Y_n)$, $C_Z := S(Z)$.  For each $\nu_n\in\mathcal{M}_n$, choose a sequence $(\tilde{Y}_n)$ and $\tilde{Z}$, defined on the same probability space, such that $\tilde{Y}_n \stackrel{d}{=} Y_n$, $\tilde{Z} \stackrel{d}{=} Z$ and $\mathbb{E}_n\bigl(\|\tilde{Y}_n - \tilde{Z}\|_2^2\bigr) = \W_2^2(Y_n,Z)$.  For $x = (x_1,\ldots,x_p)$, $\tilde{x} = (\tilde{x}_1,\ldots,\tilde{x}_p) \in \mathbb{R}^p$, we have
\[
\|S(x) - S(\tilde{x})\|_2 = \biggl\{\sum_{j=1}^p (x_j^2 - \tilde{x}_j^2)^2\biggr\}^{1/2} = \biggl\{ \sum_{j=1}^p (x_j - \tilde{x}_j)^2(x_j + \tilde{x}_j)^2\biggr\}^{1/2} \leq \|x-\tilde{x}\|_2\|x+\tilde{x}\|_2.
\]
Hence, by Cauchy--Schwarz,
\begin{align*}
\W_1(B_n,C_Z) &\leq \mathbb{E}_n\bigl(\|S(\tilde{Y}_n) - S(\tilde{Z})\|_2\bigr) \leq \mathbb{E}_n\bigl(\|\tilde{Y}_n -\tilde{Z}\|_2\|\tilde{Y}_n + \tilde{Z}\|_2\bigr) \\
&\leq \bigl\{\mathbb{E}_n\bigl(\|\tilde{Y}_n -\tilde{Z}\|_2^2\bigr)\mathbb{E}\bigl(\|\tilde{Y}_n+\tilde{Z}\|_2^2\bigr)\bigr\}^{1/2} \leq 2^{1/2}\W_2(Y_n,Z)\bigl\{\mathbb{E}_n(\|Y_n\|_2^2) + \mathbb{E}(\|Z\|_2^2)\bigr\}^{1/2} \\
&= 2p^{1/2}\W_2(Y_n,Z).
\end{align*}
Therefore, 
\[
\sup_{\nu_n\in\mathcal M_n}\W_1(B_n,C_Z)\leq 2p^{1/2}\sup_{\nu_n\in\mathcal M_n}\W_2(Y_n,Z)\rightarrow 0.
\]
Now let
\[
R_n:=\frac{\|\varepsilon_n\|^2_2}{n\sigma_n^2}, \quad
A_n:= S\biggl(\frac{Y_n}{R_n^{1/2}}\biggr),
\] 
and $\mathcal E_n:=\{R_n\geq1/2\}$.  On $\mathcal E_n$, we have $|R_n^{-1}-1|
\leq \min\bigl\{2|R_n-1|,1\bigr\}$, so for any $j\in[p]$ and $K>0$,
\[
\E_n\bigl(Y_{nj}^2|R_n^{-1}-1|
\mathbbm{1}_{\mathcal E_n}\bigr)\leq 2K\E_n(|R_n-1|)+\E_n\bigl(Y_{nj}^2\mathbbm{1}_{\{Y_{nj}^2>K\}}\bigr).
\]
By, e.g., \citet[][Lemma~B, p.~15]{serfling2009approximation}, $\sup_{n \in \mathbb{N}} \sup_{\nu_n \in \mathcal{M}_n} \mathbb{E}_n\bigl(Y_{nj}^2\mathbbm{1}_{\{Y_{nj}^2 \geq K\}}\bigr) \rightarrow 0$ as $K \rightarrow \infty$ and moreover, $\sup_{\nu_n\in\mathcal M_n}\E_n|R_n-1|\rightarrow 0$ by a uniform version of the $L^1$ law of large numbers.  It follows that
\[
\limsup_{n \rightarrow \infty} \sup_{\nu_n\in\mathcal M_n}\E_n\bigl(Y_{nj}^2|R_n^{-1}-1|
\mathbbm{1}_{\mathcal E_n}\bigr) \leq \lim_{K\rightarrow \infty} \limsup_{n \rightarrow \infty} \sup_{\nu_n\in\mathcal M_n}\E_n\bigl(Y_{nj}^2\mathbbm{1}_{\{Y_{nj}^2>K\}}\bigr) = 0.
\]

Next, by assumption~\eqref{Eq:Lindeberg}, there exists $M>0$ such that
\[
\limsup_{n \rightarrow \infty}\sup_{\nu_n\in\mathcal M_n} \frac{1}{n}\sum_{i=1}^n \E_n\biggl(\frac{\varepsilon_{ni}^2}{\sigma_n^2}
\mathbbm{1}_{\{
|\varepsilon_{ni}|/\sigma_n > M
\}}
\biggr) < \frac{1}{4}.
\] 
Let $u_{ni} := (\varepsilon_{ni}^2/\sigma_n^2)\Ind_{\{|\varepsilon_{ni}|/\sigma_n \leq M\}} \leq M^2$ for $i \in [n]$. Since $\E_n(\varepsilon_{ni}^2/\sigma_n^2) = 1$ for all $i$, it follows that
\[
\liminf_{n \rightarrow \infty} \inf_{\nu_n\in\mathcal M_n}\frac{1}{n}\sum_{i=1}^n
\E_n(u_{ni}) > \frac{3}{4}.
\]
Then by Hoeffding's inequality, for all sufficiently large $n$,
\begin{align*}
\sup_{\nu_n \in \mathcal{M}_n} \Pr_n(\mathcal{E}_n^c)
&= \sup_{\nu_n \in \mathcal{M}_n} \Pr_n\biggl(
\frac{1}{n}\sum_{i=1}^n \frac{\varepsilon_{ni}^2}{\sigma_n^2}
 < \frac{1}{2}
\biggr)
\leq \sup_{\nu_n \in \mathcal{M}_n} \Pr_n\biggl(\frac{1}{n} \sum_{i=1}^n u_{ni} < \frac{1}{2}
\biggr) \\
&\leq \sup_{\nu_n \in \mathcal{M}_n} \Pr_n\biggl(\frac{1}{n}\sum_{i=1}^n 
\bigl(u_{ni} - \E(u_{ni})\bigr) < -\frac{1}{4}
\biggr) \leq e^{-n/(8M^4)}.
\end{align*}
Since $Y_{nj}^2/R_n = n(Q_j^\top\varepsilon_n)^2/\|\varepsilon_n\|^2_2 \leq n$ for all $n$, we deduce that 
\begin{align*}
\limsup_{n \to \infty} \sup_{\nu_n \in\mathcal M_n}\E_n\bigl(Y_{nj}^2|R_n^{-1}-1|
\mathbbm{1}_{\mathcal E^c_n}\bigr) &\leq \limsup_{n \to \infty}\,\biggl\{\sup_{\nu_n\in\mathcal M_n}\E_n\biggl(\frac{Y_{nj}^2}{R_n}
\mathbbm{1}_{\mathcal E^c_n}\biggr) + \sup_{\nu_n\in\mathcal M_n}\E_n(Y_{nj}^2
\mathbbm{1}_{\mathcal E^c_n})\biggr\} \\
&\leq \limsup_{n \to \infty}\,\Bigl\{ne^{-n/(8M^4)}
+
\sup_{\nu_n\in\mathcal M_n}\E_n(
Y_{nj}^2\mathbbm{1}_{\mathcal E_n^c})\Bigr\} = 0.
\end{align*}
Summing over $j \in [p]$, we conclude that 
\begin{align*}
\sup_{\nu_n\in\mathcal M_n}\W_1(A_n, B_n) &\leq \sup_{\nu_n\in\mathcal M_n}\mathbb{E}_n\bigl(\|S(Y_n/R^{1/2}_n) - S(Y_n)\|_2\bigr) \\
&\leq \sup_{\nu_n\in\mathcal M_n}\mathbb{E}\bigl(\|S(Y_n/R^{1/2}_n) - S(Y_n)\|_1\bigr) = \sup_{\nu_n\in\mathcal M_n}\sum_{j=1}^p \mathbb{E}(Y_{nj}^2|R_n^{-1}-1|) \rightarrow 0,
\end{align*}
so $\sup_{\nu_n\in\mathcal M_n} \W_1(A_n,C_Z) \rightarrow 0$ as $n \to \infty$.  Now define $Z^*_n \sim N_n(0,\sigma_n^2 I_n)$, $Y^*_n := Q^\top Z^*_n/\sigma_n$, $R_n^* := \|Z^*_n\|_2^2/(n\sigma_n^2)$ and 
\[
A_n^* := S\biggl(\frac{Y^*_n}{(R_n^*)^{1/2}}\biggr).
\]
Then as a special case of the above argument, it follows that $\sup_{\nu_n\in\mathcal M_n}\W_1(A_n^*,C_Z)\rightarrow 0$, so $\sup_{\nu_n\in\mathcal M_n}\W_1(A_n, A^*_n)\rightarrow 0$. Finally, define $g_n \colon \R^p \to \R^3$ by
\[
g_n(x_1,\ldots,x_p)
:=
\biggl(
\sum_{j=1}^{p_0}x_j,\,
\sum_{j=p_0+1}^{p}x_j,\,
n-\sum_{j=1}^{p}x_j
\biggr).
\]
Then for any $x,y \in \R^p$, we have
\[
\|g_n(x)-g_n(y)\|_2^2 \leq \|x_{[p_0]} - y_{[p_0]}\|_1^2 + \|x_{(p_0:p]} - y_{(p_0:p]}\|_1^2 + \|x - y\|_1^2 \leq 2\|x-y\|_1^2 \leq 2p\|x-y\|_2^2.
\]
Moreover, $g_n(A_n^*) \stackrel{d}{=} nT^*$, so
\[
\sup_{\nu_n\in\mathcal M_n}\W_1(nT,nT^*)=\sup_{\nu_n\in\mathcal M_n}\W_1\bigl(g_n(A_n),g_n(A^*_n)\bigr)\leq (2p)^{1/2}\sup_{\nu_n\in\mathcal M_n}\W_1(A_n,A^*_n)\rightarrow 0.
\]
The final conclusion then follows from Theorem~\ref{thm:dK-W1-A}.
\end{proof}
The following is a simplified version of Proposition~\ref{Prop:Asymptotic}.
\begin{corollary}
\label{cor:asymptoticF}
Fix $p_0 \in \N_0,p \in \N$ with $p_0 < p$.  Let $(X_1,\varepsilon_1),(X_2,\varepsilon_2),\dotsc$ be independent and identically distributed pairs taking values in $\R^p \times \R$ such that $\E(\|X_1\|_2^2) < \infty$ and $\E(X_1 X_1^\top)$ is positive definite, while $\varepsilon_1$ is independent of $X_1$ and satisfies $\E(\varepsilon_1)=0$ and $\E(\varepsilon_1^2)=\sigma^2 \in (0,\infty)$. Let $X$ denote the $n \times p$ matrix with rows $X_1,\dotsc,X_n$, and let $\varepsilon := (\varepsilon_1,\dotsc,\varepsilon_n)$. Then under~$H_0$, we have
\[
\W_1(nT,nT^*) \rightarrow 0 \quad \text{and hence} \quad \dK(\mathsf{F},\mathsf{F}^*)\rightarrow 0.
\]
\end{corollary}
\begin{proof}
Under the given assumptions, conditions~\eqref{Eq:UI} and~\eqref{Eq:Lindeberg} follow immediately by the dominated convergence theorem, so the result follows by Proposition~\ref{Prop:Asymptotic}.
\end{proof}

\end{document}